\documentclass[12pt]{amsart}
\usepackage{amsmath,amssymb,latexsym}
\usepackage{fullpage}
\usepackage{amscd}
\theoremstyle{plain}
\newtheorem{theorem}{Theorem}

\newtheorem{proposition}{Proposition}
\newtheorem{lemma}{Lemma}
\newtheorem{definition}{Definition}
\newtheorem{conjecture}{Conjecture}
\newtheorem{corollary}{Corollary}

\def\A{\mathbb{A}}
\def\C{\mathbb{C}}

\def\Ind{\text{Ind}}

\def\diag{\text{diag}}
\def\Hom{\text{Hom}}
\def\Mat{\text{\textnormal{Mat}}}
\numberwithin{equation}{section}

\newcommand\mydots{\makebox[1em][c]{.\hfil.\hfil.}}

\begin{document}

\title[Classical Theta Lifts for Higher Metaplectic Covering Groups]{Classical Theta Lifts for Higher Metaplectic Covering Groups}
\author{Solomon Friedberg}
\author{David Ginzburg}
\date{August 29, 2020}
\address{Friedberg:  Department of Mathematics, Boston College, Chestnut Hill, MA 02467-3806, USA}
\email{solomon.friedberg@bc.edu}
\address{Ginzburg: School of Mathematical Sciences, Tel Aviv University, Ramat Aviv, Tel Aviv 6997801,
Israel}
\email{ginzburg@post.tau.ac.il}
\thanks{This work was supported by the BSF,
grant number 2016000, and by the NSF, grant number DMS-1801497 (Friedberg).}
\begin{abstract}
The classical theta correspondence establishes a relationship between automorphic representations on special orthogonal groups and automorphic representations on 
symplectic groups or their double covers. 
This correspondence is achieved by using as integral kernel a theta series on the metaplectic double cover of a symplectic group that is constructed from the Weil representation.  
There is also an analogous local correspondence.
In this work we present an extension of the classical theta correspondence to higher degree metaplectic covers of symplectic and special orthogonal groups.  
The key issue here is that for higher degree covers there is no analogue of the Weil representation, 
and additional ingredients are needed.  Our work reflects a broader paradigm: constructions in automorphic forms that work for algebraic groups or their double covers should often extend to 
higher degree metaplectic covers.
\end{abstract}

\maketitle

\section{Introduction}\label{introduction}
Theta series provide a way to construct correspondences between spaces of automorphic forms.  For example, suppose that $G_1=SO(V_1)$ and $G_2=Sp(V_2)$ are orthogonal and symplectic
groups, resp.  Then these groups embed in a symplectic group $G:=Sp(V_1\otimes V_2)$.  
Let $F$ be a number field and $\A$ its ring of adeles.
There is a family of theta functions $\theta^\phi$ on the adelic metaplectic
double cover $Mp(V_1\otimes V_2)(\A)$ (depending on 
some additional data $\phi$), and these functions may be used to create a correspondence between automorphic forms $f_1$ on $G_1(\A)$ and automorphic forms $f_2$
on $G_2(\A)$ or its double cover:
$$f_2(g_2,\phi)=\int \theta^\phi(g_1\cdot g_2)\,f_1(g_1)\,dg_1,$$
where the integral is over the adelic quotient $G_1(F)\backslash G_1(\A)$.  
This correspondence, which in a low-rank case can be used to recreate the Shimura correspondence, has been studied by many authors
over the past half century (see for example Rallis \cite{R1} and the references therein).   
There is also a local correspondence of irreducible smooth representations, the Howe correspondence, that is obtained by restricting the Weil representation to the image of $G_1\cdot G_2$ in $G$
(see Howe \cite{Ho}),
that has likewise received a great deal of attention.  The goal of this paper is to extend these constructions to higher metaplectic covers of the groups 
$G_1,G_2$ (beyond the double cover $Mp$) and to initiate an analysis of the resulting map of representations.

The double cover of the symplectic group, $Mp$, was introduced by Weil \cite{We} in his treatment of theta series.  
Over local or global fields $F$ with enough roots of unity, there are higher degree covers of
classical groups as well, related to the work of Bass-Milnor-Serre 
on the congruence subgroup problem.  These were first treated for simply connected algebraic groups over $F$ whose $F$-points are simple, simply connected, split over $F$
and of rank at least two by Matsumoto \cite{Mat},
and in the context of number theory and automorphic forms  for $GL_n$ by Kubota \cite{Ku} (when $n=2$) and 
Kahzdan-Patterson \cite{K-P} (for general $n$).  The construction of covering groups was extended to a wide class of groups
$K$-theoretically by Brylinski and Deligne \cite{B-D}.  Fan Gao \cite{Gao1} gave a thorough treatment of metaplectic covers 
over local and global fields using the Brylinski-Deligne construction, and generalized Langlands's work on the constant term of 
Eisenstein series to Eisenstein series on covering groups.  Though for more general groups there are sometimes
a number of inequivalent covers of a given degree, for the groups we treat here the covers are essentially unique, and locally may be regarded as restrictions of the covers of
Matsumoto.  We will describe the covers in detail below.

In this paper, we construct a global theta lifting taking automorphic representations on covers of a symplectic group to automorphic forms on covers of an orthogonal group.  Here the degrees
of the covers must be compatible.  The construction is via a theta kernel.  The difficulty in this construction is that there is no known analogue of the Weil representation for general degree covers.  
The classical theta functions may also be obtained as residues of Eisenstein series on $Mp(\A)$, as shown by Ikeda \cite{Ik1} 
(this may be regarded as an instance of the Siegel-Weil formula; see Ikeda \cite{Ik2}), and so a first attempt would be
to mimic the construction above:
$$\widetilde{f}_2(g_2)=\int \widetilde{\theta}(g_1\cdot g_2)\,\widetilde{f}_1(g_1)\,dg_1,$$
where now $\widetilde\theta$ is an automorphic function realized as a residue of an Eisenstein series.  This is indeed useful in many cases, including
the lift of Bump and the authors \cite{B-F-G2} for the double cover of an orthogonal group and the recent work of 
Leslie \cite{Leslie}, which 
 constructs CAP representations
of the four-fold cover of a symplectic group using such a theta kernel.  However, 
 it is not sufficient to produce the lifting here.  The reason is that this residue is too large a representation, that is, it is not attached to the minimal nontrivial coadjoint orbit or some
 other orbit of small Gelfand-Kirillov dimension. Thus it is too big to 
give a correspondence by mimicking the standard procedure (i.e., restricting to the tensor product embedding).

Instead we consider the product of two different theta functions---one coming from a residue
of an Eisenstein series, and the other coming from the Weil representation--- and use a Fourier coefficient of this product as integral kernel (given precisely in \eqref{lift2} below).
In the case of the trivial cover of the orthogonal group, one of the theta functions and the Fourier coefficient integral become trivial, and our construction
reduces to the classical theta kernel construction.  Hence the construction presented here may be viewed as an extension of the classical theta integral construction.
There is also a local analogue, and we show that for principal series induced from unramified quasicharacters in general position the lifting in the equal rank case gives a correspondence that
is indeed compatible with the expected map of $L$-groups (these are discussed below). 
For the classical theta correspondence, local functoriality was established (in greater generality) by Rallis \cite{R-functor}, Section 6.

To explain the L-group formalism that is related to our construction, 
let $G$ denote one of the groups $Sp_{2l}$ or $SO_k$. Let $r$ denote a positive integer, and $G^{(r)}({\A})$ the $r$-fold metaplectic covering group of $G({\A})$. 
For this group to be defined we need to assume that the field $F$ contains all $r$-th roots of unity; for convenience we shall assume that $F$ contains the $2r$-th roots of unity. 
The notion of an L-group for metaplectic covering groups was developed in general by Weissman \cite{W}, following a description in the split case by McNamara \cite{McN}. 
The assumption that $F$ contains the $2r$-th roots avoids certain complications in Weissman's construction (and is required by McNamara).
In particular, though the Langlands dual group for a cover of an even orthogonal group goes as usual: $(SO_{2a}^{(r)})^\vee=SO_{2a}(\C)$, the Langlands dual group for 
covers of odd orthogonal and symplectic groups depends on the parity of $r$:
$$(Sp_{2l}^{(r)})^\vee=\begin{cases} SO_{2l+1}(\C)&\text{if $r$ is odd}\\ Sp_{2l}(\C)&\text{if $r$ is even}\end{cases}\qquad
(SO_{2a+1}^{(r)})^\vee=\begin{cases} Sp_{2a}(\C)&\text{if $r$ is odd}\\SO_{2a+1}(\C)&\text{if $r$ is even.}\end{cases}$$  As in the classical case,
one then expects functorial liftings of automorphic representations between the metaplectic covering groups of symplectic and orthogonal groups.  However, as one can see
from these dual groups, the groups involved in such a theta correspondence should sometimes be different than in the classical case.
Specifically, if $r$ is odd, then for suitable $a,l$ one expects a correspondence 
between automorphic representations of $Sp_{2l}^{(r)}({\A})$ and $SO_{2a}^{(r)}({\A})$, and between automorphic representations of $Sp_{2l}^{(2r)}({\A})$ and $SO_{2a+1}^{(r)}({\A})$.
For $r=1$, these maps are the classical theta correspondence.
In this paper we introduce a construction which gives them for any odd $r$. 

Residues of Eisenstein series, when square-integrable, contribute to the residual spectrum, and these residues are frequently not generic.  
Indeed, the residue we consider generates an automorphic representation 
that should have many vanishing classes of Fourier coefficients, and this property is key to establishing our results.  However, the question of determining the maximal unipotent orbit that supports
a non-zero Fourier coefficient, or locally a nonzero twisted Jacquet module, is rather delicate, and has only recently been addressed for the general linear group (Cai \cite{C1}; see also
Leslie \cite{Leslie}). We make a conjecture about this exact orbit (Conjecture~\ref{conj1}), 
establish the vanishing of Fourier coefficients that it would imply, and establish a number of partial results towards the full conjecture,
including proving it in certain cases.
Though we do not resolve it in full, these are sufficient for the results presented here.  See Section~\ref{conjtheta} below.

The construction we present here can most likely be taken farther.  For example, one can consider the inverse correspondence from covers of orthogonal groups to covers
of symplectic groups obtained by the same integral kernel,
and in the local case one may formulate analogues of the Howe Conjecture.  One can also ask for the first non-zero occurrence of the lift
(for the classical theta lift, see Roberts \cite{Roberts}), and this will be treated in a subsequent paper \cite{F-G5}.  In particular, we will show that for given $r$, if Conjecture~\ref{conj1} below
holds, then
the integrals presented here do not vanish on a full theta tower, and also that a generic cuspidal genuine automorphic representation of $Sp_{2n}^{(r)}(\A)$
always lifts nontrivially to a generic genuine representation on $SO_{2n+r+1}^{(r)}(\A)$.  Since Conjecture~\ref{conj1} is established here for $r=3,5$,
this implies these results unconditionally for these two cases.
One can also consider the case of $r$ even, but this will require additional information about the unipotent orbits of theta residues. 

Our work suggests a broader paradigm:  {\it constructions} in the theory of automorphic forms should generalize to covers.  We note three examples.  First, as mentioned above,
Fan Gao \cite{Gao1} has extended Langlands's work on the constant term of Eisenstein series to covers.  Second, the doubling integrals of the authors, Cai and Kaplan \cite{CFGK1} 
may be extended to covers, as loosely sketched
in our announcement \cite{CFGK2} and worked out in detail by Kaplan in \cite{Kap}.  See also Cai \cite{C2}.
Third, the work here indicates that the classical theta correspondence also generalizes.  As noted above,
Weissman \cite{W} has defined a metaplectic $L$-group. Our suggestion is that not only the formalism of functoriality but also the integrals that give $L$-functions or correspondences 
should often generalize.  Of course, doing so must involve new ideas, as is the case here.  As another example, 
it is not straightforward to extend Langlands-Shahidi theory to covers as
this theory uses the uniqueness of the Whittaker model, and such uniqueness does not hold, even locally, for most covering groups. However, it should be possible to
generalize this theory by replacing the Whittaker model for a cuspidal automorphic
representation $\tau$ on a metaplectic covering group with the Whittaker model for the Speh representation attached to $\tau$. Indeed, 
this Speh representation, a residue of an Eisenstein series (and defined only conjecturally at the moment), is expected to have a unique Whittaker 
model by a generalization of Suzuki's conjecture
\cite{Su1}, \cite{Su2}.  (Suzuki's conjecture concerns covers of $GL_n$ and we expect a similar phenomenon in general.)

To conclude this introduction we describe the algebraic and representation theoretic structures that are behind our construction.  
A classical reductive dual pair is a pair of subgroups $(G_1,G_2)$
inside a symplectic group $Sp(W)$ which are mutual centralizers and which act reductively on $W$.  Our algebraic structure is this:
\begin{enumerate}
\item Groups $G_1$, $G_2$, and two symplectic vector spaces $W_1$, $W_2$,  with monomorphisms
$$\iota_1:G_1\times G_2\to Sp(W_1), \quad \iota_2:G_1\times G_2\to Sp(W_2),$$ 
such that 
\begin{enumerate} \item via the map $\iota_1$, $(G_1,G_2)$ is a reductive dual pair in $Sp(W_1)$
\item
the images under $\iota_2$ of $G_1\times 1$ and $1\times G_2$ in $Sp(W_2)$ commute, though they need not be mutual centralizers;
\end{enumerate}
\item A unipotent group $U\subset Sp(W_2)$ which is normalized by $\iota_2(G_1,G_2)$; 
\item A character $\psi_U:U\to \mathbb{C}$ such that $\iota_2(G_1,G_2)$, acting by conjugation, stabilizes $\psi_U$, and the index of $\iota_2(G_1,G_2)$ in the full stabilizer of $\psi_U$ is finite;
\item A homomorphism $l:U\to {\mathcal{H}}(W_1)$, where $\mathcal{H}(W_1)$ denotes the Heisenberg group attached to $W_1$.
\end{enumerate}
In this paper, we take $(G_1,G_2)=(SO_k,Sp_{2n})$. The map $\iota_1$ is the tensor product embedding into $Sp_{2nk}$, just as in the
classical theta correspondence, while remaining ingredients are given in Sections~\ref{basic} and \ref{basic2} below. We actually consider covers, and extend the maps $\iota_1,\iota_2$ to covers there.

To explain the representation theory that is employed, let $\omega_\psi$ denote the Weil representation of $\mathcal{H}(W_1)\rtimes {Mp}(W_1)$ 
with additive character $\psi$, over either $\A$ or a completion of $F$.
As noted above, the global theta correspondence (resp.\ locally, the Howe Correspondence) is 
based on restricting theta series constructed from $\omega_\psi$ (resp.\ the representation $\omega_\psi$) to 
the image of a reductive dual pair
$\iota_1(G_1,G_2)$.  In our structure, we have two symplectic groups, and we restrict the tensor product of two small representations.  
Specifically, we form the representation $\omega_\psi\otimes \Theta_2^{(r)}$, where $\Theta_2^{(r)}$ is a theta representation
 of $Sp^{(r)}(W_2)$.  We then restrict this to ${Mp}(W_1)\times Sp^{(r)}(W_2)$ where the groups act on the two representations via $\iota_1$ and $\iota_2$, resp.
 Globally, in making the integral kernel we also take a Fourier coefficient with respect to $U $ and $\psi_U$, where the action of $U$ on the first factor is via $\ell$; 
 see \eqref{lift2} below.
The local map, conjecturally a correspondence, makes use of a twisted Jacquet module with respect to $U$, where
 in defining this Jacquet module, the action of $U$ on the first factor is again via $\ell$. That is, if we write $(V_1,\omega_\psi)$ and $(V_2,\Theta_2^{(r)})$ as the representations over a local field,
then $J_{U,\psi_U}(\omega_\psi\otimes \Theta_2^{(r)})=(V_1\otimes V_2)/V_3$ with
$$
V_3=<\omega_\psi(\ell(u)) v_1\otimes \Theta^{(r)}_2(u)v_2- v_1\otimes \psi_U(u)v_2\mid v_1\in V_1, v_2\in V_2, u\in U>.
$$
It is this twisted Jacquet module that is restricted to $G_1\times G_2$ (with suitable covers, and once again using the embeddings $\iota_1$, $\iota_2$).

Section~\ref{basic} introduces the notation to be used in the sequel and treats the necessary foundations concerning metaplectic groups in some detail.  In particular, we shall have
occasion to use isomorphic but different covers of groups (these arise by modifying a 2-cocycle by a coboundary), and these are described precisely. Then Section~\ref{basic2}
presents the integral \eqref{lift2} that describes the theta lifting.  As indicated above, this integral requires a theta function on a higher odd degree
cover of a symplectic group $Sp_{2l}$ and also a second theta function on the double cover of a symplectic group, which is constructed using the Weil representation.  
The associated higher theta representation $\Theta_{2l}^{(r)}$ is analyzed in Section~{\ref{conjtheta}}.  In Conjecture~\ref{conj1} we describe the expected unipotent orbit
behavior of this representation, and in Theorem~\ref{theta05} we give a proof of the Conjecture in certain cases and a partial proof in others.  
Section~\ref{cusp1} concerns cuspidality. In Theorem~\ref{th2}, we show that an analogue
of the Rallis theta tower is valid for this new class of theta lifts.  The proof uses in an essential way both the properties of the higher theta representation
$\Theta_{2l}^{(r)}$ and the description of the classical theta function on the double cover.  Section~\ref{the-unramified-corr}, the final section of this work, concerns the unramified lift.
In Theorem~\ref{unramified}, this is shown to be functorial for characters in general position in the equal rank case.

\section{Groups and Covering Groups}\label{basic}

Let $F$ be a number field with ring of adeles $\A$. If $H$ is any algebraic group defined over $F$, we write $[H]$ for the quotient $H(F)\backslash H(\A)$.
In particular, $[\mathbb{G}_a]=F\backslash \A$.

We will work with the symplectic and special orthogonal groups defined over $F$.  We realize them as follows. For $m\geq1$ let 
 $J_m$ be the $m\times m$ matrix with 1 on the anti-diagonal and zero elsewhere.
 Let $Sp_{2m}$ denote the symplectic group
$$Sp_{2m}=\left\{g\in GL_{2m}~\mid~ ^t\!g \begin{pmatrix}0&J_m\\-J_m&0\end{pmatrix} g=\begin{pmatrix}0&J_m\\-J_m&0\end{pmatrix}\right\}$$
and for $k\geq2$ let $SO_k$ denote the split special orthogonal group
$$SO_k=\left\{ g\in GL_k~ \mid ~^t\!g\, J_k\,g=J_k\right\}.$$
(The construction presented here can be extended to the non-split case but we shall not do so here.)  

Throughout the paper, we make use of two embeddings. 
First, let $n\geq1$, $k\geq2$, and let $\iota_1:SO_k\times Sp_{2n}\to Sp_{2nk}$ denote the usual tensor product embedding.  In matrices, if  $h\in SO_k(\A)$ and $g\in Sp_{2n}(\A)$,
then $\iota_1((1,g))=\text{diag}(g,\dots,g)$ where $g$ appears $k$ times,
and $\iota_1((h,1))=(h_{ij}I_{2n})$.  Second, let $r\geq1$ be odd and 
let $\iota_2:SO_k\times Sp_{2n}\to Sp_{2n+k(r-1)}$ be the embedding given by $\iota_2(h,g)=\text{diag}(h,\ldots,h,g,h^*,\ldots,h^*)$, 
where $h$, $h^*$ each appear $(r-1)/2$ times and $h^*$ is determined so that the matrix is symplectic. 

Let $\Mat_{a\times b}$ denote the algebraic group of all matrices of size $a\times b$, and when $a=b$, write $\Mat_a=\Mat_{a\times a}$.
For $l\geq1$, let ${\mathcal H}_{2l+1}$ denote the Heisenberg group in $2l+1$ variables.  This is the group with elements $(X,Y,z)$ where $X,Y\in \Mat_{1\times l}$ and $z\in \Mat_1$,
and multiplication given by 
$$(X_1,Y_1,z_1)(X_2,Y_2,z_2)=(X_1+X_2,Y_1+Y_2,z_1+z_2+{\tfrac12}(X_1J_l\,^tY_2-Y_1J_l\,^tX_2)).$$
The Heisenberg group ${\mathcal H}_{2l+1}$ is isomorphic to a subgroup of $Sp_{2l+2}$ by the map
$$\tau(X,Y,z)=\begin{pmatrix}1&X&{\tfrac12}Y&z\\&I_l&&Y^*\\&&I_l&X^*\\&&&1\end{pmatrix},$$
where the starred entries are uniquely determined by the requirement that the matrix be symplectic ($X^*=-J_l\,^tX$, $Y^*={\tfrac12}J_l\,^tY$).  
Also, let
$\Mat_a^0=\{ Z\in \Mat_{a} \mid {}^tZJ_{a}=J_{a}Z\}$, and let $\Mat_a^{00}=\{ Z\in \Mat_{a}\mid{} ^tZJ_{a}=-J_{a}Z\}$.

Next we define some parabolic and unipotent groups that will be of use. All parabolic subgroups here are standard parabolic subgroups whose 
unipotent radical consists of upper triangular unipotent matrices, and all of the unipotent groups we introduce are groups of upper triangular unipotent matrices.

For non-negative integers $a,b$ and $c$,
let $P_{a,b,c}$ denote the parabolic subgroup of $Sp_{2(ab+c)}$ whose Levi part is $GL_{a}\times\ldots\times GL_{a}\times Sp_{2c}$ 
where $GL_{a}$ appears $b$ times.  Let $U_{a,b,c}$ denote the unipotent radical of the group $P_{a,b,c}$. 
Let $1\le i\le b-1$,  $\alpha=a(i-1)$,  $\beta=2(ab+c)-2a(i+1)$, and set
\begin{equation}\label{mat4}
u_{a,b,c}^i(X)=\begin{pmatrix} I_\alpha&&&&&&\\ &I_a&X&&&&\\ &&I_a&&&&\\ &&&I_\beta&&&\\
&&&&I_a&X^*&&\\ &&&&&I_a&\\ &&&&&&I_\alpha\end{pmatrix},\qquad
X\in \Mat_{a}.
\end{equation}
The set of all matrices $\{u_{a,b,c}^i(X)\mid X\in \Mat_{a}\}$  is a subgroup of $U_{a,b,c}$ which we denote by $U_{a,b,c}^i$. 
Given $u\in U_{a,b,c}$, one may write it uniquely as
$u=u_{a,b,c}^i(X_i)u'$ where $X_i\in\Mat_{a}$ 
and $u'\in U_{a,b,c}$ is such that all its $(p,j)$ entries are zero when $\alpha+1\le p\le \alpha+a$ and  $\alpha+a+1\le j\le \alpha+2a$. We call
$u_{a,b,c}^i(X_i)$ the $i$-th coordinate of $u$. 

Another subgroup of $U_{a,b,c}$ is the group 
$U_{a,c}'$ that consists of all matrices 
\begin{equation}\label{mat1}
u_{a,c}'(Y,Z)=\begin{pmatrix} I_{a(b-1)}&&&&\\ &I_{a}&Y&Z&\\ &&I_{2c}&Y^*&\\ &&&I_{a}&\\
&&&&I_{a(b-1)}\end{pmatrix},\qquad Y\in \Mat_{a\times 2c}, Z\in \Mat_a^0.
\end{equation}
Similarly, we define the $u_{a,c}'(Y,Z)$ coordinate of $u$.
Then every $u\in U_{a,b,c}$ has a  factorization  
\begin{equation}\label{mat2}
u=u_{a,c}'(Y,Z)\prod_{i=1}^{b-1} u_{a,b,c}^i(X_i)u_1,
\end{equation}
where $u_1\in U_{a,b,c}$ has entries zero in the first $b$ $a\times a$ blocks directly above the main diagonal and in positions $Y$ and $Z$ in \eqref{mat1} above
(and so also zero in the corresponding last $b$ superdiagonal blocks and in $Y^*$).
Let $U_{a,b,c}^1$ denote the subgroup of $U_{a,b,c}$ consisting of all matrices \eqref{mat2} such that $Y=0$.
Also, let $U_{a,b,c}^0$ denote the subgroup of $U_{a,b,c}$ consisting of all matrices \eqref{mat2} such that $Y=Z=0$. Note that $U_{a,b,c}^0$ is a subgroup of 
$U_{a,b,c}^1$.

The group $U_{a,c}'$  has a structure of a generalized Heisenberg group.
Define a homomorphism 
$l:U_{a,b,c}\to {\mathcal H}_{2ac+1}$ as follows.   Let $y_i\in \Mat_{1\times 2c}$ denote the $i$-th row of $Y$,
and define 
\begin{equation}\label{Heisenberg-hom}
l(u_{a,c}'(Y,Z))=(y_a,y_{a-1},\dots,y_1,\tfrac12\text{tr}(T_kZ)),
\end{equation} 
where $T_k$ is the identity matrix $I_k$ if $k$ is even, and $T_k=\diag(I_{[k/2]},2,I_{[k/2]})$ if $k$ is odd.
(Note that the center of $U_{a,c}$ consists of all matrices \eqref{mat1} such that $Y=0$, and this is mapped to the center of ${\mathcal H}_{2ac+1}$.) 
Then extend $l$ trivially from $U_{a,c}'$ to $U_{a,b,c}$.

Also, when we work with Weyl groups, we shall always work with representatives which have one non-zero entry in 
each row and column whose value is $\pm1$; we call such matrices Weyl group elements.  For the symplectic group $Sp_{2a}$, such a matrix is determined 
uniquely by its entries in the first $a$ rows and we sometimes specify it in this way (e.g.\ the proof of Proposition~\ref{theta03}).

We now set the notation for metaplectic groups.
Fix an integer $m\geq1$ such that $F$ contains a full set of $m$-th roots of unity $\mu_m$.  (Below, we will take $m$ to be an odd integer $r$ or to be $2r$.)
If $G$ is either the $F_\nu$ points (where $\nu$ is a place of $F$) or the $\A$
points of a linear algebraic group (that will be $SO$ or $Sp$), 
then an 
$m$-fold covering group $\widetilde{G}$  of $G$ is a topological central extension of $G$ by $\mu_m$.
The group consists of pairs $(g,\epsilon)$, $g\in G$, $\epsilon\in\mu_m$, with multiplication described in terms a (Borel measurable) 2-cocycle $\sigma\in Z^2(G,\mu_m)$:
$$(g_1,\epsilon_1)(g_2,\epsilon_2)=(g_1g_2,\epsilon_1\epsilon_2\sigma(g_1,g_2)).$$
After fixing an embedding of $\mu_m$ into $\C^\times$, such groups naturally give rise to central extensions of $G$ by $\C^\times$ by the same multiplication rule.
If $\beta:G\to\C^\times$ is a (Borel measurable) map with $\beta(e)=1$, then multiplying $\sigma$ by the 2-coboundary 
$(g_1,g_2)\mapsto \beta(g_1)\beta(g_2)/\beta(g_1g_2)$
gives an isomorphic group, and we also consider such groups.  
One can construct a global 2-cocycle from local cocycles at each place by taking the product, but only if almost all factors are 1 on the adelic points of $G$.  

We begin with the local constructions that we will use.  Fix $r>1$ odd, let $K$ be a local field containing a full set of $r$-th roots of unity and let $(~,~)_r\in\mu_r$ be the local Hilbert symbol.
If $Sp_{2m}(K)$ is any symplectic group, let $Sp_{2m}^{(r)}(K)$ (or simply $Sp_{2m}^{(r)}$) denote the $r$-fold covering of $Sp_{2m}(K)$
constructed by Matsumoto \cite{Mat}
using the Steinberg symbol corresponding to $(~,~)_r^{-1}$.  Since $Sp_{2m}$ is simple and simply connected, this is the unique topological $r$-fold covering group of $Sp_{2m}(K)$; 
see for example Moore \cite{Mo}.  Hence, the cocycle is unique, but only up to a coboundary.
For any integer $m$, let $\sigma_m\in Z^2(SL_{m}(K),\mu_r)$ denote the 2-cocycle on $SL_m(K)$
constructed by Banks, Levy and Sepanski \cite{B-L-S}, Section 2 (when we need to indicate the cover, we write this $\sigma_m^{(r)}$). 
This cocycle enjoys a block compatibility property (\cite{B-L-S}, Theorem 1) that will be of use.
We shall realize $Sp_{2m}^{(r)}$ by restricting the cocycle $\sigma_{2m}\in Z^2(SL_{2m}(K),\mu_r)$ to $Sp_{2m}(K)\times Sp_{2m}(K)$, as in Kaplan \cite{Kap}, Section 1.4.
Doing so, the block compatibility implies that 
if $g_1,g_2\in SO_k(K)$, $h_1,h_2\in Sp_{2n}(K)$, then
\begin{equation}\label{block-iota2}\sigma_{2n+k(r-1)}(\iota_2(h_1,g_1),\iota_2(h_2,g_2))=\sigma_{k}(h_1,h_2)^{r-1}\sigma_{2n}(g_1,g_2).
\end{equation}
Since $\sigma_k$ take values in $\mu_r$, we have $\sigma_{k}(h_1,h_2)^{r-1}=\sigma_{k}(h_1,h_2)^{-1}$.
We realize the $r$-fold cover $SO_k^{(r)}(K)$ of the orthogonal group $SO_k(K)$ by restricting the cocycle $\sigma_k^{-1}$ to $SO_k(K)\times SO_k(K)$.
Then \eqref{block-iota2} implies that the map $\iota_2$ extends to a homomorphism of the covering groups $\iota^{(r)}_2:SO^{(r)}_k(K)\times Sp_{2n}^{(r)}(K)\to Sp^{(r)}_{2n+k(r-1)}(K)$,
given by
\begin{equation}\label{iota2}
\iota_2^{(r)}\left(((h,\epsilon_1),(g,\epsilon_2))\right)=(\iota_2(h,g),\epsilon_1\epsilon_2).
\end{equation}

We will also make use of the double cover $Sp^{(2)}_{2nk}(K)$ (this cover is unique up to isomorphism), and its pullback via
the tensor product embedding $\iota_1$.  Recall that when $k$ is even (resp.\ $k$ is odd), the pair $(SO_k,Sp_{2n})$ (resp.\  $(SO_k,Sp_{2n}^{(2)})$) forms a reductive dual pair inside
$Sp_{2nk}^{(2)}$.  We recall the local theory and give a reference
for the adelization; this construction underlies the classical theta correspondence.  

For each $m\geq1$, we realize the group $Sp^{(2)}_{2m}(K)$
by a Leray cocycle $\widetilde{\sigma}_{2m}$, which takes values in (generally) the eighth roots of unity.  (We will be specific in the next paragraph.) 
This cocycle may be regarded as arising from the classical Weil representation attached to a fixed nontrivial additive character and realized in the Schr\"odinger model; 
see Kudla \cite{Ku2}, Theorem 3.1.
The cocycle is cohomologous
to a ``Rao cocycle" with values in $\mu_2=\{\pm1\}$ (see  Rao \cite{Rao}). 

Information about the behavior of the metaplectic double cover with respect to the tensor product embedding may be found in Kudla \cite{Ku1}, \cite{Ku2}. 
We use a cocycle given in those works, adjusting the notation to take into account some different normalizations.
Kudla considers vector spaces $V,W$ over the local field $K$ of dimensions $k $ and $2n$, resp., 
equipped with nondegenerate bilinear (resp.\ symplectic) forms, and defines $\mathbb{W}=V\otimes_K W$.
The tensor product of the forms gives a symplectic form on $\mathbb{W}$.
He then defines $j_W$, $j_V$ by using
the tensor product embedding $SO(V)\times Sp(W)\to Sp(\mathbb{W})$ and restricting to the first and second factors, resp. 
Let $\lambda$ be the natural isomorphism from $Sp_{2nk}(K)$ (defined above) to the group $Sp(\mathbb{W})$ in Kudla's work (i.e.\ moving one symplectic form to another),
such that  $\lambda(\iota_1(h,1))=j_W(h)$ and $\lambda(\iota_1(1,g))=j_V(g)$ for all $h\in SO_k(K)$, $g\in Sp_{2n}(K)$. 
Let $c_{\mathbb{Y}}$ be the Leray cocycle on $Sp_{2nk}(K)^2$ in Kudla's notation.
Then we realize $Sp^{(2)}_{2nk}(K)$ by means of the cocycle
$\widetilde{\sigma}_{2nk}(g_1,g_2)=c_{\mathbb{Y}}(\lambda(g_1),\lambda(g_2))$ (for $g_1,g_2\in Sp_{2nk}(K)$).  For $k=1$ this defines $Sp^{(2)}_{2n}(K)$, with 
$Y$ the maximal isotropic subspace of $W$ given on p.\ 18 of \cite{Ku2}, which corresponds to $2n$-vectors with first $n$ entries zero in our normalization.
Then it is shown in \cite{Ku2}, Ch.\ II, Proposition 3.2, that there is a map $\beta_k:Sp_{2n}(K)\to S^1\subset\C^\times$ such that the map
${\iota}^{(2)}_1$given by
\begin{align*}
\iota^{(2)}_1:SO_k(K)\times Sp_{2n}(K)\to {Sp}_{2nk}^{(2)}(K)\qquad &\iota^{(2)}_1(h,g)=(\iota_1(h,g),\beta_k(g)^{-1})\quad&&\text{if $k$ is even}\cr
\iota^{(2)}_1:SO_k(K)\times Sp_{2n}^{(2)}(K)\to {Sp}_{2nk}^{(2)}(K)\qquad&\iota^{(2)}_1(h,(g,\epsilon))=(\iota_1(h,g),\beta_k(g)^{-1}\epsilon)\quad
&&\text{if $k$ is odd}
\end{align*}
is an isomorphism.

We now turn to the global construction.  Let $m\geq1$ be fixed and write $M=Sp_{2m}$.  We will construct the $r$-fold cover of $M(\A)$.  (We will then choose $m=2n+k(r-1)$.)
Let $F$ be a number field, and for each place $\nu$ let $\sigma_\nu$ be the cocycle $\widetilde{\sigma}_{2m}^{(r)}$ described above for $M(F_\nu)$.
Then at almost all places, the local 
cocycle $\sigma_\nu$ may be adjusted by a coboundary $\eta_\nu$ so that the new cocycle $\rho_\nu$ satisfies $\rho_\nu(\kappa,\kappa')=1$ if $\kappa, \kappa'$ are in the hyperspecial 
maximal compact group $M(\mathcal O_\nu)$, where $\mathcal O_\nu$ is the ring of integers of $F_\nu$.  (This may not be true, and there is no adjustment, at a finite number of places,
namely the ramified places of $F$, the places dividing $2r$ and the archimedean places.)
 
More precisely, for unramified places, let $\eta_\nu:M(F_\nu)\to S^1$ be a continuous map that satisfies 
\begin{equation}\label{split-the-cocycle}
\sigma_\nu(m,m')=\eta_\nu(mm')\eta_\nu(m)^{-1}\eta_\nu(m')^{-1}\qquad\text{for all $m,m'\in M(O_\nu$)}.
\end{equation}
This uniquely determines $\eta_\nu$ on $M(O_\nu)$, and it is then extended to $M(F_\nu)$. See \cite{Kap}, Section 1.5.
We will choose $\eta_\nu$ so that for all $h\in SO_k(F_\nu)$, $g\in Sp_{2n}(F_\nu)$,
\begin{equation}\label{block-iota2-eta}
\eta_\nu(\iota_2(h,1))\,\eta_\nu(\iota_2(1,g))=\eta_\nu(\iota_2(h,g)).
\end{equation}
Indeed, if $h\in SO_k(O_\nu)$, $g\in Sp_{2n}(O_\nu)$, then \eqref{block-iota2-eta} holds, since $\sigma_\nu(\iota_2(h,1),\iota_2(1,g))=1$ by \eqref{block-iota2},
 and then \eqref{block-iota2-eta}
follows from \eqref{split-the-cocycle}. We may then choose the extension of $\eta_\nu$ to $M(F_\nu)$ so that \eqref{block-iota2-eta} holds in general.
Then introduce the cocycle
$$\rho_\nu(g,g')={\frac{\eta_\nu(g)\eta_n(g')}{\eta_\nu(gg')}}\sigma_\nu(g,g')\qquad g,g'\in M(F_\nu),$$
and the 2-cocycle $\rho=\prod_v\rho_v$ of $M(\A)$.  This is well-defined since if $g,g'\in M(\A)$, then at almost all places $g,g'\in M(\mathcal O_\nu)$ and $\rho_\nu(g,g')=1$.
We shall realize the global $r$-fold metaplectic cover ${M}^{(r)}(\A)$ of $M(\A)$ as the central extension of $M(\A)$ by $\mu_r$ with respect to this cocycle.  
For $g\in M(\A)$ such that $\eta(g_v)=1$ for almost all $v$, let $\eta(g)=\prod_v \eta_v(g_v)$; in particular, this function is defined for $g\in M(F)$ (see Takeda \cite{Tak}, Proposition 1.8 and 
Kaplan \cite{Kap}, Section 1.5). 
Then (as a consequence of Hilbert reciprocity) $M(F)$ embeds in ${M}^{(r)}(\A)$ by the homomorphism $m\mapsto (m,\eta^{-1}(m))$.  Since it will always be clear from context
whether we are working in a matrix group or a covering group, we shall not introduce a separate notation for the image of $M(F)$; instead, we 
regard $M(F)$ as a subgroup of ${M}^{(r)}(\A)$ via this map.
We may then form the automorphic quotient $M(F)\backslash {M}^{(r)}(\A)$ and consider genuine automorphic representations on this quotient.

For our construction below, we will proceed as follows.  Given $n,k$, let $m=2n+k(r-1)$ and construct the 2-cocycle $\rho$ on $M(\A)$.  We then define 2-cocycles on $SO_k(\A)$ and
$Sp_{2n}(\A)$ by pulling back $\rho$ via the maps $\iota_2((\star,1))$ and $\iota_2((1,\star))$, resp.
In view of \eqref{block-iota2}, this gives $r$-fold covers of these groups, which we write $SO^{(r)}_k(\A)$, $Sp^{(r)}_{2n}(\A)$, whose
local components at $\nu$ are isomorphic to the local metaplectic $r$-fold covers constructed above (with $K=F_\nu$).
Also, using \eqref{block-iota2-eta}
it is not difficult to check that the map
$\iota_2^{(r)}:SO_k^{(r)}(\A)\times Sp_{2n}^{(r)}(\A)\to Sp_{2n+k(r-1)}^{(r)}(\A)$ given by \eqref{iota2} is a homomorphism.

Similarly one can define a global two-fold cover $Sp_{2nk}^{(2)}(\A)$ whose multiplication is given by a cocycle ${\sigma}_{2nk}^{(2)}$ which is the product
of the cocycle $\widetilde{\sigma}_{2nk,\nu}$ above adjusted by a coboundary at each place $\nu$.  This cocycle restricts to the trivial cocycle on $\iota_1(SO_k(\A),1)^2$ 
and to a cocycle which is either cohomologically trivial if $k$ is even or nontrivial if $k$ is odd on $\iota_1(1,Sp_{2n}(\A))^2$. See for example Sweet \cite{Sw}, Sections 1.8 and 2.4.
Thus there are homomorphisms
\begin{align*}
\iota^{(2)}_1:SO_k(\A)\times Sp_{2n}(\A)\to {Sp}_{2nk}^{(2)}(\A)\qquad &\iota^{(2)}_1(h,g)=(\iota_1(h,g),\delta_k(g))\quad&&\text{if $k$ is even}\cr
\iota^{(2)}_1:SO_k(\A)\times Sp_{2n}^{(2)}(\A)\to {Sp}_{2nk}^{(2)}(\A)\qquad&\iota^{(2)}_1(h,(g,\epsilon))=(\iota_1(h,g),\delta_k(g)\epsilon)\quad
&&\text{if $k$ is odd},
\end{align*}
where $\delta_k:Sp_{2n}(\A)\to S^1$ is a Borel measurable map.  
We also introduce the $r$-fold cover $\widetilde{Sp}_{2n}^{(r)}(\A)$, that is obtained by multiplying
the cocycle $\rho$ by the coboundary $\delta_k(g_1,g_2):= \delta_k(g_1g_2)\delta_k^{-1}(g_1)\delta_k^{-1}(g_2)$. The map $i(g,\epsilon)=(g,\delta_k(g)\epsilon)$ is an isomorphism from 
${Sp}_{2n}^{(r)}(\A)$ to $\widetilde{Sp}_{2n}^{(r)}(\A)$.  

We realize the $2r$-fold cover $Sp_{2n}^{(2r)}(\A)$ by using the 2-cocycle which is the product of the 2-cocycles $\rho$ and $\sigma_{2n}^{(2)}$.
For any $a$ there is a canonical projection $p^{(1)}:Sp_{2n}^{(a)}(\A)\to Sp_{2n}(\A)$ given by projection onto the first factor.  Then the $2r$-fold cover $Sp_{2n}^{(2r)}(\A)$ is isomorphic to the fibre
product of $Sp_{2n}^{(r)}(\A)$ and $Sp_{2n}^{(2)}(\A)$ over $Sp_{2n}(\A)$ with respect to these projections.  Fixing such an isomorphism
(equivalently, an isomorphism $\mu_{2r}\cong \mu_r\times \mu_2$), the group $Sp_{2n}^{(2r)}(\A)$ thus comes equipped with projections 
$p^{(r)}:Sp_{2n}^{(2r)}(\A)\to Sp_{2n}^{(r)}(\A)$, $p^{(2)}:Sp_{2n}^{(2r)}(\A)\to Sp_{2n}^{(2)}(\A)$ that are homomorphisms.  We also use $p^{(1)}$ for the projection from $SO_k^{(r)}(\A)\to SO_k(\A)$.
Last, for fixed $k$ odd we introduce the group $\widetilde{Sp}_{2n}^{(2r)}(\A)$ by using the 2-cocycle $\rho\sigma_{2n}^{(2)}\delta_k$. 
Once again the map $i(g,\epsilon)=(g,\delta_k(g)\epsilon)$ is an isomorphism from 
${Sp}_{2n}^{(2r)}(\A)$ to $\widetilde{Sp}_{2n}^{(2r)}(\A)$.  

As noted above, the group $Sp_{2n}(F)$ embeds in ${Sp}_{2n}^{(\kappa r)}(\A)$ for $\kappa=1,2$.  
Though ${Sp}_{2n}^{(\kappa r)}$  is not an algebraic group, we abuse the notation slightly and write [$Sp_{2n}^{(\kappa r)}]$ for the automorphic quotient
$Sp_{2n}(F)\backslash Sp_{2n}^{(\kappa r)}(\A)$.  Also, $\delta_k(Sp_{2n}(F))=1$ (see Sweet \cite{Sw}, Proposition 2.4.2), so $i$ induces a bijection of the 
automorphic quotients of $Sp_{2n}^{(r)}(\A)$ and $\widetilde{Sp}_{2n}^{(r)}(\A)$.

We conclude this section by mentioning several additional groups of matrices that embed in the covering groups.
First, for any local field $F_\nu$ and the cover $SL^{(r)}_d(F_\nu)$ described above, any upper unipotent subgroup $N(F_\nu)$ of $SL_d(F_\nu)$ is canonically split by the trivial section 
$u\mapsto (u,1)$.  
Passing to the adelic group requires changing each local cocycle by a coboundary, so the group $N(\A)$ splits in $SL^{(r)}_d(\A)$
by means of a section of $N(\A)$, $n\mapsto (n,\eta(n)^{-1})$.  Moreover, this section is canonical (M\oe glin and Waldspurger
\cite{M-W}, Appendix 1).  Any lower unipotent group is also split by a canonical section
$n_{-}\mapsto (n_{-},\eta(n_{-}))$ and these splittings are compatible with the action of the Weyl group (which embeds in the covering group) by conjugation.
That is, if $n_{-}$ is lower triangular and $^w n_{-}:=wn_{-}w^{-1}$ is in $N(\A)$, then $^w(n_{-},\eta(n_{-}))=(^w n_{-},\eta(^w n_{-}))$.  (See for example, \cite{Kap}, equation (1.11).)
From now on, we consider all unipotent groups as embedded in the relevant covering groups by such
sections (and once again do not introduce a separate notation).   With this convention, all notions related to unipotent orbits extend to covering groups without change. 

Similarly, working over $F_{\nu}$ or $\A$, for $a\neq0$, let $t(a)=\text{diag}(a,a^{-1},\ldots,a,a^{-1})\in Sp_{2l}$ ($l\geq1$). If $a_1,a_2\in F_\nu^*$ are $r$-th powers, then
$\sigma^{(r)}_{2l,\nu}(t(a_{1,\nu}),t(a_{2,\nu}))=1$ (this follows from the well-known formula for the cocycle 
$\sigma^{(r)}_{2l,\nu}$ on diagonal matrices; see for example \cite{B-L-S}, Section 3).  
This equality implies that the map $\eta_\nu$ restricted to the group $\{t(a)\mid a\in O_{\nu}^{*,r}\}$ is a homomorphism.
Using this observation and the explicit description of $\eta_\nu$ on $SL_2(O_\nu)$ (due to Kubota), it follows that $\eta_\nu(t(a_{\nu}))=1$ for each place $\nu$ such that $a_\nu\in O_\nu^{*,r}$.  
 Accordingly, the map $t(a)\mapsto \widetilde{t}(a):=(t(a),\eta^{-1}(t(a)))$ is a well-defined embedding of $\{t(a)\mid \A^{*,r}\}$ into $Sp_{2l}^{(r)}(\A)$.
 This embedding will appear in the proof of 
Proposition~\ref{theta04} below.

\section{Description of the Integral Kernel}\label{basic2}

As mentioned in the introduction, the integral kernel we work with here requires two different  theta functions.  The first is a function on the
 metaplectic double cover $Sp_{2l}^{(2)}(\A)$ of $Sp_{2l}(\A)$.
Let $\psi$ denote a nontrivial character of $F\backslash {\A}$.  Let
$\Theta_{2l}^{(2)}$ denote the theta representation defined on the group $Sp_{2l}^{(2)}({\A})$ formed using $\psi$. 
We refer to \cite{G-R-S1}, Section 1, part 6, for the definitions and the action of the Weil representation.  This representation extends to the group ${\mathcal H}_{2l+1}({\A})\cdot Sp_{2l}^{(2)}({\A})$.
The theta kernel will involve a function $\theta_{2l}^{(2),\psi}$ in $\Theta_{2l}^{(2)}$.

To give the second, let $r>1$ be a fixed odd integer, and suppose as above that the number field $F$ contains a full set of $r$-th roots of unity $\mu_r$.
Fix an embedding $\epsilon:\mu_r\to\C^\times$.
Let $\Theta_{2l}^{(r)}$ be the theta representation on the group $Sp_{2l}^{(r)}({\A})$ that is genuine with respect to $\epsilon$; that is, the functions $\theta^{(r)}$ in $\Theta_{2l}^{(r)}$ transform
by central character $\epsilon$, i.e.,
$\theta^{(r)}((I_{2l},\mu)g)=\epsilon(\mu)\,\theta^{(r)}(g)$.
This representation is defined via residues of the minimal parabolic Eisenstein series on $Sp_{2l}^{(r)}({\A})$,
similarly to the construction for $GL_r$ in \cite{K-P}. 
For its definition and basic properties, see \cite[Section 2]{F-G2}; see also Gao \cite{Gao2}.  
We will discuss the Fourier coefficients attached
to different unipotent orbits for this representation in Section~\ref{conjtheta} below.

To describe the global construction, we work with the groups $U_{a,b,c}$ defined just before \eqref{mat4} above.
Define a character $\psi_{U_{a,b,c}}$ of the group 
$[U_{a,b,c}]$ as follows. Given $u\in U_{a,b,c}$, write 
$u=u_{a,c}'(Y,Z)\prod_i u_{a,b,c}^i(X_i)u_1$ as in \eqref{mat2}. Define  
\begin{equation}\label{character}
\psi_{U_{a,b,c}}(u)=\psi(\text{tr}(X_1+\cdots +X_{b-1})). 
\end{equation}
By restriction, this character is also a character of $[U_{a,b,c}^0]$.
We will be concerned with a specific choice of the numbers $a,b$ and $c$. 

Fix two integers $k$ and $n$, which will index the sizes of the orthogonal and symplectic groups, resp., as in Section~\ref{basic}. Let 
$r_1=(r-1)/2$, and set $a=k,\ b=r_1$ and $c=n$. In our main construction, we will take a Fourier coefficient corresponding to the unipotent orbit $((r-1)^k1^{2n})$ of the group $Sp_{2n+k(r-1)}$  
(see, for example, \cite{G}). This is an integral over $[U_{k,r_1,n}]$ against the character $\psi_{U_{k,r_1,n}}$.
It follows from \cite{C-M} that the image of $SO_k\times Sp_{2n}$ under the embedding $\iota_2$ is in the stabilizer of this orbit and this Fourier coefficient,
 so the Fourier coefficient gives rise to a function on this product.   In view of our work with cocycles above, this extends to covering groups. 

We are now ready to discuss our global integrals.   Fix $k$ and let $\kappa=1$ if $k$ is even, and $\kappa=2$ if $k$ is odd. 
Let $g\in Sp_{2n}^{(\kappa r)}(\A)$ and $h\in SO_{k}^{(r)}({\A})$.  Then
 the Fourier coefficient of interest to us, corresponding to the unipotent orbit $((r-1)^k1^{2n})$, is
\begin{equation}\label{lift1}
\int\limits_{[U_{k,r_1,n}]}
\theta_{2nk}^{(2),\psi}(l(u)\iota^{(2)}_1(p^{(1)}(h),p^{(\kappa)}(g)))\,
\theta_{2n+k(r-1)}^{(r)}(u\,\iota^{(r)}_2(h,p^{(r)}(g)))\,\psi_{U_{k,r_1,n}}(u)\,du.
\end{equation}
Here $\theta_{2n+k(r-1)}^{(r)}$ is a vector in the space of the representation $\Theta_{2n+k(r-1)}^{(r)}$,  $\theta_{2nk}^{(2),\psi}$ is a vector in the space of the representation $\Theta_{2nk}^{(2)}$
(which depends on $\psi$), and the map $l$ is as in \eqref{Heisenberg-hom} with $a=k$, $b=r_1$ and $c=n$.   Note that both of these theta functions are genuine functions. 
We shall use this Fourier coefficient as our integral kernel.

To formulate the theta lift,  
let $\pi^{(\kappa r)}$ denote a genuine irreducible cuspidal automorphic representation of ${Sp}_{2n}^{(\kappa r)}({\A})$
where the metaplectic groups are constructed in the prior section.  For fixed $k$, by means of the isomorphism $i$ we may realize  $\pi^{(\kappa r)}$ by means of genuine
automorphic functions on the group $\widetilde{Sp}_{2n}^{(\kappa r)}({\A})$, and we do so henceforth.
Here genuine means that the functions in $\pi^{(\kappa r)}$ transform under the center of 
$\widetilde{Sp}_{2n}^{(\kappa r)}({\A})$,
$\{(1,\mu)\mid\mu\in\mu_{\kappa r}\}$,
by a fixed embedding $\epsilon':\mu_{\kappa r}\to\C^\times$ which is compatible with $\epsilon$ and, if $k$ is odd, with the isomorphism of $\mu_{2r}$ with $\mu_r\times \mu_2$ selected above.
Specifically, let $p_2$ denote the projection from any covering group to its second factor (a root of unity of $F$).  Then
we require that  $\epsilon'(\mu)=\epsilon(p_2(p^{(r)}(1,\mu)))\,p_2(p^{(\kappa)}(1,\mu))$ for all $\mu\in \mu_{\kappa r}$.
Let $\varphi^{(\kappa r)}$ denote a vector in the space of $\pi^{(\kappa r)}$.
 
\begin{definition}
The theta lift of $\pi^{(\kappa r)}$ is the representation $\sigma_{n,k}^{(r)}$ of $SO_{k}^{(r)}({\A})$ generated by the functions $f(h)$ defined by
\begin{multline}\label{lift2}
f(h)=
\int\limits_{[Sp_{2n}^{(\kappa r)}]}
\int\limits_{[U_{k,r_1,n}]}\\
\overline{\varphi^{(\kappa r)}(i(g))}\, 
\theta_{2nk}^{(2),\psi}(l(u)\iota^{(2)}_1(p^{(1)}(h),p^{(\kappa)}(g)))\,
\theta_{2n+k(r-1)}^{(r)}(u\,\iota^{(r)}_2(h,p^{(r)}(g)))\,\psi_{U_{k,r_1,n}}(u)\,du\,dg.
\end{multline}
as each of $\varphi^{(\kappa r)}$, $\theta_{2nk}^{(2),\psi}$, $\theta_{2n+k(r-1)}^{(r)}$ 
varies over the functions in its representation space. 
\end{definition} 
\noindent
For convenience we sometimes shorten the notation in working with \eqref{lift2} below, writing 
$\iota_1(h,g)$ in place of $\iota^{(2)}_1(p^{(1)}(h),p^{(\kappa)}(g)))$ and $\iota_2(h,g)$ in place of 
$\iota^{(r)}_2(h,p^{(r)}(g)))$.

In \eqref{lift2} the covers in $g$ are compatible (otherwise this integral would vanish identically for trivial reasons). 
Also, the integral converges absolutely. This follows from the cuspidality of the representation $\pi^{(\kappa r)}$. Each function $f$ is a genuine function on $SO_k^{(r)}(\A)$.
The construction given by the space generated by the integrals \eqref{lift2}
defines a mapping from the set of irreducible cuspidal representations of 
${Sp}_{2n}^{(\kappa r)}({\A})$ to the set of representations of the quotient $SO_k(F)\backslash SO_k^{(r)}({\A})$. In the following Sections we will study the properties of this mapping.

As mentioned in the introduction, this construction may be viewed as an extension of the well-known theta lift associated to the reductive dual pair $(SO_k,Sp_{2n})$ if $k$ is even, and  to the 
reductive dual pair $(SO_k,Sp_{2n}^{(2)})$ if $k$ is odd. Indeed, in integral \eqref{lift2} we assumed that $r\ge 3$ is odd. 
If we formally set $r=1$, let $\Theta_{2n+k(r-1)}^{(r)}$ be the trivial representation, and let $U_{k,0,n}$ be the trivial group, then we get the theta lift.  
Just as the theta lift can be used to go from either group in the reductive dual pair to the other, 
we could also use the same integral kernel to construct a mapping from the set of irreducible cuspidal automorphic representations of $SO_k^{(r)}({\A})$ to automorphic representations of 
${Sp}_{2n}^{(\kappa r)}({\A})$.

There is also a local analogue of this integral given by a Hom space, which naturally gives rise to a map of representations 
over local fields that formally generalizes the Howe correspondence.  This will be treated for unramified principal series in Section~\ref{the-unramified-corr} below.

To conclude this section, let us mention that the theta kernel above is consistent with the Dimension Equation of \cite{G}, \cite{F-G4}.  As explained in \cite{F-G4}, Section 6, 
the Dimension Equation
states that when a theta kernel gives a correspondence, there is an equality relating the dimensions of the groups in the integral and the dimensions of the automorphic representations appearing
in and arising from the construction.  Here the dimension of an automorphic representation means its Gelfand-Kirillov dimension in the sense of \cite{G}.  
In our case, this is the equality
$$\dim(\pi)+ \dim(\Theta_{2nk}^{(2)})+ \dim(\Theta_{2n+k(r-1)}^{(r)}) =
\dim(Sp_{2n})+\dim (U_{k,r_1,n})+ \dim(\sigma),$$
where $\pi$, $\sigma$ are as above.
In the case that $\pi$ and $\sigma$ are both generic, and when $n=[k/2]$,
we expect a functorial correspondence 
between $SO_{k}^{(r)}$ and $Sp_{2n}^{(\kappa r)}$. And indeed, it may be checked
(using formulas (2), (3) in \cite{F-G4}) that in this situation
the above equation holds, assuming Conjecture~\ref{conj1} below.  
This equality is another motivation for
choosing the orbit $((r-1)^{k}1^{2n}$) for the
unipotent integration in \eqref{lift1}.

\section{The Unipotent Orbit of The Theta Representation}\label{conjtheta}

In this section we discuss the Fourier coefficients of the automorphic theta representation $\Theta_{2l}^{(r)}$.  
These coefficients, obtained by integrating functions in the representation space against certain characters of unipotent groups, 
 are indexed by unipotent orbits, which may be described by means of certain partitions of $2l$; see, for example,  \cite{G} for more information. 
The set of unipotent orbits of $Sp_{2l}$ is a partially ordered set. Let ${\mathcal O}(\Theta_{2l}^{(r)})$ denote the set of unipotent orbits ${\mathcal O}$ which are 
maximal with respect to the property that
$\Theta_{2l}^{(r)}$ has a non-zero Fourier coefficient with respect to ${\mathcal O}$.  We expect that this set is a singleton. More precisely,  we make the following
conjecture.
\begin{conjecture}\label{conj1}  Let $r>1$ be an odd integer, and write $2l=\alpha r+\beta$ where $0\le \beta< r$.
Let ${\mathcal O}_c(\Theta_{2l}^{(r)})$ denote the unipotent orbit
$${\mathcal O}_c(\Theta_{2l}^{(r)})=\begin{cases}(r^\alpha \beta)&\text{if $\alpha$ is even}\\
(r^{\alpha -1}(r-1)(\beta+1))&\text{if $\alpha$ is odd.}\end{cases}$$
Then ${\mathcal O}(\Theta_{2l}^{(r)})=\{{\mathcal O}_c(\Theta_{2l}^{(r)})\}$.
\end{conjecture}

This conjecture is also presented in \cite{F-G2}. Gao and Tsai \cite{G-T} have recently generalized Conjecture~\ref{conj1}  to other groups, and also given an 
archimedean analogue.  The compatibility of their conjecture with Conjecture~\ref{conj1}
is established in Section 4 of their paper.

Although we are not able to prove Conjecture~\ref{conj1} in general, we can prove it in some cases, as follows.

\begin{theorem}\label{theta05}
\begin{enumerate}
\item
For all positive integers $l$, if $\mathcal{O}\in {\mathcal O}(\Theta_{2l}^{(r)})$, then $\mathcal{O}\le {\mathcal O}_c(\Theta_{2l}^{(r)})$.
\item Assume that $l=0,1,2,r-3,r-2,r-1$. Let $n$ denote a non-negative integer, and assume that if $l=0$, then $n\ge 1$. 
Then Conjecture \ref{conj1} holds for the group $Sp_{2(l+nr)}^{(r)}({\A})$. In particular, when $r=3,5$ the conjecture holds for all $l$ and $n$.
\end{enumerate}
\end{theorem}
The first case for which we cannot prove Conjecture~\ref{conj1} is the group $Sp_6^{(7)}({\A})$, where the conjecture states that 
$\Theta_{6}^{(7)}$ is generic. The unramified constituents of $\Theta_{6}^{(7)}$  are generic by \cite{Gao2}, but we do not know of a global result.

We start
with the vanishing property of the representations $\Theta_{2l}^{(r)}$.  The proof requires some local preparations.
If $U$ is any unipotent subgroup of a $p$-adic group or a metaplectic cover $G$,
$\lambda$ is a character of $U$,  and $(\pi,V)$ is a representation of $G$,
then the Jacquet module $J_{U,\lambda}(\pi)$ is the quotient of $V$ by  the subspace spanned by vectors of the form $\pi(u)v-\lambda(u)v$ with $u\in U$, $v\in V$. 
When $\lambda$ is nontrivial, sometimes we refer to this as the twisted Jacquet module.
When $\lambda$ is trivial we omit it from the notation, and denote this quotient as $J_U(\pi)$, the untwisted Jacquet module. 
The map $V\mapsto J_{U,\lambda}(\pi)$ is the Jacquet functor.

Let $\nu$ be a finite place of $F$, $\Theta^{(r)}_{GL_m,\nu}$ denote the local theta representation on the $r$-fold cover of $GL_m(F_\nu)$ treated by Kazhdan and Patterson \cite{K-P}, Section~I, with $c=0$ (or the corresponding character if $m=1$), and $\Theta_{2l,\nu}^{(r)}$ be the local theta representation for $Sp_{2l}^{(r)}(F_\nu)$.  Then we have the following 
proposition.

\begin{proposition}\label{constant-term-theta} Let $P$ be the unipotent radical of $Sp_{2l}$ whose Levi factor is $GL_m\times Sp_{2l-2m}$, $1\leq m\leq l$, and let $U$ be its unipotent radical.  Then 
$$J_U\left(\Theta_{{2n},\nu}^{(r)}\right)=\Theta^{(r)}_{GL_m,\nu}\otimes \Theta^{(r)}_{{2l-2m},\nu}.$$ Here the factor $\Theta^{(r)}_{{2l-2m},\nu}$ is omitted if $m=l$.
\end{proposition}

The proof of this result follows exactly the proof of \cite{B-F-G1}, Theorem 2.3, which gives the same statement in a slightly different situation, and so is omitted here.  There is also
a global analogue; see for example \cite{F-G2}, Proposition 1.

The proof of vanishing requires two lemmas that are cases of more general results.  The first, due to Gao \cite{Gao2}, gives
 information about when an unramified local constituent of 
$\Theta_{2l}^{(r)}$ fails to be generic.  (Similar information for covers of the general linear group is due to Kazhdan and Patterson \cite{K-P}; see Theorem I.3.5 there.)

\begin{lemma}[Gao]\label{theta-generic}
Suppose that $2l>r$.  Then the local theta representation $\Theta^{(r)}_{{2l},\nu}$ at an unramfied place $\nu$ is not generic.
\end{lemma}

This follows directly from \cite{Gao2}  by combining (the more general) Theorem 1.1 with Proposition 5.1 there. 

The second is the following result, which is a special case of Leslie \cite{Leslie}, Proposition 6.5.

\begin{lemma}[Leslie]\label{supercuspidal-lemma}
Let $\pi$ be a non-zero smooth admissible genuine representation of $Sp_{2l}^{(2r)}(F_\nu)$ that is unramified.  Then $\pi$ is not supercuspidal.
\end{lemma}

We now give our result on the vanishing of Fourier coefficients for $\Theta_{2l}^{(r)}$.

\begin{proposition}\label{theta01}
Suppose that ${\mathcal O}$ is a unipotent orbit which is greater than ${\mathcal O}_c(\Theta_{2l}^{(r)})$ or that is not related to
${\mathcal O}_c(\Theta_{2l}^{(r)})$. Then the representation $\Theta_{2l}^{(r)}$ has no non-zero Fourier coefficient corresponding to
${\mathcal O}$. 
\end{proposition}
Note that Proposition~\ref{theta01} is equivalent to the first part of Theorem~\ref{theta05}.
\begin{proof}
Recall that we write $2l=\alpha r+\beta$ with $0\leq \beta<r$.
If $\alpha=0$ then ${\mathcal O}_c(\Theta_{2l}^{(r)})=(\beta)$, so $\mathcal{O}\leq {\mathcal O}_c(\Theta_{2l}^{(r)})$ for all orbits $\mathcal{O}$, and there is 
nothing to prove.
So suppose that $\alpha>0$. 
Let $\mathcal{O}$ be a unipotent orbit satisfying the conditions of Proposition~\ref{theta01}, let
$U_{\mathcal O}$ be the corresponding upper triangular unipotent subgroup of $Sp_{2l}$, and let
$\psi_{\mathcal O}$ denote any character of $U_{\mathcal O}$ such that the Fourier coefficient
\begin{equation}\label{theta1}
\int\limits_{[U_{\mathcal{O}}]}
\phi(ug)\,\psi_{\mathcal O}(u)\,du
\end{equation}
corresponds to the unipotent orbit ${\mathcal O}$ (see for example \cite{B-F-G1}, Section 4). (This is a slight abuse of notation as $\psi_{\mathcal{O}}$ is not uniquely
determined by $\mathcal{O}$.) We need to prove that  integral \eqref{theta1} is zero for all vectors $\phi$
in the space of $\Theta_{2l}^{(r)}$ and all $g\in Sp_{2l}({\A})$. 

First we show that it is enough to prove the vanishing of the coefficients for ${\mathcal O}_k:=((2k)1^{2l-2k})$ with $2k>r$. Indeed, by \cite{G-R-S3}, 
Lemma 2.6, this implies the
vanishing for all orbits of the form ${\mathcal O}=((2k)p_1^{e_1}\dots p_d^{e_d})$ with all $p_i\leq 2k$, $e_i\ge0$.  This in term implies the vanishing for orbits of the form
$((2k+1)^2p_1^{e_1}\dots p_d^{e_d})$ with all $p_i\leq 2k+1$, $e_i\ge0$ by \cite{G-R-S3}, Lemma 2.4.  These two Lemmas are formulated for the symplectic
groups but their proofs go over to their metaplectic covers without change.  Also, the arguments can be made local without difficulty, i.e.\ the
corresponding $\psi_{\mathcal{O}}$-twisted Jacquet module at a good unramified place is zero. 

We are reduced to showing the vanishing for ${\mathcal O}_k=((2k)1^{2l-2k})$ with $2k>r$. It is enough to prove the corresponding local statement.
That is, we will show that at a good unramified finite place $\nu$ (i.e.\ $\text{ord}_\nu(2r)=0$), the $\psi_{\mathcal{O}_k}$-twisted Jacquet module
$J_{V_2(\mathcal{O}_k),\psi_{\mathcal{O}_k}}(\pi_{2l})$ is zero, where $\pi_{2l}$ is the local constituent of $\Theta_{2l}^{(r)}$ at $\nu$.
(We use $\pi_{2l}$ in place of $\Theta^{(r)}_{{2l},\nu}$ to simplify the notation.)
In particular $\pi_{2l}$ is a genuine representation 
of the covering group $Sp_{2l}^{(r)}(F_\nu)$ over the nonarchimedean local field $F_\nu$.
To establish the vanishing, we follow Leslie \cite{Leslie}, Section 8, who investigated a similar problem
for the 4-fold cover.  
 Let $m$ be the integer such that the $\psi_{\mathcal{O}}$-twisted Jacquet module of $\pi_{2l}$ for 
 $\mathcal{O}=((2p)1^{2l-2p})$ is zero for $p>m$ but such that the module for $((2m)1^{2l-2m})$ is nonzero.
Observe that $m<l$. Indeed, if $m=l$ this would imply that $\pi_{2l}$ is generic. Since $2l>r$, this contradicts Lemma~\ref{theta-generic}.

We also suppose by induction on $l$ that the twisted Jacquet modules for $\pi_{2j}$, $1\leq j<l$, vanish on the orbits $((2q)1^{2j-2q})$ when $2q>r$.  
Note that the
base of the induction, that is the case $l=2$, is clear. We remark that even the case $l=2$ requires $2k>r$.  Indeed, 
if $r=3$ then the theta representation for $Sp_4^{(3)}$ has a nonzero Fourier coefficient corresponding
to $(21^2)$, but not corresponding to $(4)$ (see \cite{F-G3}, Lemma 2 and Proposition 3). 

The idea of the proof is to establish that the twisted Jacquet modules $J_{V_2(\mathcal{O}_k),\psi_{\mathcal{O}_k}}(\pi_{2l})$ vanish by studying their descent properties.  
To prove that these modules vanish, 
 it is sufficient to show that the Fourier Jacobi modules
$FJ_{m,\alpha}(\pi_{2l})$ defined in \cite{Leslie}, Section 6 (here $\alpha\in F_\nu^\times$), all vanish (\cite{Leslie}, Corollary 6.4; as noted there this is the local version of \cite{G-R-S3}, Lemma 1.1).
To do so, we will show that each Fourier Jacobi module $FJ_{m,\alpha}(\pi_{2l})$ is a supercuspidal representation of $Sp_{2l-2m}^{(2r)}(F_\nu)$.   Since $l>m$, if it were nonzero this
would contradict Lemma~\ref{supercuspidal-lemma}.
Note that in taking the Fourier Jacobi module the cover is doubled since $r$ is odd and the Fourier Jacobi module uses a twist by the Weil representation which lives on the double cover, 
so $FJ_{m,\alpha}(\pi_{2l})$ is a genuine unramified representation of the $2r$-fold cover $Sp_{2l-2m}^{(2r)}(F_\nu)$.
Aside from this the proof that the vanishing of all $FJ_{m,\alpha}(\pi_{2l})$ implies the vanishing of $J_{V_2(\mathcal{O}_k),\psi_{\mathcal{O}_k}}(\pi_{2l})$  is the same as that in \cite{Leslie}.  
(Fourier Jacobi modules 
and descents are also discussed in \cite{G-R-S4}, Section~3.8 and following.)

To prove that the representation $FJ_{m,\alpha}(\pi_{2l})$ is supercuspidal we must check that its untwisted Jacquet functors are all zero.
For $1\leq a\leq l-m$, let $J_a$ be the (untwisted) Jacquet functor on admissible 
representations of $Sp_{2l-2m}^{(2r)}(F_\nu)$ corresponding to
the unipotent radical of the parabolic subgroup of $Sp_{2l-2m}$ with Levi component $GL_a\times Sp_{2l-2m-2a}$.
Then to prove supercuspidality, we must establish the vanishing of  $J_a(FJ_{m,\alpha}(\pi_{2l}))$, $1\leq a\leq l-m$.
After conjugating by a suitable Weyl element, this Jacquet module is isomorphic to the $\psi_a$-twisted Jacquet module of $\pi_{2l}$ with respect to the subgroup
$$U_a(F_\nu):=\left\{\begin{pmatrix}I_a&X&Y\\&v&X^*\\&&I_a\end{pmatrix}\right\}$$
where $X\in {\rm Mat}_{a\times (2l-2a)}(F_\nu)$ has zero entries in the first column and $v$ is upper triangular unipotent with center block $I_{2(l-a)}$,
and where the character $\psi_a$ of $U_a(F_\nu)$ is given by
$$\psi_a(u)=\psi(u_{a+1,a+2}+\dots+u_{a+m,a+m+1}+\alpha u_{a+m,2l-a-m+1}).$$
Let $H\subset Sp_{2l}$ be the subgroup of unipotent matrices which are zero above the main diagonal except on column $a+1$ and row
$2l-a$. Using  the Geometrical Lemma of \cite{B-Z}, p.\ 448, if the module $J_{U_a,\psi_a}(\pi_{2l})$ is nonzero, then $J_{H,\psi_H}(J_{U_a,\psi_a}(\pi_{2l}))$ is nonzero for some
character $\psi_H$ of $H$.
There are two orbits under the action of $GL_a$.  Using Proposition~\ref{constant-term-theta}, the module with $\psi_H=1$ factors through a coefficient of $\pi_{2l-2a}$ with respect to the orbit
$((2m)1^{2l-2m-2a})$, and this vanishes by induction (note that since $2l>r$, also $2m>r$).  The non-trivial orbit contributes a similar Jacquet module with $a$ replaced by $a-1$.
Under the action of
$GL_{a-1}$ there are again two orbits. As above the contribution from the trivial orbit is
zero by the induction hypothesis. Thus, we are again left with the non-trivial orbit. 
Repeating, we arrive at the nonvanishing of a twisted Jacquet module of $\pi_{2l}$ with respect to the orbit $((2m+2a)1^{2l-2m-2a})$.  However, $m$ was chosen to be maximal
such that the module for $((2m)1^{2l-2m})$ is nonzero, and $a\geq1$.  
We conclude that $J_{U_a,\psi_a}(\pi_{2l})=0$, and hence that all the Jacquet modules $J_a(FJ_{m,\alpha}(\pi_{2l}))$ indeed vanish.
Thus the representation $FJ_{m,\alpha}(\pi_{2l})$ is supercuspidal, as claimed.
This completes the proof of the vanishing. 
\end{proof}

The local information developed above will also be used in treating the unramified correspondence in Section~\ref{the-unramified-corr}.
We record the local vanishing in Proposition~\ref{theta-local} below (there we use $\Theta_{2l}^{(r)}$ for the representation denoted $\Theta^{(r)}_{{2l},\nu}$ here).

Our next proposition requires a result of Jiang and Liu (\cite{J-L}, Proposition 3.3) about the 
Fourier coefficients of automorphic representations on symplectic groups and composite partitions which we recall here for 
completeness.  
The notation is as in \cite{J-L}.  In fact the result of Jiang and Liu is bit sharper since they give information about the characters supporting the Fourier coefficients but we will not need this here.
Their result is not stated for covering groups but the proof is identically the same.

\begin{proposition}[Jiang and Liu]\label{j-l} Let $\pi$ be an irreducible automorphic representation of $Sp_{2n}(\A)$ or a metaplectic cover that is realized in the space of automorphic forms,
and $\underline{p}=[(2k+1)^2p_1^{e_1}p_2^{e_2}\cdots p_r^{e_r}]$ be a symplectic partition of $2n$ with $2k+1\geq p_1\geq p_2\geq \cdots \geq p_r$; $e_i=1$ if $p_i$ is even; and 
$e_i=2$ if $p_i$ is odd.  Then $\pi$ has a non-vanishing Fourier coefficient attached to $\underline{p}$ if and only if it has a non-vanishing Fourier coefficient attached
to the composite partition $[(2k+1)^21^{2n-4k-2}]\circ [p_1^{e_1}p_2^{e_2}\cdots p_r^{e_r}]$.
\end{proposition}

Next we study the nonvanishing of the Fourier coefficients of theta functions. We start with a proposition.
\begin{proposition}\label{theta02}
Suppose Conjecture \ref{conj1} holds for a given $r$ and $l$.  Then it holds for $r+l$ and $l$ as well. 
\end{proposition}
\begin{proof}
We know from Proposition \ref{theta01} that if $\mathcal{O}\in{\mathcal O}(\Theta_{2(l+r)}^{(r)})$
then $\mathcal{O}\le {\mathcal O}_c(\Theta_{2(l+r)}^{(r)})$. Thus, we only need to prove that the representation $\Theta_{2(l+r)}^{(r)}$ has a 
non-zero Fourier coefficient corresponding to the unipotent orbit 
${\mathcal O}_c(\Theta_{2(l+r)}^{(r)})$.

Write $2l=\alpha r+\beta$ with $0\le \beta\ <r$. Then $2(l+r)=(\alpha+2)r+\beta$. By assumption, there is a unipotent subgroup 
$U({\mathcal O}_c(\Theta_{2l}^{(r)}))$ of $Sp_{2l}$, and a character $\psi_U$ 
of the quotient $[U({\mathcal O}_c(\Theta_{2l}^{(r)}))]$ such that the Fourier coefficient 
\begin{equation}\label{theta100}
\int\limits_{[U({\mathcal O}_c(\Theta_{2l}^{(r)}))]}
\varphi_{2l}^{(r)}(u)\,\psi_{U}(u)\,du
\end{equation}
corresponds to the unipotent orbit ${\mathcal O}_c(\Theta_{2l}^{(r)})$ and is not zero for some $\varphi_{2l}^{(r)}$ in the space of $\Theta_{2l}^{(r)}$. 
It follows from \cite{F-G2}, Proposition 2.1, or as in Offen and Sayag \cite{O-S}, p.\ 10, that the constant terms of theta functions with respect to the unipotent radical of a Levi 
in fact realize products of theta functions on the Levi. (In \cite{O-S}, this is formulated in terms of the surjectivity of an intertwining operator.) This is the global analogue
of Proposition~\ref{constant-term-theta} above.  We conclude that 
there is a function
$\varphi_{2(l+r)}^{(r)}$ in the space of
$\Theta_{2(l+r)}^{(r)}$ such that the integral
\begin{equation}\label{theta2}
\int\limits_{[U({\mathcal O}_c(\Theta_{2l}^{(r)}))]}
\int\limits_{[V_2]}
\int\limits_{[V_1]}
\varphi_{2(l+r)}^{(r)}\left (v_1\begin{pmatrix} v_2&&\\ &u& \\
&&v_2^*\end{pmatrix} \right )\,\psi_{V_2}(v_2)\,\psi_{U}(u)\,dv_1\,dv_2\,du
\end{equation}
is not zero. 
Here $V_1$ is the unipotent radical of the maximal parabolic subgroup of $Sp_{2(l+r)}$ whose Levi part is $GL_r\times Sp_{2l}$, 
$V_2$ is the maximal unipotent subgroup of $GL_r$ consisting of upper unipotent matrices, and
the character $\psi_{V_2}$ is the Whittaker character defined on the group $V_2$.  

To relate the Fourier coefficient in equation \eqref{theta2}, 
to the Fourier coefficient corresponding to the unipotent orbit 
${\mathcal O}_c(\Theta_{2(l+r)}^{(r)})$,
we first use the result of Jiang and Liu recalled in Proposition~\ref{j-l}.  We deduce from this that the Fourier coefficient corresponding to the 
unipotent orbit ${\mathcal O}_c(\Theta_{2(l+r)}^{(r)})$ is not zero for some choice of data, if and only if the Fourier coefficient corresponding to the 
orbit $(r^21^{2l})\circ {\mathcal O}_c(\Theta_{2l}^{(r)})$ is not zero for some choice of data.
The Fourier coefficient corresponding to the  last orbit can be written as follows 
\begin{equation}\label{theta20}
\int\limits_{[U_{2,r_1,l+1}]}
\int\limits_{[U({\mathcal O}_c(\Theta_{2l}^{(r)}))]}
\varphi_{2(l+r)}^{(r)}\left (v\begin{pmatrix} I_{2r}&&\\
&u&\\ &&I_{2r}\end{pmatrix}\right )\psi_{U_{2,r_1,l+1}}'(v)\psi_U(u)\,du\,dv.
\end{equation}
We recall that the group $U_{2,r_1,l+1}$ was defined right before equation \eqref{mat4}.  Also, the character $\psi_U$ was defined in equation \eqref{theta100}. The character 
$\psi_{U_{2,r_1,l+1}}'$ is defined as follows. For $v\in U_{2,r_1,l+1}$ consider its factorization given by equation \eqref{mat2}. Then we define $\psi_{U_{2,r_1,l+1}}'(v)=
\psi_{U_{2,r_1,l+1}}(v)\psi_V'(v)$ where $\psi_{U_{2,r_1,l+1}}$ is defined in equation \eqref{character} and $\psi_V'(v)$ is 
defined as follows. Let $Y=(y_{i,j})\in \Mat_{2\times 2(l+1)}$. Define $\psi'(Y)=\psi(y_{1,1}+y_{2,2(l+1)})$. Then, for $v$ as in equation \eqref{mat2} we define $\psi_V'(v)=\psi'(Y)$. 

Using Proposition  \ref{theta01} we claim that integral \eqref{theta2} is not zero for some choice of data if and only if integral \eqref{theta20} is not zero for some choice of data. 
Since integral \eqref{theta2} is not zero for some choice of data, this claim will imply the Proposition.

The proof of the claim is standard, and follows in a similar way as the proof of Proposition 3.3 in \cite{J-L}. 
We give some details. Let $w$ denote the monomial matrix in the Weyl group of $Sp_{2(n+l)}$ defined as 
$$w=\begin{pmatrix}  \epsilon_1&&\epsilon_2\\ &I_{2l}&\\
\epsilon_3&&\epsilon_4\end{pmatrix}\ \ \ \ \ \ \ \ \ \ 
\epsilon_i\in \Mat_r$$
Here $\epsilon_1(i,2i-1)=1$ for all $1\le i\le (r+1)/2$, and 
$\epsilon_2(i+(r+3)/2,2i+2)=1$ for all $0\le i\le (r-3)/2$. This determines $w$ uniquely. We have $\varphi_{2(l+r)}^{(r)}(h)=
\varphi_{2(l+r)}^{(r)}(wh)$. Hence, after conjugating $w$ to the right, integral \eqref{theta2} 
is not zero for some choice of data if and only if the integral
\begin{equation}\label{theta21}\notag
\int\limits_{[U({\mathcal O}_c(\Theta_{2l}^{(r)}))]}
\int\limits_{[V_2]}
\varphi_{2(l+r)}^{(r)}\left (\begin{pmatrix}  I_r&A&B\\ &I_{2l}&A^*\\ &&I_r \end{pmatrix}\begin{pmatrix} v_2&&\\
&u&\\ &&v_2^*\end{pmatrix}\begin{pmatrix} I_r&&\\ C&I_{2l}&
\\ D&C^*&I_r\end{pmatrix}
\right )\\  \psi_{V_2}(v_2)\psi_U(u)\,du\,dv
\end{equation}
is not zero for some choice of data. Here, the matrices $A$, $B$, $C$ and $D$ are suitable matrices which are obtained in a similar way 
as in the proof of Proposition 3.3 in \cite{J-L}. We omit the details. Performing a root exchange as in \cite{G-R-S4}, Section 7.1, and using Proposition \ref{theta01}, 
we deduce that the above integral is not zero for some choice of data if and only if integral \eqref{theta20} is not zero for some choice of data.
\end{proof}

It follows from Proposition~\ref{theta02} that given an odd number $r$, to prove Conjecture~\ref{conj1} for all symplectic groups $Sp_{2l}^{(r)}$ it is enough to prove the result 
for the symplectic groups $Sp_{2l}^{(r)}$ with $1\le l\le r-1$. 

Note that when $2\le 2l<r$, Conjecture~\ref{conj1}  states that ${\mathcal O}(\Theta_{2l}^{(r)})=\{(2l)\}$, that is, that the theta representation is generic. When $r<2l<2r-1$, the conjecture is that 
${\mathcal O}(\Theta_{2l}^{(r)})=\{((r-1)(2l-r+1))\}$. The next Proposition gives a lower bound for some element of the set ${\mathcal O}(\Theta_{2l}^{(r)})$ (or for the associated orbit if this
set is, as expected, a singleton).
\begin{proposition}\label{theta03}
Assume that $1\le l\le r$.
Then the representation $\Theta_{2l}^{(r)}$ has a non-zero Fourier coefficient corresponding to the unipotent orbit $(l^2)$.  
\end{proposition}
\begin{proof}
Let $V_l$ denote the unipotent radical of the parabolic subgroup of $Sp_{2l}$ whose Levi part is isomorphic to $GL_2^{l/2}$ if $l$ is even and to 
$GL_2^{(l+1)/2}$ if $l$ is odd. Let $\psi_{V_l}$ denote the character of $[V_l]$ given by 
$\psi_{V_l}(v)=\psi(\sum_{i=1}^{l-1}v_{i,i+2})$. 
To prove the proposition we will assume that the integral 
\begin{equation}\label{theta3}
\int\limits_{[V_l]}
\varphi^{(r)}(vg)\,\psi_{V_l}(v)\,dv
\end{equation}
is zero for all functions $\varphi^{(r)}$ in $\Theta_{2l}^{(r)}$ and $g\in Sp_{2l}^{(r)}(\A)$
 and derive a contradiction.  We may take $g=e$. Let $w\in Sp_{2l}(F)$ be the monomial matrix with non-zero entries
$\pm1$ and such that 
 $w_{i,2i-1}=1$ for all $1\le i\le l$. 
Since $w\in Sp_{2l}(F)$, we have 
$\varphi^{(r)}(v)=\varphi^{(r)}(wv)$. Conjugating $w$  across $v$, we deduce that the integral
\begin{equation}\label{theta4}
\int\varphi^{(r)}\left (\begin{pmatrix} I_l&A\\ &I_l\end{pmatrix}
\begin{pmatrix} B&\\ &B^*\end{pmatrix}\begin{pmatrix} I_l&\\ C&I_l\end{pmatrix}\right )\psi_0(B)\,dA\,dB\,dC
\end{equation}
is zero for all choices of data.
Here $A$ and $C$ are integrated over $[\Mat_{l,1}^0]$
where $\Mat_{l,1}^0$ is the subgroup of $\Mat_l^0$
consisting of all matrices $X=(X_{i,j})\in \Mat_l^0$ such that
$X_{i,j}=0$ for all $i\ge j$ (the group $\Mat_l^0$ was defined at the beginning of Section \ref{basic}). 
The variable $B$ is integrated over $[L_l]$ where $L_l$ is the unipotent subgroup of $GL_l$ consisting of all upper triangular matrices, and
$\psi_0$ is the Whittaker character defined on $L_l$.

Integral \eqref{theta4} is a special case of the situation dealt 
with in \cite{J-L} Propositions 3.2 and 3.3. Performing root exchange as in \cite{J-L} (see also \cite{G-R-S4}, Section 7.1), and 
using the fact that if an automorphic function is zero then all its Fourier coefficients are zero, we deduce that the integral  
\begin{equation}\notag
\int\varphi^{(r)}\left (\begin{pmatrix} I_l&A\\ &I_l\end{pmatrix}
\begin{pmatrix} B&\\ &B^*\end{pmatrix}\right )\psi_0(B)\,dA\,dB
\end{equation}
is zero for all choices of data. Here $A$ is integrated over the quotient $[\Mat_{l}^0]$, and $B$ is integrated as in \eqref{theta4}. Applying 
\cite{F-G2}, Proposition 2.1 (or, again, as in \cite{O-S} p.\ 10), we deduce that the Whittaker coefficient of the theta representation of the 
group $GL_l^{(r)}({\A})$ is zero for all choices of data. However, since $1\le l\le r$ it follows from  \cite{K-P} that this theta representation on  
$GL_l^{(r)}({\A})$ is generic. Thus we have derived a contradiction.
\end{proof}

Denote by $e_{i,j}$ the square matrix whose $(i,j)$ entry is one and with zeros elsewhere. Given $l\geq1$ and $1\leq i,j\leq 2l$, let $e_{i,j}'=e_{i,j}\pm e_{2l-j+1,2l-i+1}$, 
with the sign determined so that the matrix $I_{2l}+e_{i,j}'$ is in $Sp_{2l}$.   Also, let 
$$E_{i,2l+1-i}=\{I_{2l}+te_{i,2l+1-i}\},\qquad E_{i,j}'=\{I_{2l}+te'_{i,j}\}~~(i+j\neq 2l+1)$$ 
be the associated one parameter unipotent subgroups, corresponding to the long (resp.\ short) roots of $Sp_{2l}$.

\begin{proposition}\label{theta04}
Suppose that $1\le l\le r-1$. 
\begin{enumerate}
\item If $l$ is odd, then the set ${\mathcal O}(\Theta_{2l}^{(r)})$ contains an orbit which is greater than or equal to the
unipotent orbit $((l+1)(l-1))$.
\item Suppose $l$ is even, $l\ne r-1$ and that 
${\mathcal O}\ge (l^2)$ for all ${\mathcal O}\in {\mathcal O}(\Theta_{2l}^{(r)})$.  Then  ${\mathcal O}(\Theta_{2l}^{(r)})$ is a singleton and
${\mathcal O}(\Theta_{2l}^{(r)})\ge ((l+2)(l-2))$.
\end{enumerate}
\end{proposition}
\begin{proof}
Consider first the case that $l$ is odd. By Proposition~\ref{theta03}, the Fourier coefficient  \eqref{theta3} is nonzero.  However,
the stabilizer of the unipotent orbit $(l^2)$ in $Sp_{2l}$ is $SL_2$, 
which embeds in the parabolic subgroup with Levi isomorphic to $GL_2^{(l+1)/2}$ by the diagonal embedding. 
Note that since $1\le l\le r-1$ and $l$ is odd, in fact $1\leq l\leq r-2$.  Thus the cover $Sp_{2l}^{(r)}({\A})$ restricts via this embedding to a group isomorphic to
$SL_2^{(r)}({\A})$.
The integral \eqref{theta3} then gives an automorphic function of $g\in SL_2^{(r)}({\A})$ which is genuine. 
Therefore, it can not be constant, and so this function has a nontrivial Whittaker coefficient. 

For $y\in\A$, let $x(y)\in Sp_{2l}(\A)$ be given by
 $x(y)=I_{2l}+\sum_{i=1}^{(l-1)/2}ye'_{2i-1,2i}+
ye_{l,l+1}$. 
We conclude that there a choice of data such that the integral
\begin{equation}\label{theta6}
\int\limits_{[V_l]}
\int\limits_{[\mathbb{G}_a]}
\varphi^{(r)}(vx(y)g)\psi_{V_l}(v)\,\psi(\alpha y)\,dy\,dv
\end{equation}
is not zero. Here $\alpha\in F^*$. 
Then  is not hard to check that, after root exchange, the integral \eqref{theta6} has as inner integration a Fourier coefficient corresponding 
to the unipotent orbit $((l+1)(l-1))$. We refer to \cite{F-G1}, following (9) there, for a similar computation.
Hence the first part is proved.

Next we consider the case when $l$ is even. The argument is different 
since we do not have unipotent elements inside the stabilizer of this unipotent orbit. 
However, the stabilizer contains the subgroup 
generated by the matrices $t(a)=\text{diag}(a,a^{-1},\ldots,a,a^{-1})$ and $w_0=\text{diag}(J_2,\ldots,J_2)$, 
with $J_2$ defined in Section~\ref{basic}.  Recall that the elements $\widetilde{t}(a)=(t(a),\eta^{-1}(t(a)))\in Sp_{2l}^{(r)}(\A)$ were introduced at the end of Section~\ref{basic}.

Let $a\in \A^*$ be an $r$-th power, and define 
\begin{equation}\label{theta7}
\Lambda(a)=\int\limits_{[V_l]}
\varphi^{(r)}(v\widetilde{t}(a)w_0)\,\psi_{V_l}(v)\,dv.
\end{equation}
Denote also by $\Lambda_0(e)$ integral \eqref{theta3} with $g=e$. Since $w_0$  stabilizes  the character 
$\psi_{V_l}$ we have $\Lambda(e)=\Lambda_0(e)$.

Performing Fourier expansions and root exchanges similar to those in the proof of Proposition 
\ref{theta03}, we obtain that $\Lambda(a)$ is equal to
\begin{equation}\label{theta8}
\int\varphi^{(r)}\left (\begin{pmatrix} I_l&A\\ &I_l\end{pmatrix}
\begin{pmatrix} B&\\ &B^*\end{pmatrix}\begin{pmatrix} I_l&\\ C&I_l\end{pmatrix}\widetilde{h}(a)ww_0\right )\psi_0(B)\,dA\,dB\,dC.
\end{equation}
Here, $A$ is integrated over $[\Mat_l^0]$, $B$ is integrated as in \eqref{theta4}, and $C$ is integrated over $\Mat_{l,1}^0(\A)$ (defined after \eqref{theta4}).  
The Weyl element $w$ was defined following \eqref{theta3}, and $\widetilde{h}(a)=w\widetilde{t}(a)w^{-1}=(h(a),\eta(t(a))^{-1})$, where $h(a)=\text{diag}(aI_l,a^{-1}I_l)$.
We want to emphasize the difference between the computation performed here, and the one performed in Proposition \ref{theta03}. 
In that Proposition we needed to show that a certain integral vanished for all choices of data. In this Proposition we need a precise identity. 
This is why we made the assumption that ${\mathcal O}\ge (l^2)$ for all ${\mathcal O} \in {\mathcal O}(\Theta_{2l}^{(r)})$. Indeed, this assumption implies that the
set ${\mathcal O}(\Theta_{2l}^{(r)})$ is a singleton, and that 
the representation $\Theta_{2l}^{(r)}$ has no non-zero Fourier coefficient 
corresponding to unipotent orbits which are not related to $(l^2)$, for example the orbit $((l+2)1^{l-2})$. To prove that 
$\Lambda(a)$ is equal to integral \eqref{theta8}, we need to make use of this.

Recall that $a$ is an $r$-th power. In \eqref{theta8} conjugate the matrix $\widetilde{h}(a)$ to the left. First, we get the factor $|h|^{-l^2/2}$ from the 
change in variables in $C$. The torus $h(a)$ commutes with the matrix $\text{diag}(B,B^*)$. We are left with the computation of the 
constant term consisting of all matrices $\left(\begin{smallmatrix} I&A\\ &I\end{smallmatrix}\right)$
where $A\in \Mat_l^0$. It follows from \cite{F-G2} Proposition 2.1 that
we obtain the factor of $\chi_{Sp_{2l}^{(r)},\Theta}(\widetilde{h}(a))$. By the formula in \cite{F-G2}, top of p.\ 93, this last term is 
equal to $|a|^{l(l+1)(r-1)/2r}$.  Putting this together, we have proved that   
\begin{equation}\label{theta9}
\Lambda(a)=|a|^{\frac{l(l+1)(r-1)}{2r}-\frac{l^2}{2}}\int\varphi^{(r)}\left (\begin{pmatrix} I_l&A\\ &I_l\end{pmatrix}
\begin{pmatrix} B&\\ &B^*\end{pmatrix}\begin{pmatrix} I_l&\\ C&I_l\end{pmatrix}ww_0\right )\psi_0(B)\,dA\,dB\,dC.
\end{equation}
Hence, $\Lambda(a)=|a|^{\frac{l(l+1)(r-1)}{2r}-\frac{l^2}{2}}\Lambda(e)=|a|^{\frac{l(l+1)(r-1)}{2r}-\frac{l^2}{2}}\Lambda_0(e)$.

Next we compute $\Lambda(a)$ in a different way. Going back to the definition in \eqref{theta7}, we first conjugate $w_0$ to the left. We obtain 
\begin{equation}\label{theta10}\notag
\Lambda(a)=\int\limits_{[V_l]}
\varphi^{(r)}(v \widetilde{t}(a^{-1}))\,\psi_{V_l}(v)\,dv.
\end{equation}
Repeating the same computations we performed in equations \eqref{theta8} and \eqref{theta9}, we obtain 
$\Lambda(a)=|a|^{-\frac{l(l+1)(r-1)}{2r}+\frac{l^2}{2}}\Lambda_0(e)$. By Proposition~\ref{theta03}, there is a choice of data 
such that $\Lambda_0(e)$ is not zero. Hence, we must have $\frac{l(l+1)(r-1)}{2r}=\frac{l^2}{2}$. This is equivalent to $l=r-1$,
so the second part follows.
\end{proof}

The first part of Theorem~\ref{theta05} is  in Proposition~\ref{theta01}. We now give the proof of second part of the Theorem.
\begin{proof} 
Consider first the case $l=0$. Then applying Proposition~\ref{theta02}, it is enough to prove the theorem for the group 
$Sp_{2r}^{(r)}({\A})$. According to Conjecture~\ref{conj1}, we need to prove that $\Theta_{2r}^{(r)}$ has a non-zero Fourier 
coefficient corresponding to the unipotent orbit $(r^2)$. This was proved in Proposition~\ref{theta03}. 

For all other cases stated in Theorem~\ref{theta05}, part 2, using Proposition~\ref{theta02}, it is enough to prove the result 
for the group $Sp_{2l}^{(r)}({\A})$. The case $l=1$ is clear. Next consider $l=2$. If $r=3$, it follows from Proposition~\ref{theta01} that 
$\Theta_{4}^{(3)}$ is not generic.  From Proposition~\ref{theta02},  $\Theta_{4}^{(3)}$ has a nonzero Fourier coefficient corresponding to the 
orbit $(2^2)$. If $r\ge 5$, it follows from Proposition~\ref{theta04} that ${\mathcal O}(\Theta_{4}^{(r)})$ is greater than $(2^2)$. 
Hence  $\Theta_{4}^{(r)}$ is generic and we are done. 

When $l=r-2$, the Conjecture states that ${\mathcal O}(\Theta_{2(r-2)}^{(r)})=\{((r-1)(r-3))\}$. This follows from Proposition~\ref{theta04}. 
When $l=r-1$, the Conjecture states that ${\mathcal O}(\Theta_{2(r-1)}^{(r)})=\{((r-1)^2)\}$. This follows from Proposition~\ref{theta03}.
Thus the second part of the Theorem is proved.
\end{proof}

\section{Cuspidality of the Lift}\label{cusp1}
In this Section we discuss the cuspidality of the representation 
$\sigma_{n,k}^{(r)}$. The main result is that the first non-zero occurrence of the generalized theta lift is automatically cuspidal.  This generalizes Rallis's tower property,
and is found in Theorem \ref{th2} below. The proof requires showing the vanishing of constant terms. We will establish this by considering various Fourier expansions,
using the process of root exchange (see \cite[Section 7.1]{G-R-S4}), and making use of two key ingredients: 
 the smallness of the representation 
$\Theta_{2n+k(r-1)}^{(r)}$ established in Proposition~\ref{theta01}, and the cuspidality of the representation $\pi^{(\kappa r)}$. 

We begin with several Lemmas. Let $U_m$ be the unipotent subgroup $U_m=U_{1,m+1,n+kr_1-m-1}^0$,
and let
$\psi_{U_m}$ denote the character $\psi_{U_m}=\psi_{U_{1,m+1,n+kr_1-m-1}}$  of this group (see \eqref{character}).  The first Lemma is closely related to Lemma 2.2 in \cite{G-R-S3}.
\begin{lemma}\label{lem1}
Suppose that $r\le m\le n+kr_1-1$. Then the integral 
\begin{equation}\label{cusp10}
\int\limits_{[U_m]}
\theta_{2n+k(r-1)}^{(r)}(u)\,\psi_{U_m}(u)\,du
\end{equation}
is zero for all choices of data.
\end{lemma}
\begin{proof}
Let $x(p)=I+pe_{m+1,2n+k(r-1)-m}$, with $I$ the identity matrix of size $2n+k(r-1)$. 
Expand \eqref{cusp10} along the abelian
group $\{ x(p)\}$.  This is a sum of integrals against characters $\psi(\alpha p)$, $\alpha\in F$. The nontrivial terms contribute zero. 
Indeed, the Fourier coefficient we obtain from a nontrivial term corresponds to the unipotent orbit $((2m+2)1^{2+k(r-1)-2m-2})$. Since $m\ge r$, this last 
unipotent orbit is not comparable with the unipotent orbit ${\mathcal O}_c(\Theta_{2n+k(r-1)}^{(r)})$ defined in Conjecture~\ref{conj1} above. 
By Proposition~\ref{theta01}, these Fourier coefficients vanish.

We are left with the contribution from $\alpha=0$.  That is, integral \eqref{cusp10} is equal to
\begin{equation}\label{cusp11}
\int\limits_{[U_m]}
\int\limits_{[\mathbb{G}_a]}
\theta_{2n+k(r-1)}^{(r)}(ux(p))\,\psi_{U_m}(u)\,dp\,du.
\end{equation}
The quotient group $\{x(p)\}U_m\backslash U_{m+1}$ can be identified with a row vector of size $2(n-m-1)+k(r-1)$. 
Expand integral \eqref{cusp11} along this quotient. There are two orbits under the action of the group $Sp_{2(n-m-1)+k(r-1)}(F)$. 
The trivial orbit contributes zero to integral \eqref{cusp11}. Indeed, to prove this, we use Proposition 1 in \cite{F-G2},
which identifies the constant term of a theta function with lower rank theta functions, with $a$ there equal to $m+1$. 
This implies that as an inner integration we obtain the Whittaker coefficient of the theta function defined on an $r$-fold cover of  $GL_{m+1}({\A})$. 
Since $m+1>r$, it follows from \cite{K-P} that this Whittaker coefficient is zero. We conclude that the integral \eqref{cusp10} is a sum of integrals of the form
\begin{equation}\label{cusp12}\notag
\int\limits_{[U_{m+1}]}
\theta_{2n+k(r-1)}^{(r)}(u)\,\psi_{U_{m+1}}(u)\,du.
\end{equation}
Applying induction, these integrals are all zero. Hence integral \eqref{cusp10} is zero for all choices of data.
\end{proof}

In the next Lemma we establish another vanishing result. Assume again that $r\le m\le n+kr_1-1$. Let $V_m$ denote the maximal unipotent subgroup of $GL_{m+1}$ consisting of upper 
triangular unipotent matrices. Let $\Mat_{m+1}^{00}$ denote the subgroup of $\Mat_{m+1}^{0}$ of all matrices $x\in \Mat_{m+1}^{0}$ such that $x_{i,j}=0$ for all $i\ge j$. 
Let $U_m^0$ denote the subgroup of $U_m$ which consists of all matrices 
\begin{equation}\notag 
t(v,x)=\begin{pmatrix} v&&x\\ &I_{2(n+kr_1-m-1)}&\\ &&v^*\end{pmatrix}\ \ \ \ v\in V_m,\ \ \ x\in \Mat_{m+1}^{00}.
\end{equation}
By restriction, the character $\psi_{U_m}$ is a character of $U_m^0$.
\begin{lemma}\label{lem2}
Assume that $r\le m\le n+kr_1-1$. Then the integral 
\begin{equation}\label{cusp13}
\int\limits_{[U_m^0]}
\theta_{2n+k(r-1)}^{(r)}(u)\,\psi_{U_m}(u)\,du
\end{equation}
is zero for all choices of data.
\end{lemma}
\begin{proof}
We start by defining some unipotent subgroups of $Sp_{2n+k(r-1)}$. Let $a=2(n+kr_1-m-1)$.
Let $b(m)=c(m)=m/2$ if $m$ is even, and $b(m)=(m+1)/2$, $c(m)=(m-1)/2$ if $m$ is odd. 
For $1\le j\le b(m)$ let $Y_j$ be the upper triangular unipotent subgroup
\begin{equation}\label{uni1}\notag
Y_j=\{y_j(p_1,\ldots,p_{a+j})=I_{2(n+kr_1)}+\sum_{i=1}^{a+j}p_ie'_{j,m+i+1}\},
\end{equation}
and for $1\le j\le c(m)$ let $Y_j'$ be the upper triangular unipotent subgroup
\begin{equation}\label{uni2}\notag
Y_j'=\{y_j'(p_1,\ldots,p_{a+c(m)-j+1},q)=I_{2(n+kr_1)}+\sum_{i=1}^{a+c(m)-j+1}p_ie'_{b(m)+j,m+i+1}+qe_{b(m)+j,a+c(m)+m-j+3}\}.
\end{equation}
We also define corresponding lower unipotent groups. For $1\le j\le c(m)$ let
\begin{equation}\notag
Z_j=\{z_j(p_1,\ldots,p_{a+j})=I_{2(n+kr_1)}+\sum_{i=m+2}^{a+j+m+1}p_ie'_{j,j+1}\},
\end{equation}
and for $1\le j\le b(m)$ let
\begin{multline}\notag
Z_j'=\{z_j'(p_1,\ldots,p_{a+b(m)-j+1})=\\ =I_{2(n+kr_1)}+\sum_{i=1}^{a+b(m)-j}p_ie'_{i,c(m)+j+1}+p_{a+b(m)-j+1}e_{a+b(m)-j+1,c(m)+j+1}\}.
\end{multline}
To prove the Lemma, we expand \eqref{cusp13} along the quotient $[Y_1]$. 
Doing so, we see that the integral \eqref{cusp13} is equal to
\begin{equation}\label{cusp14}
\sum_{\xi_i\in F}\,
\int\limits_{[Y_1]}
\int\limits_{[U_m^0]}
\theta_{2n+k(r-1)}^{(r)}(y_1(p_1,\ldots,p_{a+1})u)\,\psi_{U_m}(u)\,
\psi(\sum\xi_i p_i)\,dy_1\,du.
\end{equation}
Since the function $\theta_{2n+k(r-1)}^{(r)}$ is automorphic, for $\xi_j\in F$ we have 
$$\theta_{2n+k(r-1)}^{(r)}(z_1(-\xi_1,\ldots,-\xi_{a+1})h)=\theta_{2n+k(r-1)}^{(r)}(h).$$ Using this in the integral \eqref{cusp14} and then conjugating the matrix 
$z_1(-\xi_1,\ldots,-\xi_{a+1})$ to the right, we obtain (after a change of variables in $u$),
\begin{equation}\label{cusp15}
\sum_{\xi_i\in F}\,
\int\limits_{[Y_1]}
\int\limits_{[U_m^0]}
\theta_{2n+k(r-1)}^{(r)}(y_1(p_1,\ldots,p_{a+1})uz_1(-\xi_1,\ldots,-\xi_{a+1}))\,\psi_{U_m}(u)\,dy_1\,du.
\end{equation}
Thus if we prove that the inner integration in \eqref{cusp15} is zero for all choices of data, this will imply that the  integral \eqref{cusp13} is zero for all choices of data. 

Now we repeat this process with the inner integration of \eqref{cusp15}, this time using the groups $Y_2$ and $Z_2$. 
The process is visibly inductive, and we repeat it $b(m)$ times. More precisely, for all $1\le j\le b(m)$, we expand the corresponding integral along the group $Y_j$, 
and use the group $Z_j$ as above unless $j=b(m)$, $m$ odd, in which case we use the group $Z_1'$. 
Let $U_m^{00}$ denote the unipotent group generated by $U_m^0$ and by all $Y_j$ for $1\le j\le b(m)$. Then this shows that if the integral
\begin{equation}\label{cusp16}
\int\limits_{[U_m^{00}]}
\theta_{2n+k(r-1)}^{(r)}(u)\,\psi_{U_m}(u)\,du
\end{equation}
is zero for all choices of data, then the integral \eqref{cusp13} is zero for all choices of data. Here  $\psi_{U_m}$ is a character of $U_m^{00}$ obtained from $U_m^0$ by extending it trivially. 

Next, expand the integral \eqref{cusp16} along the unipotent abelian group 
$$\{ x'(p)=I_{2(n+kr_1)}+pe_{b(m)+1,a+c(m)+m+2}\}.$$  We claim that the contribution to the expansion from the non-constant terms is zero. 
Indeed, for these terms we obtain a Fourier coefficient which corresponds to the unipotent orbit $((2b(m)+2)1^{2(n+kr_1-b(m)-1)})$. 
Since $r\le m$, it follows that this unipotent orbit is not related to the orbit
${\mathcal O}_c(\Theta_{2n+k(r-1)}^{(r)})$.  
Hence by Proposition~\ref{theta01} these coefficients vanish. Thus the Lemma will follow once we prove that the integral
\begin{equation}\label{cusp17}
\int\limits_{[\mathbb{G}_a]}
\int\limits_{[U_m^{00}]}
\theta_{2n+k(r-1)}^{(r)}(x'(p)u)\,\psi_{U_m}(u)\,dp\,du
\end{equation}
is zero for all choices of data. To show this, observe that the group $\{x'(p)\}$ is the center of the group $Y_1'$. Thus we can expand the integral \eqref{cusp17} 
along the quotient $Y'_1/\{x'(p)\}$. Depending on the parity of $m$, we use $Z_1'$ or $Z_2'$ as above. Once again the argument is inductive, and after we carry it out with 
the groups $Y'_j$ for $1\le j\le c(m)$, we obtain the integral \eqref{cusp10}. By Lemma \ref{lem1}, since $r\le m$, this integral is zero for all choices of data. 
Lemma~\ref{lem2} follows.
\end{proof}

For the next Lemma, let $\alpha$ and $l$ be two positive integers which satisfy $2\le 2\alpha\le k$, and $(r-1)/2< l\le r-1$. 
We work with the unipotent group $U_{a,b,c}$ with $a=\alpha, b=l$ and $c=n+kr_1-l\alpha$. 
Consider the Fourier coefficient 
\begin{equation}\label{cusp18}
f(g)=
\int\limits_{[U_{\alpha,l,n+kr_1-l\alpha}^0]}
\theta_{2n+k(r-1)}^{(r)}(ug)\,\psi_{U_{\alpha,l,n+kr_1-l\alpha}}(u)\,du, \qquad g\in Sp_{2n+k(r-1)}^{(r)}(\A).
\end{equation}
For short we shall write $\psi_0$ for $\psi_{U_{\alpha,l,n+kr_1-l\alpha}}$. Because of the factorization in equation \eqref{mat2}, we 
may view \eqref{cusp18} as a function of $u'_{\alpha,n+kr_1-l\alpha}(0,Z)$ where $Z\in \Mat_\alpha^0({\A})$. We have
\begin{lemma}\label{lem3}
For $\alpha$ and $l$ in the range specified above, the Fourier coefficient \eqref{cusp18} is invariant under the adelic points of the group $\{u'_{\alpha,n+kr_1-l\alpha}(0,Z)\}$.  That is, 
for all $Z\in \Mat_\alpha^0({\A})$ we have $f(u'_{\alpha,n+kr_1-l\alpha}(0,Z)g)=f(g)$. 
\end{lemma}
\begin{proof}
Since the group of all matrices  $u'_{\alpha,n+kr_1-l\alpha}(0,Z)$ is an abelian subgroup of $Sp_{2n+k(r-1)}$, we can expand the function $f(g)$ along it. We obtain
\begin{equation}\label{cusp19}
f(g)=\sum_{\gamma\in \Mat_\alpha^0(F)}\quad
\int\limits_{[\Mat_\alpha^0]}
f(u'_{\alpha,n+kr_1-l\alpha}(0,Z)g)\,\psi_\gamma(Z)\,dZ,
\end{equation}
where $\gamma\to\psi_\gamma$ is an isomorphism of the abelian group $\Mat_\alpha^0(F)$ with its dual.
We need to prove that the nontrivial characters contribute zero to the expansion. The group $GL_\alpha(F)$, embedded in $Sp_{2n+k(r-1)}(F)$ 
by the map $\delta\mapsto\text{diag}(\delta,\delta,\ldots,\delta,I,\delta^*,\ldots,\delta^*)$, acts on the set $\{\gamma\}$. 
Here $\delta\in GL_\alpha(F)$ appears $l$ times. It is enough to consider representatives under this action. If $\gamma$ is not zero then there are two cases to consider. 
The first case is when $\psi_\gamma(x(p))$ is not zero for $x(p)=I_{2n+k(r-1)}+pe_{j_1, j_2}$ where $j_1=(l-1)\alpha+1$ and 
$j_2=2n+k(r-1)-(l-1)\alpha-1$. The second case is when $\psi_\gamma(x(p))$ is not zero for $x(p)=I_{2n+k(r-1)}+pe'_{j_1, j_2}$ where $j_1=(l-1)\alpha+1$ and 
$j_2=2n+k(r-1)-l\alpha+1$.

We start with the first case. Let $w_0$ be the Weyl group element $w_0=\text{diag}(w,I_{2(n+kr_1-l\alpha)},w^*)$ in $Sp_{2n+k(r-1)}(F)$. Here $w$ in $GL_{l\alpha}(F)$ is the matrix
whose only nonzero entries are $1$ in positions $(j,(j-1)\alpha +1)$ for $1\le j\le l$ and positions $(j_1,j_2)$ for  $j_1=l+(j-1)\alpha+a-j+1$ and 
$j_2=(j-1)\alpha +a+1$ with $1\le a\le \alpha-1$ and $1\le j\le l$.

Conjugating the integral in equation \eqref{cusp19} by $w_0$, we obtain the integral
\begin{equation}\label{cusp20}
\int\limits_{[X]}
\int\limits_{[U_l']}
\int\limits_{[\mathbb{G}_a]}
\theta_{2n+k(r-1)}^{(r)}(uu'_{1,n+kr_1-l}(0,m)x)\,\psi_{U_l}(u)\,\psi(\beta m)\,dm\,du\,dx
\end{equation}
as an inner integration to the integral \eqref{cusp19}. Here $\beta\in F^*$, and $U_l'$ is the subgroup of $U^0_{1,l+1,n+kr_1-l-1}$ defined as follows. 
An element $u=(u_{i,j})\in U_l'$ if $u_{i,j}=0$ for all $1\le i\le l-1$ and $l+1\le j\le l+i(\alpha-1)$. As for the group $X$, it consists of all matrices of the form
\begin{equation}\label{mat6}
x=\begin{pmatrix} I_l&&&&\\ y&I_{l(\alpha-1)}&&&\\ &&I&&\\ &&&I_{l(\alpha-1)}&\\
&&&y^*&I_l\end{pmatrix}
\end{equation}
where $y\in \Mat_{l(\alpha-1)\times l}$ satisfies the conditions $y_{i,j}=0$ for all $(i,j)$ such that
$1\le j\le l$ and $(j-1)(\alpha-1)+1\le i\le l(\alpha-1)$.
Thus, to prove that the contribution from this case to the sum in equation \eqref{cusp19} is zero, it is enough to prove that the integral \eqref{cusp20} is zero for all choices of data. 

We claim that, by means of root exchange, the vanishing of the integral \eqref{cusp20} follows from the vanishing of 
\begin{equation}\label{cusp21}
\int\limits_{[U_l]}
\int\limits_{[\mathbb{G}_a]}
\theta_{2n+k(r-1)}^{(r)}(uu'_{1,n+kr_1-l}(0,m))\,\psi_{U_l}(u)\psi(\beta m)\,dm\,du\,dx
\end{equation}
for all choices of data. To prove this claim, for $1\le j\le l-1$ we consider the two families of unipotent subgroups. Let $V_j$ denote the unipotent subgroup of $U_l$ defined by
\begin{equation}\label{cusp22}\notag
V_j=\{ x_i(p_i)=I_{2n+k(r-1)}+p_ie'_{j,l+i},\ \ :\ 1\le i\le j(\alpha-1)\},
\end{equation}
and let $X_j$ denote the unipotent subgroup of $X$ defined by matrices of the form \eqref{mat6} with all entries of $y$ equal to zero outside the $(j+1)$-st column.
We proceed inductively, starting with $j=1$. We expand the integral \eqref{cusp15} along the group $[V_j]$, 
and then we perform root exchange with the group $X_j$.
After the root exchange corresponding to $j=l-1$, we obtain the integral \eqref{cusp21} as an inner integration to integral \eqref{cusp20}. 

However,  the integral \eqref{cusp21} corresponds to the unipotent orbit $((2l+2)1^{2(n+kr_1-l-1)})$. Since $(r-1)/2< l$,  this unipotent orbit is not 
comparable with the unipotent orbit ${\mathcal O}_c(\Theta_{2(n+kr_1)}^{(r)})$ in Conjecture~\ref{conj1}. Hence Proposition~\ref{theta01} implies
that \eqref{cusp21} is zero for all choices of data. This completes the first case of representative in \eqref{cusp19}. 

Next we consider the second case. For this case we use a different Weyl group element in $Sp_{2(n+kr_1)}(F)$, which we denote by $w$. 
To define $w$,  we set $w_{i,(i-1)\alpha+1}=w_{l+i,2n+k(r-1)-(l-i+1)\alpha}=1$ for $1\le i\le l$. Then we extend it in an arbitrary way to a Weyl group element 
of $Sp_{2(n+kr_1)}(F)$. Conjugating the corresponding integral in the expansion \eqref{cusp19}, we obtain integral \eqref{cusp13} with $m=2l-2$ as an inner integration. 
Applying Lemma \ref{lem2}, the result follows.
\end{proof}

With this preparation, we now establish the tower property for this theta lift.  Fix $k\geq3$.
\begin{theorem}\label{th2}
Suppose that the representation $\sigma_{n,m}^{(r)}$
is zero (i.e.\ every function in this space is identically zero) for all $m$, $1\le m< k$, $m\equiv k \bmod 2$. Then $\sigma_{n,k}^{(r)}$ is a cuspidal representation.
\end{theorem}

\begin{proof}
To prove the cuspidality of the lift, we need to show that the constant terms of the representation $\sigma_{n,k}^{(r)}$ along any unipotent radical of a
standard maximal parabolic subgroup of $SO_k$ are zero for all choices of data. 
These are the subgroups $N_\alpha$, $1\le \alpha\le [k/2]$, consisting of all matrices in $SO_k$ of the form
\begin{equation}\notag 
\begin{pmatrix} I_\alpha&*&*\\ &I_\beta&*\\ &&I_\alpha\end{pmatrix},\qquad  \beta=k-2\alpha.
\end{equation}
Thus, with $f(h)$ given by \eqref{lift2}, we need to prove that the integral
\begin{equation}\label{cusp30}
\int\limits_{[N_\alpha]}
f(n_\alpha h)\,dn_\alpha
\end{equation}
is zero for all choices of data.  

We start by unfolding the theta series $\theta_{2nk}^{(2),\psi}$, which is expressed as a sum over $\xi\in F^{nk}$. We choose
the polarization $\xi=(\xi_1,\xi_2)$, where $\xi_1\in F^{2n\alpha}$ and $\xi_2\in F^{n\beta}$. Write $\xi_1=(\xi_{1,1},\xi_{1,2},\ldots,\xi_{1,\alpha})$ where $\xi_{1,i}\in F^{2n}$. 
The action of $Sp_{2n}$ is given by multiplication on the right on each $\xi_{1,i}$. Now the projection $l(u)$ that appears in \eqref{lift2} depends only on the 
values of  $l(u_{k,n}'(Y,Z))$ (for the notation, see \eqref{mat2}). 
Write
\begin{equation}\label{mat7}
Y=\begin{pmatrix}  y_\alpha\\ y_\beta\\ y_\alpha'\end{pmatrix},\qquad
Z=\begin{pmatrix} z_\alpha&*&*\\ z_1&z_\beta&*\\ z_2&*&*\end{pmatrix}\in \Mat_{k}^0.
\end{equation}
Here $y_\alpha,\ y_\alpha'\in \Mat_{\alpha\times 2n}$ and $y_\beta\in \Mat_{\beta\times 2n}$. 
Also, $z_\alpha\in \Mat_\alpha$, $z_\beta\in \Mat_\beta^0$, $z_1\in \Mat_{\beta\times\alpha}$ and $z_2\in \Mat_\alpha^0$. 
Using the formulas for the Weil representation $\omega_\psi$ (see for example \cite[Section 1, part 3]{G-R-S1}), we obtain
\begin{equation}\label{cusp31}
\theta_{2nk}^{(2),\psi}(l(u_{k,n}'(Y,Z))\iota_1(n_\alpha,g))=\sum_{\xi_1,\xi_2}\omega_\psi((\xi_1,0)l(u_{k,n}'(Y,Z))\iota_1(n_\alpha,g))\phi(0,\xi_2)
\end{equation}
Here  $\phi$ is a Schwartz function of ${\A}^{nk}$.  

From the definition of the homomorphism $l$, we have
$l(u_{k,n}'(Y_{\xi_1},0))=(\xi_1,0)$ where $Y_{\xi_1}=\left(\begin{smallmatrix} 0\\ \xi_1\end{smallmatrix}\right)$. 
In the integral \eqref{cusp30} the variables $y_\alpha'$, defined in \eqref{mat7}, are integrated over the quotient $\Mat_{\alpha\times 2n}(F)\backslash \Mat_{\alpha\times 2n}({\A})$. 
Hence, after conjugating the element  $u_{k,n}'(Y_{\xi_1},0)$ to the right we may combine 
summation with integration. It follows that to prove the vanishing of the integral \eqref{cusp30} for all choices of data, it suffices to prove the vanishing of the integral
\begin{multline}\label{cusp32}
\int\limits_{[Sp_{2n}]}
\int
\overline{\varphi^{(\kappa r)}(i(g))}
\theta_{2n(k-2\alpha)}^{(2),\psi}(l_0(u_{k,n}'(Y,Z))\iota_1(1,g))\\
\theta_{2n+k(r-1)}^{(r)}(u_{k,n}'(Y,Z)u\,\iota_2(n_\alpha,g))\,\psi_{U_{k,r_1,n}}(u)\,
\psi_\alpha(Z)\,dY\,dZ\,du\,dn_\alpha \,dg
\end{multline}
for all choices of data, where the notation is as follows.

 In the coordinates of \eqref{mat7}, $l_0$ is defined by
$l_0(u_{k,n}'(Y,Z))=(y_\beta,\text{tr}(z_\beta T_{k-2\alpha}))\in {\mathcal H}_{2n\beta+1}$ 
(where the rows of $y_\beta$ are listed with the bottom row first, similarly to the definition of the map $l$ in Section~\ref{basic}).
The variable 
$Y$ is integrated over $[Y_0]$, 
where $Y_0$ is the subgroup of $\Mat_{k\times 2n}$ given by all matrices $Y$ as in \eqref{mat7} with $y_\alpha'=0$, and  $Z$ is 
integrated over $[\Mat_{k\times k}^0]$. 
Also, $u$ is integrated over $[U_{k,r_1,n}^0]$, 
and $n_\alpha$ is integrated over $[N_\alpha]$. 
Using the coordinates of equation \eqref{mat7},  the character $\psi_\alpha$ is defined as $\psi_\alpha(Z)=\psi(\text{tr}(z_\alpha))$. 
Also, notice that (when $\beta>0$) the theta series appearing in  \eqref{cusp32} is defined on the double cover of $Sp_{2n(k-2\alpha)}({\A})$. This follows since after the above collapsing of summation
and integration we are left with the summation over $\xi_2\in F^{2n(k-2\alpha)}$. 
The rightmost argument of the theta series is $\iota_1(1,g)$ since  $\omega_\psi(\iota_1(n_\alpha,g))\phi(0,\xi_2)=
\omega_\psi((\iota_1(1,g))\phi(0,\xi_2)$.  If $k_0=\beta=0$ and $k=2\alpha$ then the theta function $\theta_{2n(k-2\alpha)}^{(2),\psi}$ in \eqref{cusp32} is omitted.

The next step is to define a certain Weyl group element $w\in Sp_{2n+k(r-1)}(F)$, which we will then use to conjugate the various groups. This Weyl element is of the form
\begin{equation}\label{wmat1}
w=\begin{pmatrix} w_1&&w_2\\ &I_{2n}&\\ w_3&&w_4\end{pmatrix}
\end{equation}
where $w_i\in \Mat_{kr_1}$. To specify it, it is enough to specify the matrices $w_1$ and $w_2$. 
These matrices will be block matrices whose only nonzero entries are the  identity  matrices $I_\alpha$ and $I_\beta$. 
To describe the location of each such identity block, it is enough to specify the location in each $w_j$, $j=1,2$, of its first $1$ on the diagonal. 
The matrix $w_1$ has an identity matrix $I_\alpha$ whose first $1$ is at position $(\alpha(i-1)+1,k(i-1)+1)$ for $1\le i\le r_1$.
Since $w$ has only one non-zero entry in each row, note that
this implies that the first $\alpha r_1$ rows of $w_2$ are all zeros. Next, the matrix $w_2$ has an identity matrix $I_\alpha$ whose first $1$ is at position
 $(\alpha r_1+\alpha(i-1)+1,k(i-1)+1)$ for $1\le i\le r_1$.
This then implies that the corresponding rows in the matrix $w_1$ are all zero. 
Finally, in $w_1$ there is a block of $I_\beta$ whose first 1 is at position $(\alpha(r-1)+\beta(i-1)+1,k(i-1)+\alpha+1)$ for $1\le i\le r_1$.

For example, when $r=7$, we have 
\begin{equation}\label{wmat2}\notag
w_1=\begin{pmatrix}
I_\alpha&0_\beta&0_\alpha&0_\alpha&0_\beta&0_\alpha&0_\alpha&0_\beta&0_\alpha\\ 0_\alpha&&&I_\alpha&&&&&\\ 0_\alpha&&&&&&I_\alpha&&\\ 0_\alpha&&&&&&&&\\ 
0_\alpha&&&&&&&&\\ 0_\alpha&&&&&&&&\\ 0_\beta&I_\beta&&&&&&&\\ 
0_\beta&&&&I_\beta&&&&\\ 0_\beta&&&&&&&I_\beta&
\end{pmatrix},\  
w_2=\begin{pmatrix}
0_\alpha&0_\beta&0_\alpha&0_\alpha&0_\beta&0_\alpha&0_\alpha&0_\beta&0_\alpha\\ 0_\alpha&&&&&&&&\\ 0_\alpha&&&&&&&&\\ I_\alpha&&&&&&&&\\ 
0_\alpha&&&I_\alpha&&&&&\\ 0_\alpha&&&&&&I_\alpha&&\\ 0_\beta&&&&&&&&\\ 
0_\beta&&&&&&&&\\ 0_\beta&&&&&&&&
\end{pmatrix}
\end{equation}
where all the blank entries are zero.

Before conjugating by $w$, we perform a certain root exchange. To do that, let $L_{\alpha,\beta}$ denote the unipotent subgroup of $GL_k$ consisting of all matrices of the form
\begin{equation}\label{em0}
l=\begin{pmatrix} I_\alpha&a&b\\ &I_\beta&c\\ &&I_\alpha\end{pmatrix}
\end{equation}
and let $L_0$ denote any unipotent subgroup of $L_{\alpha,\beta}$ such $L_{\alpha,\beta}=L_0 N_\alpha$. 
For example, one may choose the group of all matrices $l$ as above such that $c=0$ and such that $bJ_\alpha$ 
is strictly upper triangular. Consider the direct sum $L_{\alpha,\beta}\oplus\ldots\oplus L_{\alpha,\beta}\oplus L_0$ where $L_{\alpha,\beta}$ appears $r_1-1$ times. 
We embed this group inside $Sp_{2n+k(r-1)}$ as 
\begin{equation}\label{em1}
\text{diag}(l_1,l_2,\ldots,l_{r_1-1},l_0,I_{2n},l_0^*,l_{r_1-1}^*,\ldots,l_1^*).
\end{equation}
We will also need to consider the subgroup of $\Mat_k$ which consists of all matrices of the form
\begin{equation}\notag 
l^-=\begin{pmatrix} 0_\alpha&&\\ a_1&0_\beta&\\ b_1&c_1&0_\alpha\end{pmatrix}
\end{equation}
We denote this group by $L^-$. 

Now we carry out root exchange with the embedded copies of the groups $L_{\alpha,\beta}$ and $L^-$. We begin with $L_{\alpha,\beta}$ is embedded in the first component of 
$L_{\alpha,\beta}\oplus\ldots\oplus L_{\alpha,\beta}\oplus L_0$
and then inside $Sp_{2n+k(r-1)}$ as in \eqref{em1}. 
Since $L_{\alpha,\beta}$ is not abelian, this root exchange needs to be carried out in stages, as follows. 
First, we exchange the unipotent elements which are in the center of $L_{\alpha,\beta}$, i.e.\ the (abelian) group of all 
matrices $l$ in \eqref{em0} such that $a=c=0$, with the group of all matrices  $u_{k,r_1,n}^1(l^-)$  (see \eqref{mat4}) such that $l^-\in L^-$ with $a_1=c_1=0$. 
After doing this, we proceed with the abelian group consisting of 
all matrices $l\in L_{\alpha,\beta}$ with $b=c=0$, and then with the abelian group of all matrices with $a=b=0$.  
Next, we exchange $L_{\alpha,\beta}$ embedded in the second component of $L_{\alpha,\beta}\oplus\ldots\oplus L_{\alpha,\beta}\oplus L_0$
and then inside $Sp_{2n+k(r-1)}$ as in \eqref{em1}.
For this group we use the copy of $L^-$ embedded in $Sp_{2n+k(r-1)}$ as $l^-\mapsto u_{k,r_1,n}^2(l^-)$. We continue this process for all $i$ with $1\le i\le r_1-1$, 
exchanging the $i$-th copy of $L_{\alpha,\beta}$ inside $L_{\alpha,\beta}\oplus\ldots\oplus L_{\alpha,\beta}\oplus L_0$
with a subgroup of $u_{k,r_1,n}^i(L^-)$.  Then, we perform root exchange corresponding to $l^-\in L^-$, embedded in  $U_{k,r_1,n}^0$ as all matrices 
$l^-\mapsto u_{k,r_1,n}^1(l^-)$
We exchange this group with the group of all matrices $u_{k,n}'(0,Z)$ where
\begin{equation}\label{em3}
Z=\begin{pmatrix} &&\\ z_1&&\\ z_2&z_1^*&\end{pmatrix}.
\end{equation}

After performing these root exchanges, we conjugate by the Weyl element $w$ defined in \eqref{wmat1}. 
Thus, to prove that integral \eqref{cusp32} is zero for all choice of data, we conclude that it is enough to prove that the integral
\begin{multline}\label{cusp33}
\int
\overline{\varphi^{(\kappa r)}(i(g))}\,
\theta_{2n(k-2\alpha)}^{(2),\psi}(l_0(u)\iota_1(1,g))\,\psi_{U_{\beta,r_1,n}}(u)\,
\psi_{V_{\alpha,r-1}}(v)\\ \theta_{2n+k(r-1)}^{(r)}
\left (\begin{pmatrix} I_{\alpha(r-1)}&C&D\\ &I&C^*\\ &&I_{\alpha(r-1)}\end{pmatrix}\begin{pmatrix} v&&\\ &u&\\ &&v^*\end{pmatrix}
\begin{pmatrix}  I_{\alpha(r-1)}&&\\ A&I&\\ B&A^*&I_{\alpha(r-1)}\end{pmatrix}   \iota_2(1,g)\right )\,d(\mydots)
\end{multline}
is zero for all choices of data.  Here $I$ denotes the identity matrix of size $2n+(k-2\alpha)(r-1)$.
The description of the variables in this integral and the domain of integration is elaborate, and we give it now.

First, $g$ is integrated as in the integral \eqref{cusp32}. The variable $u$ is integrated over the quotient
$[U_{\beta,r_1,n}]$, 
where $U_{\beta,r_1,n}\subseteq Sp_{2n+\beta(r-1)}$  is embedded inside $Sp_{2n+k(r-1)}$ by the map $u\to \text{diag}(I_{\alpha(r-1)},u,I_{\alpha(r-1)})$. 
Also, $V_{\alpha,r-1}$ denotes the subgroup of 
$U_{\alpha,r-1,n+\beta r_1}$ generated by all matrices of the form $\prod_{i=1}^{r-2}u_{\alpha,r-1,n+\beta r_1}^i(X_i)$ with $X_i\in \Mat_\alpha$;  $V_{\alpha,r-1}$ is isomorphic
to a unipotent subgroup of $GL_{\alpha(r-1)}$. The character
 $\psi_{V_{\alpha,r-1}}$ is the restriction of $\psi_{U_{\alpha,r-1,n+\beta r_1}}$ to $V_{\alpha,r-1}$. The variable $v$ is integrated over $[V_{\alpha,r-1}]$. 

Next we define the regions over which the variables $A,B,C$ and $D$ (each matrices of a certain size) are integrated. 
These are given by considering them as block matrices and imposing various conditions.
We start with the variable $D$. Consider the subgroup $D_0\subset \Mat_{\alpha(r-1)}^0$ consisting of block matrices with blocks
of size $\alpha$ such that each $(i,j)$ block with $j<i$ is the zero matrix $0_\alpha$. Thus, for example if $r-1=4$, then $D_0$ consists of all matrices of the form
\begin{equation}\notag 
D=\begin{pmatrix} X_1&X_2&X_3&X_4\\ 0_\alpha&X_5&X_6&X_3^*\\ 
0_\alpha&0_\alpha&X_5^*&X_2^*\\ 0_\alpha&0_\alpha&0_\alpha&X_1^*\end{pmatrix}
\ \ \ X_1,X_2,X_3,X_5\in \Mat_\alpha;\ \ X_4,X_6\in \Mat_\alpha^0.
\end{equation}
With these notations $D$ is integrated over $[D_0]$. 
Similarly, let $B_0\subset \Mat_{\alpha(r-1)}^0$ consist of block matrices with blocks of size $\alpha$ such that each $(i,j)$ block with $j\le i+1$ is the zero matrix $0_\alpha$. 
For example if $r-1=5$, then $B_0$ consists of all matrices of the form
\begin{equation}\notag 
B=\begin{pmatrix} 0_\alpha&0_\alpha&X_1&X_2&X_3\\ 0_\alpha&0_\alpha&0_\alpha&X_4&X_2^*\\ 
0_\alpha&0_\alpha&0_\alpha&0_\alpha&X_1^*\\ 0_\alpha&0_\alpha&0_\alpha&0_\alpha&0_\alpha \\ 0_\alpha&0_\alpha&0_\alpha&0_\alpha&0_\alpha\end{pmatrix}
\ \ \ X_1,X_2\in \Mat_\alpha;\ \ X_3,X_4\in \Mat_\alpha^0.
\end{equation}
Then $B$ is integrated over $[B_0]$.  

As for $C$, define a subgroup $C_0\subset \Mat_{\alpha(r-1)\times (2n+2\beta r_1)}$ as follows.  Write 
\begin{equation}\notag 
C=\begin{pmatrix} C_1&C_2&C_3\\ 0&0&C_4\end{pmatrix},
\end{equation}
with $C_2\in \Mat_{\alpha r_1\times 2n}$ and $C_1,C_3,C_4\in \Mat_{\alpha r_1\times \beta r_1}$. 
Then $C_0$ is the subgroup of such matrices such that  $C_1$ and $C_4$ may be written as block matrices, with blocks of size $\alpha\times\beta$, such that 
each $(i,j)$ block with $j<i$ is $0_{\alpha\times\beta}$. For example, when $r_1=3$, then both $C_1$ and $C_4$ are matrices of the form  
\begin{equation}\notag 
\begin{pmatrix} X_1&X_2&X_3\\ 0&X_4&X_5\\ 0&0&X_6\end{pmatrix}\ \ \ \ 
X_i\in \Mat_{\alpha\times\beta},
\end{equation}
where the zeroes indicate the zero matrices of the corresponding sizes.  
 The variable $C$ is integrated over $[C_0]$. 

Finally, let $A_0$ denote the subgroup of $\Mat_{(2n+\beta(r-1))\times \alpha(r-1)}$ consisting of matrices of the form
\begin{equation}\notag 
A=\begin{pmatrix} 0&A_1&A_2\\ 0&0&A_3\\ 0&0&A_4\\ 0&0&0&\end{pmatrix}\ \ \ \ 
\begin{pmatrix} A_2\\ A_3\end{pmatrix}\in \Mat_{(2n+\beta r_1)\times \alpha(r_1-1)},
\end{equation}
where the first column has width $2\alpha$, the last row has height $\beta$, and
$A_1,A_4\in\Mat_{\beta(r_1-1)\times \alpha(r_1-1)}$ satisfy the following conditions. 
First viewing $A_1$ as a block matrix with blocks of size $\alpha\times \beta$, all $(i,j)$ blocks with $j<i$ are zero matrices. 
Second, viewing $A_4$ similarly, all $(i,j)$ blocks with $j\le i$ are the zero matrix. For example, for $r_1-1=4$ we have 
\begin{equation}\label{block6}
A_1=\begin{pmatrix} X_1&X_2&X_3&X_4\\ 0&X_5&X_6&X_7\\ 
0&0&X_8&X_9\\ 0&0&0&X_{10}\end{pmatrix},\ \ 
A_4=\begin{pmatrix} 0&Y_1&Y_2&Y_3\\ 0&0&Y_4&Y_5\\ 
0&0&0&Y_6\\ 0&0&0&0\end{pmatrix}
\ \ \ X_i, Y_j\in \Mat_{\beta\times \alpha}.
\end{equation}
Then $A$ is integrated over $[A_0]$. 

Our goal is to prove that the integral \eqref{cusp33} is zero for all choices of data. To do that we start with a sequence of root exchanges, and a repeated application of 
Lemma~\ref{lem3}.
Recall that in terms of blocks of height $\alpha$, the matrices $C$ and $D$ each have $r-1$ rows. 
The integral runs over the first row of blocks of each, while for the second row, we have variables of integration 
everywhere except for the blocks in position $(2,1)$, a block which is $0_{\alpha\times \beta}$ for $C$ and $0_{\alpha}$ for $D$, and so on for the remaining rows.

We first perform the root exchange using 
$(1,3)$ block in $B$ and then using the $(1,3)$ block of $A$. Note that 
 there is a compatibility in the block sizes with the $(2,1)$ blocks of the matrices $D$ and $C$, resp. 
 Indeed, the $(1,3)$ block of $B$  is of size $\alpha\times \alpha$ which is exactly the size of the block in the $(2,1)$ position  of $D$. 
 Similarly, the $(1,3)$ block of $A$ is of size $\beta\times \alpha$ which is exactly what is needed for 
 root exchange with the block matrix at the $(2,1)$ position of the $C$, which is of size $\alpha\times \beta$. We 
 continue this process with the third row of the block matrices $C$ and $D$. For these two matrices the $(3,1)$ and $(3,2)$ blocks are each zero. 
 We integrate over the $(1,4)$ and $(2,4)$ block matrices in $A$ and $B$, which allows
 us to perform root exchange. Repeating this process with all rows up to and including the $r_1$-th row, we obtain the integral \eqref{cusp18} with $l=(r+1)/2$ as inner integration. 
 Applying Lemma \ref{lem3}, we get the invariance of this integral along the adelic points of the  unipotent subgroup $u'_{\alpha,n+kr_1-l\alpha}(0,Z)$ with $l=(r+1)/2$. 
 Notice that this subgroup is realized as the block matrices of size $\alpha$ which are in position $((r+1)/2,(r-1)/2)$ in $D$. 
 With this invariance, we can then proceed with root exchange of the blocks of $C$ and $D$ with the corresponding blocks of  $A$ and $B$. 
 
 We conclude from this root exchange that the integral \eqref{cusp33} is zero for all choices of data if the integral 
\begin{multline}\label{cusp34}
\int
\overline{\varphi^{(\kappa r)}(i(g)}
\theta_{2n(k-2\alpha)}^{(2),\psi}(l_0(u)\iota_1(1,g))\,\psi_{U_{\beta,r_1,n}}(u)\,
\psi_{U_{\alpha,r-1,n+\beta r_1}}(u_1)\\
\theta_{2n+k(r-1)}^{(r)}(u'_{\alpha,n+\beta r_1}(Y_{3,r_1},Z)u_1
\begin{pmatrix} I_{\alpha(r-1)}&&\\ &u&\\ &&I_{\alpha(r-1)}\end{pmatrix}
\iota_2(1,g))\,du_1\,du\,dY_{3,r_1}\,dZ\,dg
\end{multline}
is zero for all choices of data.
Here $u$ and $g$ are integrated as in \eqref{cusp33}, $u_1$ is integrated over the quotient $[U_{\alpha,r-1,n+\beta r_1}]$, and 
$Z$ is integrated over $[\Mat_{\alpha}^0]$.  Also, 
$Y_{3,r_1}$ is integrated over $[\Mat_{\alpha\times\beta}]$, 
where the notation $Y_{3,r_1}$ is explained as follows.
Recall that $u'_{\alpha,n+\beta r_1}(Y,Z)$ was defined in \eqref{mat1}. Here we have $Y\in \Mat_{\alpha\times 2(n+\beta r_1)}$.
Write $Y=\begin{pmatrix} Y_1&Y_2&Y_3\end{pmatrix}$, where $Y_1, Y_3\in \Mat_{\alpha\times \beta r_1}$ and $Y_2\in \Mat_{\alpha\times 2n}$, and let
$$Y_3=\begin{pmatrix} Y_{3,1}&Y_{3,2}&\ldots &Y_{3,r_1}\end{pmatrix}\ \ \ \ Y_{3,i}\in \Mat_{\alpha\times \beta};\ 1\le i\le r_1$$ 
Then in \eqref{cusp34} instead of writing  $Y=\begin{pmatrix} 0_{\alpha\times 2(n+\beta(r_1-1))}&Y_{3,r_1}\end{pmatrix}$ we write  $Y_{3,r_1}$ for short.

Next we expand the integrand in \eqref{cusp34} along the group $Y_3=\begin{pmatrix} 0&0&\ldots&0&Y_{3,r_1-1}&0\end{pmatrix}$ where 
$Y_{3,r_1-1}\in [\Mat_{\alpha\times\beta}]$.
The nontrivial characters in this expansion correspond to matrices of size $\beta\times\alpha$ of fixed positive rank. We will prove that all nontrivial terms are zero. 
Consider the group $GL_\alpha(F)\times GL_\beta(F)$ embedded in $Sp_{2n+k(r-1)}(F)$ by 
$$(h_1,h_2)\to \text{diag}(h_1,\ldots,h_1,h_2\ldots,h_2,I_{2n},h_2^*,\ldots,h_2^*,h_1^*,\ldots,h_1^*),$$
where $h_1\in GL_\alpha(F)$ appears $r-1$ times and $h_2\in GL_\beta(F)$ appears $r_1$ times. 
This group may be used to collect terms of this expansion in the usual way.
As representatives with respect to it, we choose the characters
$$\psi_A(Y_{3,r_1-1})=\psi(\text{tr}(Y_{3,r_1-1}A))\ \ \ \ A=\begin{pmatrix} I_m&\\ &&0\end{pmatrix}\in \Mat_{\beta\times\alpha}.$$
The contribution to the integral \eqref{cusp34} from a nontrivial orbit is
\begin{multline}\label{cusp35}
\int
\overline{\varphi^{(\kappa r)}(i(g))}\,
\theta_{2n(k-2\alpha)}^{(2),\psi}(l_0(u)\iota_1(1,g))\,\psi_{U_{\beta,r_1,n}}(u)\,
\psi_{U_{\alpha,r-1,n+\beta r_1}}(u_1)\,\psi_A(Y_{3,r_1-1})\\
\theta_{2n+k(r-1)}^{(r)}(u'_{\alpha,n+\beta r_1}(Y_{3,r_1-1},Y_{3,r_1},Z)u_1
\begin{pmatrix} I_{\alpha(r-1)}&&\\ &u&\\ &&I_{\alpha(r-1)}\end{pmatrix}
\iota_2(1,g))\,du_1\,du\,dY_{3,r_1-1}\,dZ\,dg.
\end{multline}

To prove that this integral is zero, let $w$ be a Weyl element of $Sp_{2n+k(r-1)}(F)$ which has entry 1 at positions
$(i,\alpha(i-1)+1)$, $1\le i\le r-1$, and at positions
 $(r,(\alpha+\beta)(r-1)+2(n-\beta)+1)$ and  $(r+1,(\alpha+\beta)(r-1)+2n-\beta+1)$.  (We do not specify it in rows $r+2$ to $n+kr_1$.)
Using the automorphicity of $\theta_{2n+k(r-1)}^{(r)}$ we can conjugate the argument of this function by $w$. 
After doing so, we obtain the integral \eqref{cusp13} with $m=r$ as inner integration, and then from Lemma \ref{lem2} it follows that  \eqref{cusp35} is zero for all choices of data. 
We conclude that the only nonzero contribution  to the integral \eqref{cusp34} from the expansion along  $Y_{3,r_1-1}$ is from the constant term. 

Continuing this process, next with $Y_{3,r_1-2}$, we obtain by induction that \eqref{cusp34} is equal to the integral
\begin{multline}\label{cusp36}
\int
\overline{\varphi^{(\kappa r)}(i(g))}
\theta_{2n(k-2\alpha)}^{(2),\psi}(l_0(u)\iota_1(1,g))\,\psi_{U_{\beta,r_1,n}}(u)\,
\psi_{U_{\alpha,r-1,n+\beta r_1}}(u_1)\\
\theta_{2n+k(r-1)}^{(r)}(u'_{\alpha,n+\beta r_1}(Y_{3},Z)u_1
\begin{pmatrix} I_{\alpha(r-1)}&&\\ &u&\\ &&I_{\alpha(r-1)}\end{pmatrix}
\iota_2(1,g))\,du_1\,du\,dY_{3}\,dZ\,dg,
\end{multline}
where $Y_3$ is integrated over $[\Mat_{\alpha\times \beta r_1}]$. 

Next, we expand \eqref{cusp36} along $Y_2$ (as defined above, following \eqref{cusp34}) over the quotient $[\Mat_{\alpha\times 2n}]$.
We will show that all nontrivial characters contribute zero to this expansion. Note
that the group $GL_\alpha(F)\times Sp_{2n}(F)$, embedded in $Sp_{2n+k(r-1)}(F)$ by the map
$$(h,g)\to \text{diag}(h,\ldots,h,I_{\beta r_1},g,I_{\beta r_1},h^*,\ldots,h^*)\ \ \ h\in GL_\alpha(F);\ \ g\in Sp_{2n}(F),$$
acts, so we may consider the characters modulo this action.
There are two types of representatives for the  nontrivial orbits, as follows.

The first is given by characters of the form
$$\psi_A(Y_2)=\psi(\text{tr}(Y_2A))\ \ \ \ \ A=\begin{pmatrix} 1&0&*\\ 0_{2(n-1)\times 1}& 0_{2(n-1)\times 1}&*\\ 0&1&*\end{pmatrix}\in 
\Mat_{2n\times\alpha}(F).$$
These appear in the expansion only if $\alpha\ge 2$. Let $w$ denote a Weyl element of $Sp_{2n+k(r-1)}(F)$ with entry 1 in
positions $(i,\alpha(i-1)+1)$ and $(r+i,(\alpha+\beta)(r-1)+2n+i\alpha -1)$ for $1\le i\le r-1$, and in position $(r,\alpha(r-1)+\beta r_1+1)$. Then, arguing 
as above with \eqref{cusp35}, by conjugating by this $w$, we obtain the integral \eqref{cusp13} with $m=2r-2$ as inner integration. 
Notice that since $\alpha\ge 2$, then $k\ge 4$, and we do have $r\le 2r-2\le n+kr_1-1$. Applying Lemma \ref{lem2}, we see that the 
contribution to the expansion  from these representatives is zero.

The second type of representative is given by the characters
\begin{equation}\label{cha1}
\psi_A(Y_2)=\psi(\text{tr}(Y_2A))\ \ \ \ \ A=\begin{pmatrix} I_a&\\ &\end{pmatrix}\in \Mat_{2n\times\alpha}(F);\ \ \ \ 1\le a\le n,\alpha.
\end{equation}
(If $a>n$ then the orbit is represented by a character already considered above.) 
The stabilizer inside $Sp_{2n}$ contains the unipotent radical of the maximal parabolic subgroup of $Sp_{2n}$ whose 
Levi part is $GL_a\times Sp_{2(n-a)}$. We denote this unipotent group by $R_a$. It is embedded inside $Sp_{2n+k(r-1)}$ as all matrices of the form
\begin{equation}\label{matt1}
\begin{pmatrix} I_{kr_1}&&&&\\ &I_a&B&C&\\ &&I_{2(n-a)}&B^*&\\ &&&I_a&\\
&&&&I_{kr_1}\end{pmatrix}\ \ B\in \Mat_{a\times 2(n-a)};\ \ \ C\in \Mat_{a\times a}^0.
\end{equation}
The claim is that, as a function of $g\in Sp_{2n}({\A})$, the integral
\begin{multline}\label{cusp37}
\int\theta_{2n(k-2\alpha)}^{(2),\psi}(l_0(u)\iota_1(1,g))\,\psi_{U_{\beta,r_1,n}}(u)\,
\psi_{U_{\alpha,r-1,n+\beta r_1}}(u_1)\,\psi_A(Y)\\
\theta_{2n+k(r-1)}^{(r)}(u'_{\alpha,n+\beta r_1}(Y,Z)u_1
\begin{pmatrix} I_{\alpha(r-1)}&&\\ &u&\\ &&I_{\alpha(r-1)}\end{pmatrix}
\iota_2(1,g))\,du_1\,du\,dY\,dZ
\end{multline}
is left invariant under $R_a({\A})$. Here $Y$ runs over all matrices 
of the form $\begin{pmatrix} 0&Y_2&Y_3\end{pmatrix}$ with $Y_2$ integrated over $[\Mat_{\alpha\times 2n}]$
and $Y_3$ 
integrated as in \eqref{cusp36}. The character $\psi_A(Y)$ is the trivial extension of $\psi_A(Y_2)$ defined above. 

To prove the claim we first unfold the theta series $\theta_{2n(k-2\alpha)}^{(2),\psi}$, similarly to \eqref{cusp31}. After collapsing summation and integration, we obtain
\begin{multline}\label{cusp38}
\int\sum_{\xi\in F^m}\omega_\psi(l_0(u)\iota_1(1,g))\phi(0,\xi)\,\psi_{U_{\beta,r_1,n}}(u)\,
\psi_{U_{\alpha,r-1,n+\beta r_1}}(u_1)\,\psi_A(Y)\\
\theta_{2n+k(r-1)}^{(r)}(u'_{\alpha,n+\beta r_1}(Y,Z)u_1
\begin{pmatrix} I_{\alpha(r-1)}&&\\ &u&\\ &&I_{\alpha(r-1)}\end{pmatrix}
\iota_2(1,g))\,du_1\,du\,dY\,dZ.
\end{multline}
Here $m=0$ if $k$ is even, and $m=n-a$ if $k$ is odd. The integration in $u$ is over the quotient 
$U_{\beta,r_1,n}'(F)\backslash U_{\beta,r_1,n}({\A})$ 
where $U_{\beta,r_1,n}'$ is a certain subgroup of $U_{\beta,r_1,n}$ (whose definition we omit as it plays no role in the sequel). 
It follows from the action of the Weil representation that the function $\sum_{\xi\in F^m}\omega_\psi(\iota_1(1,g))\phi(0,\xi)$ is left invariant under $R_a(\A)$. 
After moving the matrix $r_a\in R_a({\A})$ to the right and changing variables, it is enough to prove that the function of $g$ 
\begin{equation}\label{cusp39}
\int \theta_{2n+k(r-1)}^{(r)}(u'_{\alpha,n+\beta r_1}(Y,Z)u_1
\iota_2(1,g))\,\psi_{U_{\alpha,r-1,n+\beta r_1}}(u_1)\,\psi_A(Y)\,du_1\,dY\,dZ
\end{equation}
is left invariant under $r_a\in R_a({\A})$.  Here, the variables $u_1, Y$ and $Z$ are integrated as in \eqref{cusp38}. 

To prove this invariance, we argue as in the proof of Lemma \ref{lem3}. First expand \eqref{cusp39} along the abelian group $\Mat_{a\times a}^0$, embedded in $Sp_{2n+k(r-1)}$ as the
group of all matrices of the form \eqref{matt1} with $B=0$. Next observe that only the trivial character contributes. Indeed, arguing similarly to the proof of Lemma \ref{lem3}, 
we obtain as inner integration either a Fourier coefficient which corresponds to a unipotent orbit which is not related to ${\mathcal O}(\Theta_{2n+k(r-1)}^{(r)})$, or 
an integral of the form of \eqref{cusp13} with $r\le m$. Then  
expand along the group $\Mat_{a\times 2(n-a)}$ embedded inside $Sp_{2n+k(r-1)}$ as all matrices of the form \eqref{matt1} with $C=0$. Similar arguments imply that only the 
constant term gives a non-zero contribution. From this it follows that integral \eqref{cusp39}, and hence integral \eqref{cusp37}, is left invariant under  $r_a\in R_a({\A})$. 
Using this invariance property in the expansion of
\eqref{cusp36} along $Y_2$, when we consider the contribution from the characters \eqref{cha1}, we obtain (after measure factorization) the integral
$$
\int\limits_{[R_a]}
\overline{\varphi^{(\kappa r)}(i(r_a)i(g))}\,dr_a$$
as inner integration. By cuspidality this integral is zero for all choices of data.

We deduce that the integral \eqref{cusp36} is equal to
\begin{multline}\label{cusp40}
\int
\overline{\varphi^{(\kappa r)}(i(g))}
\theta_{2n(k-2\alpha)}^{(2),\psi}(l_0(u)\iota_1(1,g))\,\psi_{U_{\beta,r_1,n}}(u)\,
\psi_{U_{\alpha,r-1,n+\beta r_1}}(u_1)\\
\theta_{2n+k(r-1)}^{(r)}(u'_{\alpha,n+\beta r_1}(Y_2,Y_{3},Z)u_1
\begin{pmatrix} I_{\alpha(r-1)}&&\\ &u&\\ &&I_{\alpha(r-1)}\end{pmatrix}
\iota_2(1,g))\,du_1\,du\,dY_2\,dY_{3}\,dZ\,dg.
\end{multline}
The final expansion we need to consider is the expansion of \eqref{cusp40} along the quotient 
$Y_1\in [\Mat_{\alpha\times \beta r_1}]$. Here $Y_1$ is embedded inside $Sp_{2n+k(r-1)}$ as the group of all matrices 
$u'_{\alpha,n+\beta r_1}(Y,0)$ with $Y=\begin{pmatrix} Y_1&0&0\end{pmatrix}$ (with $Y$ as described following \eqref{cusp34}). 
Similarly to the expansion of \eqref{cusp34} along the subgroup $Y_3$, we see that all nontrivial characters for $Y_1$ give zero.
Thus \eqref{cusp40} is equal to 
\begin{multline}\label{cusp41}
\int
\overline{\varphi^{(\kappa r)}(i(g))}
\theta_{2n(k-2\alpha)}^{(2),\psi}(l_0(u)\iota_1(1,g))\,\psi_{U_{\beta,r_1,n}}(u)\,
\psi_{U_{\alpha,r-1,n+\beta r_1}}(u_1)\\
\theta_{2n+k(r-1)}^{(r)}(u'_{\alpha,n+\beta r_1}(Y,Z)u_1
\begin{pmatrix} I_{\alpha(r-1)}&&\\ &u&\\ &&I_{\alpha(r-1)}\end{pmatrix}
\iota_2(1,g))\,du_1\,du\,dY\,dZ\,dg
\end{multline}
where $Y$ is now integrated over 
$[\Mat_{\alpha\times 2n+\beta r_1}]$.

To complete the proof of the Theorem we need to prove that \eqref{cusp41} is zero for all choices of data. 
Recall that $u_1$ is integrated over $[U_{\alpha,r-1,n+\beta r_1}^0]$. 
Combining this integration with the integration over $u'_{\alpha,n+\beta r_1}(Y,Z)$, we obtain as inner integration the  
constant term of the function $\theta_{2n+k(r-1)}^{(r)}$ along the unipotent radical of the maximal parabolic subgroup of $Sp_{2n+k(r-1)}$ whose 
Levi part is $GL_{\alpha(r-1)}\times Sp_{2n+\beta(r-1)}$. This means that we can use Proposition 1 in \cite{F-G2} to obtain as inner integration a function 
which is realized in the space of the representation $\sigma_{n,\beta}^{(r)}=\sigma_{n,k-2\alpha}^{(r)}$. Indeed, this follows since $u$ in 
 \eqref{cusp41} is integrated over $[U_{\beta,r_1,n}]$. 
By our assumption this representation is zero. This completes the proof of the Theorem.
\end{proof}

In fact it follows from the above proof that we have the following Corollary.
\begin{corollary}\label{cor1}
Suppose that the representation $\sigma_{n,k}^{(r)}$ is zero. Then, for all $1\le m\le [k/2]$,  the representation $\sigma_{n,k-2m}^{(r)}$ is zero.
\end{corollary}
\begin{proof}
Suppose that $\sigma_{n,k}^{(r)}$ is zero. Then for $1\le \alpha\le k_1$ the constant term of this representation along the unipotent subgroup $N_\alpha$ of $SO_k$ is zero for all choices of data. 
(The group $N_\alpha$ was defined before \eqref{cusp30}.) In Theorem \ref{th2} we proved that if the representation $\sigma_{n,k-2\alpha}^{(r)}$ is zero, then the constant term given by 
\eqref{cusp30} is zero for all choices of data. In fact the converse is also true.  The key point is that the process of root exchange shows that one expression vanishes
for all choices of data if and only if another does
(see \cite{G-R-S4}, Corollary 7.1).   Accordingly, the steps of the above proof may be reversed.  We omit the detail.
\end{proof}

\section{The Unramified Correspondence}\label{the-unramified-corr}

Our goal in this section is to establish that the theta correspondence developed here is functorial on unramified principal series.  In this section we let $F$ be a nonarchimedean
local field containing $\mu_r$, and to simplify the notation we write a group or covering group for its $F$-points (so we write $Sp_{2n}$ instead of $Sp_{2n}(F)$, etc.).  
Recall that given an unramified character of the torus of a
symplectic or orthogonal group, one may define its associated principal series by parabolic induction from the Borel subgroup.  
One may do this for covering groups as well, but this requires a two-step process.
Let $G$ be the ($F$-points) of one of the groups considered above and $B=TU$ its standard Borel subgroup, with $T$ be its maximal split torus.  
Let $\widetilde{G}$ be one of the covering groups under consideration here, and if $H$ is any subgroup of $G$ let $\widetilde{H}$ denote its inverse image in $\widetilde{G}$.
Then $\widetilde{T}$ is generally not abelian, 
but it is a two-step nilpotent group. If $\chi$ is a genuine character of the center $Z(\widetilde{T})$ of $\widetilde{T}$, then
genuine principal series representations are constructed by first extending $\chi$ to a character of a maximal abelian group $A$ of $\widetilde{T}$ containing $Z(\widetilde{T})$, next
inducing this extension from $A$ to $\widetilde{T}$, then extending trivially on $N$ to obtain a representation of $\widetilde{B}$, 
and finally taking the normalized induction from $\widetilde{B}$ to $\widetilde{G}$.  By an analogue of the Stone-Von Neumann Theorem, this representation
is determined by its central character $\chi$, and we write the representation $\pi(\chi)$.  
For more details see for example \cite{McN} (analogously to the general linear group, in \cite{K-P}, Section I.1) or \cite{Gao2}.
For characters in general position these induced representations are irreducible.

Let $\chi$ be the unramified character of the maximal torus $T$ of $Sp_{2n}$ given by
$$\chi\left(\diag(t_1,\dots,t_n,t_n^{-1},\dots,t_1^{-1})\right)=\prod_{i=1}^n\chi_i(t_i),$$ where the $\chi_i$ are unramified quasicharacters of $F^\times$.
This character determines a genuine character $\chi$ of $Z(\widetilde{T})$, and we form the 
unramified principal series $\pi_{Sp_{2n}}^{(\kappa r)}(\chi)$ as outlined above. 
If $\kappa=2$ then this requires a Weil factor, as in \cite{B-F-H}, equation (1.10).
The functions in this space are genuine with respect to the character $\epsilon'$  specified in Section~\ref{basic2}.
Similarly, let  $\xi$ be an unramified character of the maximal torus of $SO_k$, 
$$\xi\left(\diag(t_1,\dots,t_{k_1},I_{k-2k_1},t_{k_1}^{-1},\dots,t_1^{-1})\right)=\prod_{i=1}^{k_1}\xi_i(t_i),$$
where $k_1=[k/2]$, the $\xi_j$ are quasicharacters, and the maximal torus has a 1 in the middle entry if $k$ is odd, and form the unramified principal series $\pi_{SO_k}^{(r)}(\xi)$
 (for more details see \cite{B-F-G1}, Section 6).
We consider the case that $k_1=n$, so that each representation has the same number of Satake parameters, and study the theta correspondence between these representations. 
(One may describe matters slightly more generally using the notion of a Brylinski-Deligne extension, as in Gao \cite{Gao1};
see for example Leslie \cite{Leslie}, Section 3.  This does not change the discussion below in any significant way.)
 
Let $\psi$ be an
additive character of $F$ which is trivial on the ring of integers of $F$ but on no larger fractional ideal, and let
$(\omega_\psi,V_{\omega_\psi})$ be the Weil representation of $Sp^{(2)}_{2nk}$ with respect to $\psi$.  
Let $(\Theta^{(r)},V_{\Theta^{(r)}})$ be the local theta representation of $Sp_{2n+k(r-1)}^{(r)}$.  (We shall write $\Theta_{2n+k(r-1)}^{(r)}$ when we want to indicate the size of the symplectic group.)
Form the vector space $V_{\omega_\psi}\otimes V_{\Theta^{(r)}}$.
This space admits a representation $\omega_{\psi}\otimes \Theta_{}^{(r)}$ of $(Sp_{2nk}^{(2)}\rtimes{\mathcal H}_{2kn+1})\times Sp_{2n+k(r-1)}^{(r)}$.  We restrict this to a representation
of $SO_k^{(r)}\times Sp_{2n}^{(\kappa r)}$ using the same embeddings $\iota_1$, $\iota_2$ as in the global integral.  More precisely, 
the action of $(h,g)\in SO_k^{(r)}\times Sp_{2n}^{(\kappa r)}$
on a vector $v_1\otimes v_2$, $v_1\in V_{\omega_\psi}$, $v_2\in V_{\Theta^{(r)}}$ is given by
$$(\omega_{\psi}\otimes \Theta_{}^{(r)})(h,g)\cdot(v_1\otimes v_2)=\omega_\psi(\iota_1^{(2)}(p^{(1)}(h),p^{(\kappa)}(g)))v_1\otimes \Theta^{(r)}(\iota_2^{(r)}(h,p^{(r)}(g)))v_2.$$
Note that the subgroup $\{((1,\zeta),(1,\zeta^{-1}))\mid \zeta\in \mu_r\}$ acts trivially.
Let $J_{U_{k,r_1,n},\psi_{U_{k,r_1,n}}}$ or, for notational convenience, simply $J_{U_{k,r_1,n},\psi}$ be the twisted Jacquet functor with respect to the character $\psi_{U_{k,r_1,n}}$ of 
$U_{k,r_1,n}\subseteq Sp_{2n+k(r-1)}$, acting on the first factor in the tensor product by the map $l:U_{k,r_1,n}\to {\mathcal H}_{2kn+1}$ and acting on the second factor by the theta
representation. (This functor is defined near the end of Section~\ref{introduction} above.) This is the local analogue of the integral \eqref{lift1}. Similarly to the treatment of the global situation, the group 
$SO_k^{(r)}\times Sp_{2n}^{(\kappa r)}$ with the above action stabilizes the group $U_{k,r_1,n}$ and character, so the Jacquet module $J_{U_{k,r_1,n},\psi}(\omega_{\psi}\otimes \Theta_{}^{(r)})$ 
also affords a representation of $SO_k^{(r)}\times Sp_{2n}^{(\kappa r)}$. 
Let
$$\Hom_{SO_k^{(r)}\times Sp_{2n}^{(\kappa r)}}(J_{U_{k,r_1,n},\psi}(\omega_{\psi}\otimes \Theta_{}^{(r)}),\pi_{SO_k}^{(r)}(\xi)\otimes \pi_{Sp_{2n}}^{(\kappa r)}(\chi))$$ 
denote the space of
$(SO_k^{(r)}\times Sp_{2n}^{(\kappa r)})$-equivariant maps 
from $J_{U_{k,r_1,n},\psi}(\omega_{\psi}\otimes \Theta_{}^{(r)})$ to  $ \pi_{SO_k}^{(r)}(\xi)\otimes \pi_{Sp_{2n}}^{(\kappa r)}(\chi).$
Then we shall show
\begin{theorem}\label{unramified}
Let $k=2n$ or $2n+1$ and suppose that the characters $\chi$ and $\xi$ are in general position.  If 
\begin{equation}\label{nonvan-hypoth-thm}
\Hom_{SO_k^{(r)}\times Sp_{2n}^{(\kappa r)}}(J_{U_{k,r_1,n},\psi}(\omega_{\psi}\otimes \Theta_{}^{(r)}),\pi_{SO_k}^{(r)}(\xi)\otimes \pi_{Sp_{2n}}^{(\kappa r)}(\chi))\neq 0
\end{equation} 
then after applying a Weyl group element, $\xi_i=\chi_i$ for each $i$.  
\end{theorem}
Here we recall that the Weyl group of $Sp_{2n}$ acts on $\chi$ by permuting the indices and by inverting each $\chi_i$; the Weyl group of $SO_k$ acts on the $\xi_i$,
by a similar action (if $k$ is odd) and by permutations and by inversions of an even number of quasicharacters (if $k$ is even).  
Hence the conclusion is equivalent to the assertion
that the sets $\{\chi_i^{\pm1}\}$ and $\{\xi_j^{\pm1}\}$, which are each of cardinality $2n$ as the characters are in general position, are in bijection.
Also, if $k$ is even then we recall that the  principal series representations of $SO_k^{(r)}$ attached to $(\xi_1,\dots,\xi_{n-1},\xi_n)$ and $(\xi_1,\dots,\xi_{n-1},\xi_n^{-1})$
are isomorphic under the outer automorphism of $SO_k$  that is conjugation by
$$\begin{pmatrix}I_{k/2-1}&&&\\&0&1&\\&1&0&&\\&&&I_{k/2-1}\end{pmatrix}.$$

The proof of Theorem~\ref{unramified} requires the local version of the smallness of the representation $\Theta^{(r)}_{2l}$.  
Recall that  ${\mathcal O}_c(\Theta_{2l}^{(r)})$ is defined in Conjecture~\ref{conj1} above.
As in Section~\ref{conjtheta}, for a given unipotent orbit $\mathcal{O}$ let $U_{\mathcal{O}}$ denote the upper unipotent subgroup attached to this orbit.
Then the proof of Proposition~\ref{theta01} gives the following result.
 
\begin{proposition}\label{theta-local}
Suppose that ${\mathcal O}$ is a unipotent orbit which is greater than ${\mathcal O}_c(\Theta_{2l}^{(r)})$ or that is not related to
${\mathcal O}_c(\Theta_{2l}^{(r)})$. Then for any character $\psi_{\mathcal{O}}$ attached to $\mathcal{O}$, the 
twisted Jacquet module $J_{U_{\mathcal{O}},\psi_{\mathcal{O}}}(\Theta_{2l}^{(r)})=0$.
\end{proposition}

We turn to the proof of Theorem~\ref{unramified}. We will follow, in a loose sense, the approaches 
of Bump and the authors \cite{B-F-G2}, Section 6, and Leslie \cite{Leslie}, Section 10, 
in their proofs of unramified correspondences for two other theta maps that use a single theta function as an integral kernel. Below, we refer to these proofs for some details that 
are similar.  

\begin{proof}[Proof of Theorem~\ref{unramified}]  
We shall prove this by induction on $n$.   More precisely, suppose that the 
nonvanishing \eqref{nonvan-hypoth-thm} holds for given $n$, $k$, $\chi$, and $\xi$ with $k=2n$ or $k=2n+1$ and the characters $\chi$, $\xi$
in general position.  Then we show that up to permutation of the indices, 
we have $\chi_1=\xi_1$ or $\chi_1=\xi_1^{-1}$.  
We also show that a similar hypothesis holds true for $(n,k)$ replaced by $(n-1,k-2)$ and $(\chi,\xi)$ 
replaced by $(\chi',\xi')$ where $\chi'=(\chi_2,\dots,\chi_n)$, $\xi'=(\xi_2,\dots,\xi_n)$, that is, we show that
$$\Hom_{SO_{k-2}^{(r)}\times Sp_{2n-2}^{(\kappa r)}}(J_{U_{k-2,r_1,n-1},\psi}(\omega_{\psi}\otimes \Theta_{}^{(r)}),\pi_{SO_{k-2}}^{(r)}(\xi')\otimes \pi_{Sp_{2n-2}}^{(\kappa r)}(\chi'))\neq 0$$ 
Then the result follows by reduction to the case $n=1$, which is straightforward to confirm by the same method used below.

Let $P'_{1,k-2}$ denote the unipotent radical of $SO_k$ with Levi factor $M'_{1,k-2}:=GL_1\times SO_{k-2}$ and unipotent radical $U'_{1,2k-2}$.
Let $P_{1,2n-2}$ denote the unipotent radical of $Sp_{2n}$ with Levi factor $M_{1,2n-2}:=GL_1\times Sp_{2n-2}$ and unipotent radical $U_{1,2n-2}$.
Denote the inverse images of these groups in their respective covering groups by a tilde.
First, by transitivity of induction, we realize  $\pi_{SO_k}^{(r)}(\xi)$ as parabolically induced from $\xi_1\otimes \pi_{SO_k}^{(r)}(\xi')$ with 
$\xi'=(\xi_2,\dots,\xi_n)$. (Note that the inverse images of  $GL_1$ and $SO_{k-2}$ in the local covering group $SO_{k}^{(r)}$ commute 
due to block compatibility, so the tensor product construction is straightforward.)  Similarly,
$\pi_{Sp_{2n}}^{(\kappa r)}(\chi)$ is parabolically induced from the representation $\chi_1\otimes \pi_{Sp_{2n-2}}^{(\kappa r)}(\chi')$
where $\chi'=(\chi_2,\dots,\chi_n)$.  Second, the Jacquet module 
$J_{U'_{1,k-2}}(\pi_{SO_k}^{(r)}(\xi))$ 
is isomorphic to the direct sum of $\xi_i^{\pm1}\otimes \pi_{SO_{k-2}}^{(r)}(\xi'_{i,\pm})$ where $\xi'_{i,+}$ is obtained from $\xi$
by removing $\xi_i$ and $\xi'_{i,-}$ is obtained from $\xi$ by by removing $\xi_i$ and inverting one of the remaining quasicharacters.   (See, for example, Section 5.2 of \cite{B-Z}.)
Note that the action of the Weyl group of $SO_k$ on $\xi$ permutes these factors.
Similarly the Jacquet module 
$J_{U_{1,2n-2}}(\pi_{Sp_{2n}}^{(\kappa r)}(\chi))$ 
is isomorphic to the direct sum of $\chi_j^{\pm1}\otimes \pi_{Sp_{2n-2}}^{(\kappa r)}(\chi'_j)$ where $\chi_j'$ is obtained from $\chi$
by removing $\chi_j$.  
 It follows that the tensor product
$\pi_{Sp_{2n}}^{(\kappa r)}(\chi)\otimes \pi_{SO_k}^{(r)}(\xi)$ is a sum of the induced representations 
$$\Ind^{SO_k^{(r)}\times Sp_{2n}^{(\kappa r)}}_{{\widetilde P}'_{1,k-2}\times \widetilde{P}_{1,2n-2}^{(\kappa r)}}( \xi_i^{\pm1}\otimes \pi_{SO_{k-2}}^{(r)}(\xi_{i,\pm}')\otimes
\chi_j^{\pm1}\otimes \pi_{Sp_{2n-2}}^{(\kappa r)}(\chi_j')\delta_{P'_{1,k-2}}^{1/2}
\delta_{P_{1,2n-2}}^{1/2}).$$
(The weaker statement that this tensor product is glued from these representations (see \cite{B-Z}) would be sufficient below.)
Since the Weyl groups permutes these factors for different indices, we shall focus on the case $\xi_i^{\pm1}=\xi_1$, $\chi_j^{\pm1}=\chi_1$ without loss of generality, and in this
case we write $\xi'_{1,+}=\xi'$, $\chi_1'=\chi'$ (as in the prior paragraph).

Applying Frobenius reciprocity, the nonvanishing hypothesis \eqref{nonvan-hypoth-thm} in Theorem~\ref{unramified} implies that, up to the action of the Weyl groups as just explained,
 there is a nonzero ${{\widetilde M}'_{1,k-2}}\times {\widetilde M}_{1,n-1}$-equivariant map 
$$
J_{U'_{1,k-2}}J_{U_{1,2n-2}}(J_{U_{k,r_1,n},\psi}(\omega_{\psi}\otimes \Theta_{}^{(r)}))\to  \xi_1\otimes \pi_{SO_{k-2}}^{(r)}(\xi')\otimes 
\chi_1\otimes  \pi_{Sp_{2n-2}}^{(\kappa r)}(\chi')\delta_{P'_{1,k-2}}^{1/2}\delta_{P_{1,2n-2}}^{1/2}.
$$
To keep track of the characters that arise, let $A,B$ denote the block matrices
\begin{equation}\label{two-blocks}A=\begin{pmatrix} a&&\\&g'&\\&&a^{-1}\end{pmatrix}\in Sp_{2n}\qquad B=\begin{pmatrix} b&&\\&h'&\\&&b^{-1}\end{pmatrix}\in SO_k,
\end{equation}
with $a,b\in GL_1$, $g'\in Sp_{2n-2}$, $h'\in SO_{k-2}$.
Then $\delta_{P'_{1,k-2}}^{1/2}(B)=|b|^{k/2-1}$ and $\delta_{P_{1,2n-2}}^{1/2}(A)=|a|^{n}$.

Next we apply the argument in \cite{G-R-S4}, Proposition 6.6 (see p.\ 129).  This implies that
 the Hom space of such maps is a quotient of the space
\begin{equation}\label{jacquet-mod-1}
\Hom_{{\widetilde{M}'_{1,k-2}}\times \widetilde{M}_{1,n-1}}(J_{V_{k,r_1,n},\psi}(\omega_\psi\otimes \Theta^{(r)}),
 \xi_1\otimes \pi_{SO_{k-2}}^{(r)}(\xi')\otimes \chi_1\otimes \pi_{Sp_{2n-2}}^{(\kappa r)}(\chi')\delta_{P'_{1,k-2}}^{1/2}\delta_{P_{1,2n-2}}^{1/2}\delta)
 \end{equation}
and hence this space is nonzero. Here $V_{k,r_1,n}\subseteq Sp_{2n+k(r-1)}$ is the upper unipotent subgroup
$V_{k,r_1,n}=\iota_2(U'_{1,k-2},U_{1,2n-2})U_{k,r_1,n}^\flat$ where the group $U_{k,r_1,n}^\flat$
is defined as follows.  Let $\Mat'_{k,2n}$ denote the subgroup
of $\Mat_{k,2n}$ consisting of all matrices $Y=(y_{i,j})$ whose entries $y_{j,1}$, $2\leq j\leq k$ and $y_{k,i}$, $1\leq i\leq 2n$ are each zero. Then
$U_{k,r_1,n}^\flat$ is the subgroup of the group $U_{k,r_1,n}$ consisting of all matrices $u$ whose factorization~\eqref{mat2}  (with $a=k$, $b=r_1$, $c=n$) has the property that $Y$ is
in $\Mat'_{k,2n}$. The character $\psi$  of
$V_{k,r_1,n}$ is the character $\psi_{U_{k,r_1,n}}$ restricted to $U_{k,r_1,n}^\flat$ and then extended trivially to $\iota_2(U'_{1,k-2},U_{1,2n-2})$.
Also, the map $l$ restricted to $V_{k,r_1,n}$ gives a homomorphism from $V_{k,r_1,n}$ (or $U_{k,r_1,n}^\flat$) onto the Heisenberg group $\mathcal{H}_{(n-1)(k-2)}$
and this defines the action of $V_{k,r_1,n}$ on $\omega_\psi$.
The  character $\delta$ factors through the projection from ${\widetilde{M}'_{1,k-2}}\times \widetilde{M}_{1,n-1}$ to ${M}'_{1,k-2}\times {M}_{1,n-1}$,
and evaluated on \eqref{two-blocks} has two factors, a contribution of $|a|^{k/2-1}|b|^n$ due to the action of the Weil representation and a factor from
 the modular function of the quotient $U_{k,r_1,n}/U_{k,r_1,n}^\flat$.  This quotient may be identified 
with the group of all matrices $u_{k,n}'(Y_1,0)$ (see \eqref{mat1}) such that
$Y_1$ is an element in $\Mat_{k,2n}^{00}$ which is the complement  to $\Mat'_{k,2n}$  in $\Mat_{k,2n}$, so this factor is $|a|^{2-k}|b|^{-2n}$.
Combining these, we see that the value of the character $\delta$ on \eqref{two-blocks} above is $|a|^{1-k/2}|b|^{-n}$, so
$\delta_{P'_{1,k-2}}^{1/2}\delta_{P_{1,2n-2}}^{1/2}\delta$ takes the value $|ab^{-1}|^{1-k/2+n}$.
We remark that this step is the local analogue of using the definition of the theta function as a sum and unfolding part of the sum, a process that is 
used in the proof of Theorem~\ref{th2}; see equation \eqref{cusp31} and the discussion following, in the case $\alpha=1$, $\beta=k-2$. 

Next, we make a series of root exchanges. We have already used root exchanges for global integrals. For root exchange in the context of 
representations of local fields, see \cite{G-R-S2}, Section 2.2.   In the local case, they
allow one to replace the Jacquet module in \eqref{jacquet-mod-1} by another isomorphic Jacquet module and
to conclude that the corresponding Hom space is nonzero.  Globally or locally, the root exchanges require an additive character, 
and as in the global case, we use the additive character $\psi$ of $V_{k,r_1,n}$ obtained from $\psi_{U_{k,r_1,n}}$ (see \eqref{character})
and also of the additive character that comes from the Weil representation
$\omega_\psi$.   

The first set of exchanges are the same as the root exchanges used in the proof of Theorem~\ref{th2} above, where we treated the groups $L_{\alpha,\beta}$ introduced
in \eqref{em0}.  We use these same exchanges in the case $\alpha=1$, $\beta=k-2$.  The root exchanges described through \eqref{em3} allow us to conclude that
the space
\begin{equation}\notag 
\Hom_{{\widetilde{M}'_{1,k-2}}\times \widetilde{M}_{1,n-1}}(J_{V"_{k,r_1,n},\psi}(\omega_\psi\otimes \Theta^{(r)}),
 \xi_1\otimes \pi_{SO_{k-2}}^{(r)}(\xi')\otimes \chi_1\otimes \pi_{Sp_{2n-2}}^{(\kappa r)}(\chi')\,\delta_1)
 \end{equation}
 is nonzero.  Here $V''_{k,r_1,n}$ is the unipotent subgroup of $Sp_{2n+k(r-1)}$ generated by the following two types of matrices.  First, 
 the matrices with factorization \eqref{mat2} with
 the $X_i$, $1\leq i\leq r_1-1$, $Y$, and $Z$ of the form
\begin{align*}\notag
&X_i=\begin{pmatrix}*&*&*\\ 0_{(k-2)\times 1}&*&*\\ 0&0_{1\times(k-2)}&*\end{pmatrix}\in \Mat_k,\qquad Y=\begin{pmatrix} *&*&*\\ 0_{(k-2)\times 1}&*&*\\ 0&0_{1\times(2n-2)}&0\end{pmatrix}\in \Mat_{k\times 2n},\\
&Z=\begin{pmatrix} *&*&*\\ 0_{(k-2)\times 1}&*&*\\ 0&0_{1\times(k-2)}&*\end{pmatrix}\in \Mat_k^0,
\end{align*} 
where the stars indicate arbitrary elements of $F$. Second, 
the block-diagonal matrices in the Levi of $P_{k,r_1,n}$ whose entries in the first $r_1$ diagonal $k\times k$ blocks are of the form \eqref{em0} with $\alpha=1$, $\beta=k-2$.
These root exchanges also change the character $\delta_{P'_{1,k-2}}^{1/2}\delta_{P_{1,2n-2}}^{1/2}\delta$ above to $\delta_1$ whose value on \eqref{two-blocks} is 
$\delta_{P'_{1,k-2}}^{1/2}\delta_{P_{1,2n-2}}^{1/2}\delta |b|^{-k-(r_1-1)(2k-2)}=|ab^{-1}|^{1-k/2+n}|b|^{-k-(r_1-1)(2k-2)}$.

Next we conjugate by the Weyl group element $w\in Sp_{2n+k(r-1)}$ which is the shortest Weyl group element that conjugates the matrix
\begin{equation}\label{first-matrix}\diag(B,\dots,B,A,B^*,\dots,B^*),
\end{equation}
to the matrix
\begin{equation}\label{conj-matrix-jacmod}
\diag(b,\dots,b,a,h',\dots,h',g',(h')^*,\dots.(h')^*,a^{-1},b^{-1},\dots,b^{-1}).
\end{equation}
In \eqref{first-matrix} above, $A$ and $B$ are as given in \eqref{two-blocks}, the notation $B^*$ is defined in Section~\ref{basic}, and each of $B,B^*$ appears $r_1$ times;
in \eqref{conj-matrix-jacmod}, each of $b,b^{-1}$ is repeated $r-1$ times and each of $h',(h')^*$ is repeated $r_1$ times.
After doing so, we conclude that the Hom space
\begin{equation}\notag 
\Hom_{\widetilde{M}}(J_{V',\psi'}(\omega_\psi\otimes \Theta^{(r)}),
 \xi_1\otimes \pi_{SO_{k-2}}^{(r)}(\xi')\otimes \chi_1\otimes \pi_{Sp_{2n-2}}^{(\kappa r)}(\chi')\,\delta_1)
\end{equation}
is nonzero, where $\widetilde{M}$ is isomorphic to ${\widetilde{M}'_{1,k-2}}\times \widetilde{M}_{1,n-1}$ and $\iota_2(\widetilde{M'})$ projects to the group of matrices of the 
form \eqref{conj-matrix-jacmod}. 
The group $V'=wV''_{k,r_1,n}w^{-1}$ is closely related to the unipotent group that appeared in the integration in \eqref{cusp33} above. More precisely, an element in $V'$ may be 
written as a product of the form
\begin{equation}\label{first}
\begin{pmatrix} I_r&C&D\\ &I&C^*\\ &&I_r\end{pmatrix} \begin{pmatrix} v&&\\
&u&\\ &&v^*\end{pmatrix} \begin{pmatrix} I_r&&\\ R&I& \\ S&R^*&I_r\end{pmatrix}
\end{equation}
where $I=I_{2n+k(r-1)-2r}$, $u\in U_{k-2,r_1,n-1}$ and $v\in V_r^0$, which is defined as follows. Let $V_r$ be the maximal upper unipotent subgroup of $GL_r$. Then $V_r^0$ is the subgroup of $V_r$ 
consisting of all matrices $(v_{i,j})$ such that $v_{r-1,r}=0$. The matrices $C$, $D$, $R$ and $S$ are defined similarly to the corresponding matrices $C$, $D$, $A$ and $B$ 
in integral  \eqref{cusp33}.  With these notations the character $\psi'$ is described as follows. First
$$\psi'(\text{diag}(v,u,v^*))=\psi_r^0(v)\psi_{U_{k-2,r_1,n-1}}(u)$$
where $\psi_r^0$ is the Whittaker character of $V_r$ restricted to $V_r^0$. Then $\psi'$ is extended trivially to the rest of $V'$.

In view of this, the space we are studying is analogous to the integral \eqref{cusp33}, with the unipotent integration in this local context being the process of taking
the twisted Jacquet module, except that in that integral one is integrating
$g$ over (a cover of) $Sp_{2n}$ while here we have the Hom space with respect to the parabolic subgroup $\widetilde{M}_{1,n-1}$ which projects to 
Levi factor $GL_1\times Sp_{2n-2}$.  Note that in the proof of Theorem~\ref{th2} we also conjugated by a Weyl group element, but not quite the same one we are using now,
due to the difference between $Sp_{2n}$ and $GL_1\times Sp_{2n-2}$.
We next use many of the same steps in the proof of Theorem~\ref{th2} with minor modifications.

First, we carry out a series of root exchanges that
are the same as the ones given following \eqref{block6} with $\alpha=1$, $\beta=k-2$, but with an adjustment on the size of the blocks, replacing $r-1$ by $r$.
(For example, $I_{\alpha(r-1)}$
in \eqref{cusp33} is now replaced by $I_r$, as in \eqref{first}.)   These are performed row by row. 
Specifically, we start with the second row of the matrix $C$ as given in \eqref{first}. Choosing the third column in the matrix $R$ 
in \eqref{first}, we perform root exchange; this is the same as the first in the sequence of root exchanges that was carried out following 
\eqref{block6}, and contributes the Jacobian $|b|^k$. We continue to make root exchanges,  again following the procedure
described following \eqref{block6}, until we have completed the first $r-2$ rows.
In the local case, these root exchanges imply the nonvanishing of
the Hom space 
\begin{equation}\label{hom-space2}
\Hom_{\widetilde{M}}(J_{V_1',\psi_1'}(\omega_\psi\otimes \Theta^{(r)}),
 \xi_1\otimes \pi_{SO_{k-2}}^{(r)}(\xi')\otimes \chi_1\otimes \pi_{Sp_{2n-2}}^{(\kappa r)}(\chi')\,\delta_2).
\end{equation}
Here $V_1'$ is the subgroup of $Sp_{2n+k(r-1)}$ of matrices which have a factorization \eqref{first} that satisfies the following conditions:
for $C$: $C_{i,j}=0$ if $i=r-1$, $1\leq j\leq 2n+(r-1)(k-2)-k$ or $i=r$, $1\leq j\leq r_1(k-2)$; for $D$, $D_{r-1,1}=D_{r-1,2}=D_{r,2}=0$;
for $R$, $R_{i,j}=0$ if $1\leq j\leq r-1$ or if $j=r$, $i> (r_1-1)(k-2)$; and $S=0$.
Also, $\psi_1'$ is the character of $V_1'$ that is $\psi'$ extended trivially to $V_1'$.  (Note that studying this nonvanishing is the local analogue of studying the 
possible nonvanishing of \eqref{cusp34}
with $\alpha=1$ and $r-1$ replaced by $r$, and this same group
and character, up to this modification, appear there.)  The character $\delta_2$ in \eqref{hom-space2} is computed as follows.
In carrying out these root exchanges, 
the character $\delta_1$, evaluated on \eqref{two-blocks}, 
is multiplied by a power of $|b|$ at each step of the root exchange process.
For $1\leq j\leq r_1$, to fill in the $j$-th row by root exchange the adjustment on the character is $|b|^{k(j-1)}$, 
and for row $r_1+j$, $1\leq j\leq r_1-1$, the adjustment is $|b|^{r_1k+2n-2+(j-1)(k-4)}$.  The character $\delta_2$ in \eqref{hom-space2} is obtained by
combining these adjustments with the prior character $\delta_1$.  (We will compute the final character later in the proof.)

We next treat the $(r-1)$-st row. To do so, we perform a local version of Fourier expansion. 
To describe the procedure, let $L$ denote a unipotent abelian subgroup of $Sp_{2n+k(r-1)}$ such that $LV_1'$ is also 
a unipotent group, and such that $\psi_1'$ is a well-defined character of $LV_1'$ when extended trivially from $V_1'$ to $LV_1'$. 
Also suppose that the group $M$ normalizes this unipotent group. Then it 
follows from the nonvanishing of the Hom space \eqref{hom-space2} that at least
one of the spaces
\begin{equation}\label{617}
\Hom_{\widetilde{M}_{\psi_L}}( J_{L,\psi_L}(J_{V_1',\psi_1'}(\omega_\psi\otimes 
\Theta^{(r)})),\xi_1\otimes \pi_{SO_{k-2}}^{(r)}(\xi')\otimes \chi_1\otimes \pi_{Sp_{2n-2}}^{(\kappa r)}(\chi')\,\delta_2).
\end{equation}
is non-zero, as $\psi_L$ runs over a set of representatives of the characters of the group $L$ under the action of the group $M$.
Here $\widetilde{M}_{\psi_L}$ is the stabilizer of $\psi_L$ in $\widetilde{M}$.
Indeed, this follows from the gluing described in 
Bernstein and Zelevinsky \cite{B-Z}, Theorem 5.2, applied as in the work of Bump and the authors \cite{B-F-G2}, pp.\ 392-3.
As explained there, one obtains an exact sequence from the Geometrical Lemma of \cite{B-Z}, p.\ 448, and then applies a Jacquet functor which preserves the sequence (as it is an exact functor).
One term involves a Jacquet functor applied to a compactly-induced representation, and to study the Hom space from this piece, one applies Mackey
theory as extended in \cite{B-Z}. 

We begin to make this expansion with the root supported on one-parameter subgroup  $E_{r-1,2n+k(r-1)-r+2}$ (defined before Proposition~\ref{theta04}). 
In fact, this is the same procedure used to analyze \eqref{cusp34} and described
in detail following \eqref{cusp34}. The Jacquet module with the character (the `non-constant term' in the
global computation) factors through the twisted Jacquet module corresponding to the
orbit $((r+1)1^{2n+k(r-1)-r+1})$ for $\Theta^{(r)}$.  However,  this vanishes by the smallness of the representation $\Theta^{(r)}$,
Proposition~\ref{theta-local}.  
Let $V_2'=E_{r-1,2n+k(r-1)-r+2}V_1'$, and let $\psi_2'$ denote the character of $V_2'$ which is the trivial extension of $\psi_1'$. Then we deduce that the space \eqref{hom-space2}
is isomorphic to 
\begin{equation}\label{hom2}
\Hom_{\widetilde{M}} (J_{V_2',\psi_2'}(\omega_\psi\otimes 
\Theta^{(r)}),\xi_1\otimes \pi_{SO_{k-2}}^{(r)}(\xi')\otimes \chi_1\otimes \pi_{Sp_{2n-2}}^{(\kappa r)}(\chi')\,\delta_2),
\end{equation}
hence nonzero.

We move to the next root, with root subgroup $E'_{r-1,2n+k(r-1)-r+1}$.  
For convenience denote this group $L$, and let 
$$x(t)=I_{2n+k(r-1)}+te'_{r-1,2n+k(r-1)-r+1}\in L$$
($e'_{i,j}$ is defined before Proposition~\ref{theta04}).
The group $M$ consisting of all matrices of the form \eqref{conj-matrix-jacmod} acts on $L$ by conjugation with two orbits; indeed  if $m$ is given by \eqref{conj-matrix-jacmod}, then
$mx(t)m^{-1}=x(abt)$. Let $M_0\subset M$ denote the stabilizer of this action; this is the  subgroup consisting of all matrices  \eqref{conj-matrix-jacmod} with $b=a^{-1}$. 
Let $\psi_L$ denote a nontrivial character of $L$. Then, since \eqref{hom2} is not zero, we deduce that at least one of the two spaces 
\begin{equation}\label{hom3}
\Hom_{\widetilde{M}_0}( J_{L,\psi_L}(J_{V_2',\psi_2'}(\omega_\psi\otimes 
\Theta^{(r)})),\xi_1\otimes \pi_{SO_{k-2}}^{(r)}(\xi')\otimes \chi_1\otimes \pi_{Sp_{2n-2}}^{(\kappa r)}(\chi')\,\delta_2)
\end{equation}
and
\begin{equation}\label{hom4}
\Hom_{\widetilde{M}}( J_{V_3',\psi_3'}(\omega_\psi\otimes 
\Theta^{(r)}),\xi_1\otimes \pi_{SO_{k-2}}^{(r)}(\xi')\otimes \chi_1\otimes \pi_{Sp_{2n-2}}^{(\kappa r)}(\chi')\,\delta_2)
\end{equation}
is also not zero.  Here $V_3'=LV_2'$ and the character $\psi_3'$ is obtained by extending $\psi_2'$ trivially to $V_3'$.

We start with the case that the space \eqref{hom3} is not zero. (We will deal with the space \eqref{hom4} afterwards.)
There the character $\psi_L$, which appears through the twisting, allows us to do root exchanges (similar to the
the treatment of \eqref{cusp34} above), in order 
to fill in all the remaining positive root subgroups $E'_{r-1,j}$ in row $r-1$, as follows.  First, we use $\psi_L$ to move roots into row $r-1$ 
from positive roots $E'_{r,j}$ with $r+r_1(k-2)<j\leq 2n+k(r-1)-r$ and $E_{r,2n+k(r-1)-r+1}$ that appear in $V_2'$
 in row $r$. We then use root exchange to exchange the $(r_1-1)(k-2)$ negative roots $E_{j,r}'$, $r<j\leq r+(r_1-1)(k-2)$
of $V_2'$ (these appear in $R$ and $R^*$ of \eqref{first}) into row $r-1$.  
Note that this root exchange is possible because there is an additive character appearing from the Weil representation $\omega_\psi$.
Doing these exchanges multiplies the character $\delta_2$ by a factor of $|b|^{r_1k+2n-2+(r_1-1)(k-4)}|a|^{(r_1-1)(k-2)}$.  
Thus we obtain the nonvanishing of the Hom space
$$
\Hom_{\widetilde{M}_0}( J_{L',\psi_{L'}}(J_{V_2'',\psi_2'}(\omega_\psi\otimes 
\Theta^{(r)})),\xi_1\otimes \pi_{SO_{k-2}}^{(r)}(\xi')\otimes \chi_1\otimes \pi_{Sp_{2n-2}}^{(\kappa r)}(\chi')\,\delta_3).
$$
Here $L'$ is the (abelian) group generated by $L$ and all the positive root subgroups of the form $E_{r-1,j}'$, $r\leq j\leq 2n+k(r-1)-r$ (note that this group includes $E_{r-1,2n+k(r-1)-r+2}$),
$\psi_{L'}$ is $\psi_L$ extended trivially from $L$ to $L'$, $V_2''$ is the subgroup of $V_2'$ consisting of upper unipotent matrices with 0 in position $(r,j)$, $r<j\leq 2n+k(r-1)-r+1$
and $\delta_3=\delta_2 |b|^{r_1k+2n-2+(r_1-1)(k-4)}|a|^{(r_1-1)(k-2)}.$

At this point we conjugate by the shortest Weyl group element, call it $w'$, that interchanges $a$ and $a^{-1}$ in \eqref{conj-matrix-jacmod}.
We then fill in the remaining positive roots in row $r$. To do so we repeat the same process as above. We start with the group
$E_{r,2n+k(r-1)-r+1}$.  Expanding, the twisted Jacquet module vanishes by the smallness of $\Theta^{(r)}$.  Then we consider the group
$L''/E_{r,2n+k(r-1)-r+1}$ where $L''$ is the group generated by $E_{r,2n+k(r-1)-r+1}$ and the $E'_{r,j}$ with $r+1\leq j\leq 2n+k(r-1)-r$.
The Hom spaces with nontrivial characters at these roots each are zero, since 
nonvanishing for a nontrivial character in row $r$ would imply that the representation $\Theta^{(r)}$ supports a functional for a unipotent orbit
that has at least $r+1$ in its partition, and this would contradict Proposition~\ref{theta-local}.

Putting all this together, the nonvanishing of the Hom space \eqref{hom3} implies that
$$\Hom_{\widetilde{M}_1}(J_{U_{k-2,r_1,n-1},\psi}(\omega_{\psi}\otimes J_{U_{GL_r},\psi_{Wh}}(J_{U_{r,1,n+kr_1-r}}(\Theta^{(r)}))),
\xi_1\otimes\pi_{SO_{k-2}}^{(r)}(\xi')\otimes \chi_1^{-1}\otimes\pi_{Sp_{2n-2}}^{(\kappa r)}(\chi')\,\delta_3)$$
is nonzero.  Here $\widetilde{M}_1=w'\widetilde{M}_0(w')^{-1}$. The projection of $\widetilde{M}_1$  to $Sp_{2n+k(r-1)}$ is
\begin{equation}\notag
\diag(a^{-1},\dots,a^{-1},h',\dots,h',g',(h')^*,\dots.(h')^*,a,\dots,a)
\end{equation}
where $a^{-1}$ and $a$ are each repeated $r$ times and $h'$, $(h')^*$ are each repeated $r_1$ times.
The Jacquet functor applied to $\Theta^{(r)}:=\Theta^{(r)}_{2n+k(r-1)}$ is the untwisted Jacquet functor with respect to $U_{r,1,n+kr_1-r}$, 
the unipotent radical of the parabolic with Levi factor $GL_r\times Sp_{2n+k(r-1)-2r}$.  By Proposition~\ref{constant-term-theta}, this functor gives the 
module $\Theta_{GL_r}^{(r)}\otimes \Theta_{{2n+2k(r-1)-2r}}^{(r)}$.  Here $\Theta_{GL_r}^{(r)}$ denotes the local theta representation for $GL_r^{(r)}(F)$.
The Jacquet functor $J_{U_{GL_r},\psi_{Wh}}$ is the twisted Jacquet functor with respect to the
Whittaker character of $GL_r$, and is applied to the first component of this tensor product. 

However,
the theta representation on the $r$-fold cover of $GL_r$ is generic by \cite{K-P}, and its central character is computed there. Thus the action of $GL_1$ (more precisely, 
the subgroup of central elements in the 
metaplectic group) is as follows.
 If $B$ is the Borel of $Sp_{2n+k(r-1)}$ then we obtain a total 
contribution from this step of
$$\delta_B^{\frac{r-1}{2r}}\left(\begin{pmatrix}a^{-1}I_r&&\\&I_{2n+k(r-1)-2r}&\\&&aI_r\end{pmatrix}\right)=|a|^{-r_1(2n+(k-1)(r-1))}.$$
The character $\delta_3$ is given on \eqref{two-blocks} with $b=a^{-1}$ by $|a|^{-1}$ raised to the power obtained by collecting all the terms arising from
the unfolding and the root exchanges, each of which is noted above.
This power is
\begin{align*}k+(r_1-1)(2k-2)+\sum_{j=1}^{r_1-1} jk+r_1(r_1k+2n-2)+\sum_{j=1}^{r_1-1} j (k-4) - (r_1-1)(k-2)
\\=r_1(2n+(k-1)(r-1)).
\end{align*}
This power of $|a|^{-1}$ thus exactly matches the corresponding contribution from the Whittaker coefficient of $\Theta_{GL_r}^{(r)}$.

We conclude that if the Hom space \eqref{hom3}
 is nonzero, then  $\chi_1=\xi_1$ and moreover the space
$$\Hom_{SO_{k-2}^{(r)}\times Sp_{2n-2}^{(\kappa r)}}
(J_{U_{k-2,r_1,n-1},\psi}(\omega_{\psi}\otimes \Theta_{2n-2+(k-2)(r-1)}^{(r)})),  \pi_{SO_{k-2}}^{(r)}(\xi')\otimes \pi_{Sp_{2n-2}}^{(\kappa r)}(\chi'))$$
is nonzero. This is the same nonvanishing as in \eqref{nonvan-hypoth-thm} but with $(n,k)$ replaced by $(n-1,k-2)$ and $(\chi,\xi)$
replaced by $(\chi',\xi')$. Thus in this case we are done by induction.   

We must also analyze the situation that it is the Hom space \eqref{hom4} which is non-zero. 
In this case, we continue to expand along the roots of row $r-1$ using the gluing or local Fourier expansion procedure.  
We already have the roots in this row that sit over the last $h'$ and the final $r$ entries of the matrix \eqref{conj-matrix-jacmod}.
We first expand along the root spaces $E'_{r-1,j}$ with $r<j\leq 2n+k(r-1)-r-(k-2)$, taking into account the groups $SO_{k-2}$ and $Sp_{k-2}$ which
act on the characters in the Fourier expansion by conjugation (see following \eqref{617}). Thus we expand using the abelian groups consisting of
$k-2$ root spaces $E'_{r,j}$ at a time for the roots whose root spaces include an entry in the $r$-th row above each $h'$ and $h$ and
$2n-2$ at a time for the roots above the middle $g'$.
It follows by using the smallness of $\Theta^{(r)}$ or general position (when a copy of $GL_1$, restricted to $r$-th powers, acts on the module under consideration by a fixed character)
that the contributions coming from each nontrivial character in this expansion vanish.
We conclude that the space
$$
\Hom_{\widetilde{M}}( J_{V_4',\psi_4'}(\omega_\psi\otimes 
\Theta^{(r)}),\xi_1\otimes \pi_{SO_{k-2}}^{(r)}(\xi')\otimes \chi_1\otimes \pi_{Sp_{2n-2}}^{(\kappa r)}(\chi')\,\delta_2)
$$
is non-zero, where $V_4'$ is the group generated by $V_3'$ and the root subgroups $E'_{r-1,j}$ with $r<j\leq 2n+k(r-1)-r+1$, and the
character $\psi_4'$ is $\psi_3'$ extended trivially
to $V_4'$.

 We now do the expansion along the root subgroup $E'_{r-1,r}$.
 This is again glued from two terms, but the situation here is different.  
 First, there is a constant term at this root.  In this case, the Jacquet module arising from the expansion factors through $J_{U_{r-1,1,2n+(k-2)(r-1)}}(\Theta^{(r)})$.
By Proposition~\ref{constant-term-theta}, this gives
the representation $\Theta_{GL_{r-1}}^{(r)}\otimes \Theta^{(r)}_{2n+(k-2)(r-1)}$.  However, here the $GL_1$ in position $b$ in \eqref{conj-matrix-jacmod}
(restricted to $r$-th powers) acts by a fixed character.  Since we are assuming that $\chi$ and $\xi$ are in general position, the Hom space from this term is zero.
We conclude that it is the non-constant piece that must be non-zero.
 
To analyze the non-constant piece, i.e.\ the contribution
 from a nontrivial character $\psi_{r-1,r}$ of the group $E'_{r-1,r}$,
we must restrict to the stabilizer in $\widetilde{M}$ of this character.  If $x_{r-1,r}(t)=I_{2n+k(r-1)}+te'_{r-1,r}$
and $m$ is given by \eqref{conj-matrix-jacmod} then
$mx_{r-1,r}(t)m^{-1}=x_{r-1,r}(ba^{-1}t)$. Thus this character is stabilized by all matrices  \eqref{conj-matrix-jacmod} with $a=b$. 
One argues similarly to the above to see that this Hom space is nonzero only if $\chi_1=\xi_1^{-1}$
 and if
 $$\Hom_{SO_{k-2}^{(r)}\times Sp_{2n-2}^{(\kappa r)}}
(J_{U_{k-2,r_1,n-1},\psi}(\omega_{\psi}\otimes \Theta_{2n-2+(k-2)(r-1)}^{(r)})),  \pi_{SO_{k-2}}^{(r)}(\xi')\otimes \pi_{Sp_{2n-2}}^{(\kappa r)}(\chi'))\neq(0).$$
Then we are done by induction. 

This concludes the proof of Theorem~\ref{unramified}.
\end{proof}


\begin{thebibliography}{AAAAA}

\bibitem[B-L-S]{B-L-S} Banks, W.; Levy, J.; Sepanski, M. R.:
Block-compatible metaplectic cocycles.
J. Reine Angew. Math. {\bf 507} (1999), 131--163.

\bibitem[B-Z]{B-Z} Bernstein, I.N.; Zelevinsky, A.V.: Induced representations of reductive $p$-adic groups, I.
Ann. scient. \'Ec. Norm. Sup. {\bf 10} (1977), no.\ 4, 441--472.

\bibitem[B-F-G1]{B-F-G1} Bump, D.; Friedberg, S.; Ginzburg, D.: Small representations for odd orthogonal groups.
Intern. Math Res. Notices (IMRN) {\bf 2003} (2003), no. 25, 1363--1393.

\bibitem[B-F-G2]{B-F-G2} Bump, D.; Friedberg, S.; Ginzburg, D.: Lifting automorphic representations on the double covers of orthogonal groups.
Duke Math. J. {\bf 131} (2006), no. 2, 363--396.

\bibitem[B-F-H]{B-F-H} Bump, D.; Friedberg, S.; Hoffstein, J.: $p$-adic Whittaker functions on the metaplectic group. 
Duke Math. J. {\bf 63} (1991), no.\ 2, 379--397.

\bibitem[B-D]{B-D} Brylinski, J-L; Deligne, P.:
Central extensions of reductive groups by $K_2$.
Publ. Math. Inst. Hautes \'Etudes Sci. No. 94 (2001), 5--85.

\bibitem[C1]{C1} Cai, Y.: Fourier coefficients for theta representations on covers of general linear groups.
Trans. Amer. Math. Soc. {\bf 371} (2019), no. 11, 7585--7626.

\bibitem[C2]{C2} Cai, Y.: Twisted doubling integrals for Brylinski-Deligne extensions of classical groups.  arXiv:1912.08068.
 
 \bibitem[CFGK1]{CFGK1} Cai, Y.; Friedberg, S.; Ginzburg, D.; Kaplan, E.: Doubling constructions and tensor product L-functions: the linear case. 
Invent. Math. {\bf 217} (2019), no. 3, 985--1068.

 \bibitem[CFGK2]{CFGK2} Cai, Y.; Friedberg, S.; Ginzburg, D.; Kaplan, E.:  Doubling constructions for covering groups and tensor product L-functions.
arXiv:1601.08240.

\bibitem[C-M]{C-M}  Collingwood; D. H., McGovern, W. M.:
Nilpotent orbits in semisimple Lie algebras.
Van Nostrand Reinhold Mathematics Series. Van Nostrand Reinhold Co., New York (1993).

\bibitem[F-G1]{F-G1}  Friedberg, S.; Ginzburg, D.: Criteria for the existence of cuspidal theta representations. Res. Number Theory 2 (2016), Art. 16, 16 pp.

\bibitem[F-G2]{F-G2}  Friedberg, S.; Ginzburg, D.: Theta functions on covers of symplectic groups. 
In: Automorphic Forms, L-functions and Number Theory, a special issue of BIMS in honor of Prof. Freydoon Shahidi's 70th birthday.  Bull. Iranian Math. Soc. {\bf 43}(4) (2017), no. 4, 89--116.

\bibitem[F-G3]{F-G3} Friedberg, S.; Ginzburg, D.: Descent and theta functions for metaplectic groups. J. Eur. Math. Soc. 20 (2018), 1913--1957. 

\bibitem[F-G4]{F-G4} Dimensions of automorphic representations, $L$-functions and liftings.  To appear in Trace Formulas, a volume in the series Springer Symposia.

\bibitem[F-G5]{F-G5}  Friedberg, S.; Ginzburg, D.: The first occurrence for the generalized classical theta lift, in preparation.

\bibitem[Gao1]{Gao1} Gao, F.: The Langlands-Shahidi $L$-functions for Brylinski-Deligne extensions. 
Amer. J. Math. {\bf 140} (2018), no. 1, 83--137.

\bibitem[Gao2]{Gao2} Gao, F.: Distinguished theta representations for certain covering groups.
Pacific J. Math. 290 (2017), no. 2, 333--379.

\bibitem[G-T]{G-T} Gao, F.; Tsai, W.-Y.: On the wavefront sets associated with theta representations.  arXiv:2008.03630. 

\bibitem[G]{G} Ginzburg, D.: Certain conjectures relating unipotent orbits to automorphic representations. Israel J. Math. {\bf 151} (2006), 323--355.

\bibitem[G-R-S1]{G-R-S1} Ginzburg, D.; Rallis, S.; Soudry, D.: L-functions for symplectic groups. Bull. Soc. Math. France {\bf 126} (1998), no. 2, 181--244.

\bibitem[G-R-S2]{G-R-S2} Ginzburg, D.; Rallis, S.; Soudry, D.: On a correspondence between cuspidal representations of $GL_{2n}$ and $\widetilde{Sp}_{2n}$. J. Amer. Math. Soc. 12 (1999), 
no. 2, 243--266.

\bibitem[G-R-S3]{G-R-S3}  Ginzburg, D.; Rallis, S.; Soudry, D.: On Fourier coefficients of automorphic forms of symplectic groups. Manuscripta Math. {\bf 111} (2003), no. 1, 1--16. 

\bibitem[G-R-S4]{G-R-S4}  Ginzburg, D.; Rallis, S.; Soudry, D.: The descent map from automorphic representations of $GL(n)$ to classical groups. 
World Scientific Publishing Co. Pte. Ltd., Hackensack, NJ (2011). 

\bibitem[Ho]{Ho} Howe, R.: $\theta$-series and invariant theory. Automorphic forms, representations and L-functions (Proc. Sympos. Pure Math., Oregon State Univ., Corvallis, Ore., 1977), Part 1, 
pp. 275--285, Proc. Sympos. Pure Math., XXXIII, Amer. Math. Soc., Providence, R.I., 1979.

\bibitem[Ik1]{Ik1} Ikeda, T.: On the theory of Jacobi forms and Fourier-Jacobi
coefficients of Eisenstein series. J. Math. Kyoto Univ. {\bf 34} (1994), no. 3, 615--636.

\bibitem[Ik2]{Ik2} Ikeda, T.: On the residue of the Eisenstein series and the Siegel-Weil formula. Compositio Math. {\bf 103} (1996), no.\ 2, 183--218.

\bibitem[J-L]{J-L}  Jiang, D.; Liu, B.: On special unipotent orbits and Fourier coefficients for automorphic forms on symplectic groups. J. Number 
Theory {\bf 146} (2015), 343---389.

\bibitem[Kap]{Kap} Kaplan, E.: Doubling Constructions and Tensor Product L-Functions: coverings of the symplectic group. arXiv:1902.00880.

\bibitem[K-P]{K-P} Kazhdan, D. A.; Patterson, S. J.: Metaplectic forms. Inst. Hautes \'Etudes Sci. Publ. Math., No. 59 (1984), 35--142.

\bibitem[Ku]{Ku} Kubota, T.:
Some results concerning reciprocity and functional analysis. Actes du Congr\`es International des Math\'ematiciens (Nice, 1970), Tome 1, pp. 395--399. Gauthier-Villars, Paris, 1971.

\bibitem[Ku1]{Ku1} Kudla, S.: Splitting metaplectic covers of dual reductive pairs.  Israel J. Math. {\bf 87} (1994), 361--401.

\bibitem[Ku2]{Ku2} Kudla, S.: Notes on the local theta correspondence.  Notes from 10 lectures given at the European School on Group Theory in September 1996.
Available at: http://www.math.toronto.edu/$\sim$skudla/ssk.research.html.

\bibitem[L]{Leslie} Leslie, S.: A Generalized Theta lifting, CAP representations, and Arthur parameters. 
Trans. Amer. Math. Soc. {\bf 372} (2019), no. 7, 5069--5121.

\bibitem[Mat]{Mat} Matsumoto, H.:
Sur les sous-groupes arithm\'etiques des groupes semi-simples d\'eploy\'es. 
Ann. Sci. \'Ecole Norm. Sup. (4) 2 (1969), 1--62.

\bibitem[McN]{McN} McNamara, P.: Principal series representations of metaplectic groups over local fields. In: Multiple Dirichlet series, L-functions and automorphic forms, 299--327,
Progr. Math., 300, Birkh\"auser/Springer, New York, 2012.

\bibitem[M-W]{M-W} M\oe glin, C.; Waldspurger, J.-L.: 
Spectral decomposition and Eisenstein series.
Une paraphrase de l'\'Ecriture [A paraphrase of Scripture]. Cambridge Tracts in Mathematics, 113. Cambridge University Press, Cambridge, 1995

\bibitem[Mo]{Mo} Moore, C. C.: Group extensions of $p$-adic and adelic linear groups.  Inst. Hautes \'Etudes Sci. Publ. Math., No. 35 (1968), 5--70.

\bibitem[O-S]{O-S}  Offen, O.; Sayag, E.: Global mixed periods and local Klyachko models for the general linear group. 
Int. Math. Res. Not. IMRN {\bf 2008}, no. 1, Art. ID rnm 136, 25 pp. 

\bibitem[Ra1]{R-functor} Rallis, S.: Langlands' functoriality and the Weil representation.
Amer. J. Math. {\bf 104} (1982), no. 3, 469--515.

\bibitem[Ra2]{R1} Rallis, S.: $L$-functions and the oscillator representation.
Lecture Notes in Mathematics, 1245. Springer-Verlag, Berlin, 1987.

\bibitem[Rao]{Rao} Ranga Rao, R.:
On some explicit formulas in the theory of Weil representation. 
Pacific J. Math. {\bf 157} (1993), no. 2, 335--371.

\bibitem[Ro]{Roberts} Roberts, B. Nonvanishing of global theta lifts from orthogonal groups. J. Ramanujan Math. Soc. {\bf 14} (1999), no.\ 2, 131--194.

\bibitem[Su1]{Su1} Suzuki, T.: Distinguished representations of metaplectic groups. Amer. J. Math. {\bf 120} (1998), no.\ 4, 723--755.

\bibitem[Su2]{Su2} Suzuki, T.: On the Fourier coefficients of metaplectic forms. Ryukyu Math. J. {\bf 25} (2012), 21--106.

\bibitem[Sw]{Sw} Sweet, W. J. Jr.: The metaplectic case of the Weil-Siegel formula. 
Thesis, Univ.\ of Maryland, 1990.

\bibitem[Tak]{Tak} Takeda, S.:  The twisted symmetric square $L$-function of $GL(r)$. Duke Math. J. {\bf 163} (2014), no. 1, 175--266.

\bibitem[We]{We} Weil, A.:
Sur certains groupes d'op\'erateurs unitaires. Acta Math. {\bf 111} (1964), 143--211.

\bibitem[W]{W}  Weissman, M. H.:  L-groups and parameters for covering groups.  
In:
L-groups and the Langlands program for covering groups.
Ast\'erisque 2018, no. 398, 33--186.

\end{thebibliography}
\end{document}